\documentclass[10pt,english,reqno]{amsart}

\usepackage[T1]{fontenc}
\usepackage[latin9]{inputenc}

\usepackage[letterpaper, margin=2.6cm]{geometry}
\usepackage{verbatim} 
\usepackage{amsthm}
\usepackage{amssymb}
\usepackage{setspace}
\usepackage{enumerate}
\usepackage{fancyhdr}
\usepackage{xcolor}
\usepackage{cancel}
\usepackage{upgreek} 
\usepackage{mathabx,graphicx}
\usepackage[colorlinks=true]{hyperref}
\hypersetup{linkcolor=red,citecolor=blue,filecolor=dullmagenta,urlcolor=blue}

\numberwithin{equation}{section}
\theoremstyle{plain}
\newtheorem{thm}{Theorem}[section]
\newtheorem{lem}[thm]{Lemma}
\newtheorem{prop}[thm]{Proposition}
\newtheorem{cor}[thm]{Corollary}
\newtheorem{claim}[thm]{Claim}
\newtheorem{rem}{Remark}[thm]

\theoremstyle{definition}

\theoremstyle{remark}

\newcommand{\ca}[1]{{\mathcal{#1}}}
\newcommand{\R}{\mathbb{R}}

\newcommand{\I}{\mathcal{I}(t)}
\newcommand{\LL}{\mathcal{L}}

\newcommand{\J}{\mathcal{J}(t)}

\newcommand{\N}{\mathcal{N}(t)}

\newcommand{\vA}{\varphi_A}
\newcommand{\vB}{\varphi_B}
\newcommand{\zA}{\zeta_A}
\newcommand{\zB}{\zeta_B}

\newcommand{\IOpg}{(1-\gar \partial_x^2)^{-1}}

\newcommand{\Jap}[2]{\left\langle {#1}, {#2} \right\rangle}

\newcommand{\sgn}{\operatorname{sgn}}

\newcommand{\bd}[1]{\boldsymbol{#1}}
\newcommand{\Q}{\boldsymbol{Q}}
\newcommand{\QQr}{\boldsymbol{Q}_{c,\rho}}
\newcommand{\QQrt}{\boldsymbol{Q}_{c(t),\rho(t)}}
\newcommand{\QQc}{\boldsymbol{Q}_{c}}
\newcommand{\Qc}{Q_{c}}
\newcommand{\Qr}{Q_{c,\rho}}

\newcommand{\uu}{\boldsymbol{u}}

\newcommand{\JJ}{\boldsymbol{J}}
\newcommand{\NN}{{\bf{N}}}
\newcommand{\T}{{\bf{T}}_{c,\rho}}
\newcommand{\D}{{\bf{D}}_{c,\rho}}
\newcommand{\G}{{\bf{G}}_{c,\rho}}

\newcommand{\gar}{\varepsilon}
\newcommand{\Opg}{(1-\gar\partial_x^2)^{-1}}
\newcommand{\et}{\eta}
\makeatother

\usepackage{babel}
\addto\shorthandsspanish{\spanishdeactivate{~<>}}

\addto\captionsenglish{}
\addto\captionsenglish{}
\addto\captionsenglish{}
\addto\captionsenglish{}
\addto\captionsenglish{}
\addto\captionsenglish{}
\addto\captionsenglish{}
\addto\captionsspanish{}
\addto\captionsspanish{}
\addto\captionsspanish{}
\addto\captionsspanish{}
\addto\captionsspanish{}
\addto\captionsspanish{}
\addto\captionsspanish{}

\newcommand{\sech}{\operatorname{sech}}

\makeatletter
\@namedef{subjclassname@2020}{\textup{2020} Mathematics Subject Classification}
\makeatother

\begin{document}
	\title[Asymptotics of Good Boussinesq solitary waves]{On asymptotic stability of stable Good Boussinesq solitary waves}
	\author[Christopher Maul\'en]{Christopher Maul\'en}  
	\address{Fakultat f\"ur Mathematik, Universit\"at Bielefeld,  Postfach 10 01 31, 33501 Bielefeld, Germany.}
	\email{cmaulen@math.uni-bielefeld.de}
	\thanks{Ch.Ma. was funded by the Deutsche Forschungsgemeinschaft (DFG, German Research Foundation) -- Project-ID 317210226 -- SFB 1283, and Chilean research grants FONDECYT 1231250.}
	
		\author[Claudio Mu\~noz]{Claudio Mu\~noz}  
	\address{Departamento de Ingenier\'{\i}a Matem\'atica and Centro
de Modelamiento Matem\'atico (UMI 2807 CNRS), Universidad de Chile, Casilla
170 Correo 3, Santiago, Chile.}
	\email{cmunoz@dim.uchile.cl}
	\thanks{Cl.Mu. was partially funded by Chilean research grants FONDECYT 1231250 and Basal CMM FB210005.}



	\begin{abstract}
		We consider the generalized Good-Boussinesq (GB) model in one dimension, with subcritical power nonlinearity $1<p<5$ and data in the energy space $H^1\times L^2$. This model has solitary waves with speeds $c\in (-1,1)$. If $c^2>\frac{p-1}{4}$, Bona and Sachs showed in 1988 the orbital stability of such waves. Previously, one of us showed the existence of an odd-even data submanifold of the energy space where unstable GB standing waves can be perturbed to get the asymptotic stability property if $p\ge 2$. In this paper we prove that for any $2\leq p<5$ stable GB solitary waves are asymptotically stable for initial data placed in the energy space and speeds $|c|>c_+(p)\geq \sqrt{\frac{p-1}{4}}$. For $p$ slightly above 3, one has $c_+(p)= \sqrt{\frac{p-1}{4}}$. The proof involves the introduction of a new set of virial estimates specifically adapted to the GB system in a moving setting. In particular, a new virial estimate with mixed variables is considered to treat arbitrary scaling and shift modulations. Another new ingredient is the understanding the corresponding linear matrix operator under mixed orthogonality conditions, a feature absent in the previous works in Boussinesq models.
	\end{abstract}
	\maketitle

	\section{Introduction}

	\subsection{Setting}
Consider the generalized good Boussinesq \cite{MK} model posed in $\mathbb R_t \times \mathbb R_x$:
		\begin{align}
		  \partial_t^2 \phi+\partial^4_x \phi-\partial^2_x \phi+\partial_{x}^2(f(\phi))=0. \label{eq:GB}
		\end{align}
Here, $\phi(t,x)$ is a real-valued function.  In this paper, we shall consider the case where $f(s)=|s|^{p-1}s$ and $p>1$.  If $u_1:=\phi$ and $u_2:=\partial_x^{-1}\partial_t \phi$, \eqref{eq:GB} can be written as
\begin{align}\label{eq:gGB}
		(g\mbox{GB)}\ \ \ \ 
		\begin{cases}
		\partial_t u_1=\partial_x u_2 \\
		\partial_t u_2=\partial_x (-\partial_x^2 u_1+u_1-f(u_1)).
		\end{cases}
\end{align}
The model \eqref{eq:gGB} has the following associated conserved quantities:	
\begin{equation}\label{eq:energy}
\begin{aligned}
E[u_1,u_2] =  \frac12 \int \left(u_2^2 +u_1^2 +(\partial_x u_1)^2-2F(u_1)\right)\quad  (\mbox{Energy}),\quad P[u_1,u_2] = \int u_1u_2 \quad  &(\mbox{Momentum}) .
\end{aligned}
\end{equation}	
(Here $\int$ means $\int_{\mathbb R} dx$.) From \eqref{eq:energy} these laws define a standard energy space $(u_1,u_2)\in H^1\times L^2$. As well as the Korteweg-de Vries (KdV) equation, \eqref{eq:gGB} in the case $f(s)=s^2$ is considered as a canonical model of shallow water waves \cite{whitman}. In addition, \eqref{eq:gGB} arises in the so-called ``nonlinear string equation'' describing small nonlinear oscillations in an elastic beam \cite{FST}. Finally, \eqref{eq:gGB} is also derived from the Kadomtsev--Petviashvilli equation with negative dispersion (see \cite{pegoKP,Nishita}).

The study of the Boussinesq-type equations \cite{SR,Bou1,BCS1,BCS2} has increased recently, mainly due to the versatility of these models when describing nonlinear phenomena. The Good Boussinesq equation has been the subject of extensive mathematical research, particularly in the areas of solitary and soliton theory, in particular, the long time behavior of solitary waves. This interesting question has attracted the attention of several authors before us, showing that the behavior of solitary waves in the standard energy space $H^{1}\times L^2$ is not an easy problem. It is well-known that solitary waves have the explicit form
	\begin{equation}\label{eq:QQ}
	\begin{aligned}
	\Q_c(t,x)=(Q_c, -cQ_c)(x-ct -x_0), \quad x_0\in\mathbb R, \quad |c|<1.
	\end{aligned}
	\end{equation}
	Here,
	\begin{equation}\label{eq:scaling_Q}
	Q_c(x)=\gamma^{\frac{2}{p-1}}Q(\gamma x),\quad \gamma=\sqrt{1-c^2},\quad 	Q(x)=\left(\frac{p+1}{2\cosh^{2}\left(\frac{p-1}{2}x\right)}\right)^{1/(p-1)}.
	\end{equation}
Bona and Sachs \cite{Bona-Sachs}, applying the theory developed by Grillakis, Shatath and Strauss (see \cite{GSS_stability}), proved that solitary waves are stable if the speed $c$ obeys the condition $(p-1)/4<c^2<1$ and $1<p<5$. The asymptotic stability of these nonlinear waves has been an open problem since 1988. 

Linares \cite{Linares,Notes_linares}, using Strichartz estimates, proved that the Cauchy problem is globally well-posed in the energy space in the case of small data. This result was later improved by Farah \cite{Farah} considering spaces of negative regularity. Kishimoto \cite{kishimoto}, in the case of a quadratic nonlinearity, proved that the Cauchy problem is  globally well-posed in $H^{s}(\R)$,  for $s\geq -1/2$, and ill-posed for $s<-1/2$. Kalantarov and Ladyzhenkaya in \cite{KZ} proved that solutions associated to  initial data with nonpositive energy may blow up in some sense. Inspired by this work, Liu \cite{Liu} showed that there are solutions with initial data arbitrarily near the ground state ($c=0$) that blow up in  finite time. The same year Sachs \cite{Sachs} also showed blow up in the case of a negative energy condition, and Alexander and Sachs showed instability if $|c|<\frac12$ \cite{AS}.  In \cite{MPP} (see also \cite{MPPerratum}), it was proved that small solutions in the energy space must decay to zero as time tends to infinity in proper subsets of space. 

Recently, one of us \cite{Mau} showed the existence of a finite codimensional manifold of initial data with odd-even parity under which the standing wave ($c=0$) is asymptotically stable. In this paper we address the asymptotic stability of Bona-Sachs stable solitary waves with no restriction on the symmetry of the data.

\begin{thm}\label{MT}
Let $2\leq p< 5$. There exists  $\sqrt{\frac{p-1}{4}} \leq  c_+(p)<1$ such that, for all $c_+(p) < |c| <1$, the following is satisfied. Let $x_0\in\mathbb R$ be a fixed shift parameter, let $\Q_c(t,x)$ be as in \eqref{eq:QQ}. There exists $C_0,\delta_0>0$ such that for all $0<\delta<\delta_0$, the following holds. Assume that $(\phi_{1,0},\phi_{2,0})\in H^1\times L^2$ are such that $\| (\phi_{1,0},\phi_{2,0})- \Q_c(0,\cdot)\|_{H^1\times L^2}<\delta$. Then the corresponding solution $(\phi_1,\phi_2)(t)$ to \eqref{eq:gGB} with initial data $(\phi_{1,0},\phi_{2,0})$ at time $t=0$ is globally defined in $H^1\times L^2$ for all $t\geq 0$ and there are $c_+\in (-1,1)$, $|c_+-c|\leq C_0\delta$, and a shift $\rho(t)$, such that for any bounded interval $I$,
\begin{equation}\label{AS}
\lim_{t\to +\infty} \| \phi_1(t,\cdot + \rho(t))- Q_{c_+}\|_{(L^2\cap L^\infty)(I)}=0.
\end{equation}
\end{thm}
This result may be understood as the first result on asymptotic stability of stable solitary waves in shallow water waves models of Boussinesq's type having a ``two wave'' vector-valued description. These include well-known models such as \emph{good, improved}, and \emph{abcd} Boussinesq equations. Unlike KdV type models, which have a preferred direction of movement, two wave dynamics are characterized by two possible directions of propagation, making the dynamics richer and more involved. Standing solitary waves are also allowed. In the GB case, the two wave representation is expressed in two equally important directions of propagation, revealing that previous works on asymptotic stability of generalized KdV type models may not be suitable to get full control on the perturbed solitary wave for large times. Therefore, an important part of the problem is to get the correct framework to study the asymptotic problem. Additionally, almost monotonicity of mass, a property key to get convergence, is  not present in our model. 

An important point to remark is the validity of the parameter $c_+(p)$, which is proved to exist without requiring costly numerical computations and for all $2\leq p<5$. This follows the state of the art in virial identities, where numerical computations to check spectral properties can be reduced to its minimum presence or with no use at all \cite{KMM2017,KM}. Technically, $c_+(p)$ represents the validity of a \emph{spectral condition} for the \emph{dual} or transformed system  to get asymptotic stability. It turns out that one can obtain an explicit formula for this number, depending on some integrals of $Q=Q_{c=1}$ and other functions easily computable. Then, the number $c_+(p)$ can be explicitly computed with high precision for any value of $p$ with any suitable numerical software (see Fig. \ref{positividad}). One can easily see that for $p$ slightly above 3, $c_+(p)=\sqrt{\frac{p-1}{4}}$, however, in the regime $p$ below this number the situation becomes more complicated and one has, with the methods used in this paper, that $c_+(p)>\sqrt{\frac{p-1}{4}}$.  Moreover, we have evaluated other methods involving highly depending numerical tests, and a more complicated panorama is found. Therefore, the proposed method is simpler and strictly not depending on highly demanding numerical computations. A possible solution to this problem would involve to let the scaling modulation evolve with a less demanding orthogonality condition, but to ensure the convergence of the scaling parameter is not obvious at all. Specially interesting here is the case of scaling critical dynamics, where the scaling parameter converges thanks to the use of an additional blow-up modulation mode that has a precise law (but also requiring sometimes particular initial data). We believe that some of these ideas might be useful to treat the more demanding case $2\leq p\leq 3$. On the other hand, the condition $p\geq2$ in Theorem \ref{MT} already appears in \cite{Mau}, as a consequence of the lack of sufficient smoothness on the nonlinearity to perform stable Taylor's expansions if $p<2$.

The reader may ask on the second component of the solution in \eqref{AS}. The situation is more involved than for the first component, due to certain loss of derivatives present in the model. We will obtain that for some $\varepsilon>0$ small but fixed,
\begin{equation}\label{AS2}
\lim_{t\to +\infty} \| \IOpg ( \phi_2(t,\cdot + \rho(t)) + c_+Q_{c_+})\|_{H^s(I)} =0, \quad 0\leq s<2.
\end{equation}
therefore, there is convergence to zero of the difference $\phi_2(t,\cdot + \rho(t)) + c_+Q_{c_+}$, but in an space slightly more regular than $L^2$. Precisely, the value $s=2$ represents the $L^2$ norm of the solution, which as far as we understand, cannot be controlled in terms of the local integrated dynamics without assuming additional regularity or decay properties. It is also interesting to notice that although $\IOpg$ represents a smoothing operator, it does not affect the actual dynamics, in the sense that no convergence below the natural energy-space regularity $L^2$ is obtained.

Another key ingredient in GB is its vector-valued character, consequence of its fourth order character. This vector-valued setting makes GB harder but still tractable enough to be considered as the first step towards the understanding of \emph{improved and abcd} solitary waves. In the first case, solitary waves are superluminal, making the analysis harder because energy and momentum seem not enough to capture shifts and scaling modulations, and in abcd, one of the most prominent models studied today, the parameters defining the model are as important as the initial data to obtain a suitable representation of the long time dynamics, making the choice of the particular abcd regime very important. Among these particular abcd regimes, we identify KdV-KdV (probably the easiest regime) and BBM-BBM regimes, and the KBK integrable case. In any case, in addition to clarify the asymptotic stability of stable GB solitary waves, Theorem \ref{MT} is a first but fundamental step in the direction of understanding all Boussinesq's solitary wave dynamics. This program of research is supported by previous works \cite{MPP,Mau,MaMu} characterized by a simplified dynamics, and carried out to the most general setting (allowing scaling and shift modulations).

As previously expressed, Theorem \ref{MT} complements the orbital stability proof obtained by Bona and Sachs \cite{Bona-Sachs} in 1988. This nearly 40 years gap between orbital and asymptotic stability is heavily influenced by the lack of correct functional setting to treat two-wave problems, and very importantly, the absence of transformations of Darboux type. Important advances were obtained by Pego and Weinstein in an influential work \cite{PW2}, where linear asymptotic stability in Boussinesq type models was studied. Precisely, the program described above is strongly based on this influential work. We believe that the ideas in this work are suitable to treat each model described there, after reasonable improvements are done. In the meantime, Pego and Weinstein proved asymptotic stability of gKdV solitary waves \cite{PW1}, later improved by the works by Martel and Merle \cite{Martel-Merle1,Martel-Merle2,MM_solitonsKdV}. In the case of Improved Boussinesq, in addition to the convective stability proved by Pego and Weinstein, one can find the recent work \cite{MaMu0} where decay on compact spaces is proved.

There are interesting open results that can be proved in the case of GB solitary waves. A first one is the description of the full regime $c\in(-1,1)$. Since $c^2 <\frac{p-1}{4}$ is a regime of instability, we believe that merging ideas of this work together with \cite{Mau} will be necessary to construct an asymptotically stable manifold of initial data for GB in the moving case. Another interesting direction is dictated by the critical case $c^2 =\frac{p-1}{4}$, where a different behavior can be expected. Another, more involved problem, is to describe the asymptotic stability of abcd solitary waves, without using the particular character of the parameters of the model. Naturally one may think that KdV-KdV systems may be favored by classical Martel-Merle's techniques, but devising a method able to consider the full stable regime of space parameters in abcd remains an important open question.

\subsection{Main difficulties} The proof of Theorem \ref{MT} is mainly based on the previously published works \cite{KMMV20,Mau,MaMu} whose main ingredient is the use of combined virial estimates to obtain the convergence of perturbations of a GB soliton at large times. However, to get \eqref{AS} we required the introduction of several new ingredients not present in previous works. Some of them are mentioned now:

\medskip

\noindent
{\it A suitable setting for orbital stability}. It turns out that the classical Grillakis-Shatah-Strauss/Bona-Souganidis-Strauss setting described in \cite{Bona-Sachs} to obtain the orbital stability of GB solitary waves is unsatisfactory when dealing with GB long time convergence properties. A first issue is the vectorial character of the problem. Indeed, a natural instability direction associated to scalings in GB is the first eigenfunction associated to the operator $L$ in the matrix operator
\[
\bd{L}=\begin{pmatrix} L &~{}&c \\ c &~{}& 1\end{pmatrix}, \quad L=-\partial_y^2+1- f'(Q_{c}).
\]
An important advantage of this choice of eigenfunction is that the spectral analysis can be split into two disjoint pieces, passing from a matrix-valued operator to a scalar valued one. It turns out that in our case this direction needs to be replaced by the more suitable vector-valued translation and scaling directions
\[
\Jap{\Q_c'}{\bd{v}}=\Jap{\bd{J}\QQc}{\bd{v}}=0, \quad \bd{J}\QQc =  \begin{pmatrix} -cQ_c  \\ Q_c \end{pmatrix}. 
\]
In particular, the second orthogonality condition is directly related to the obtention of a quadratic estimate on the scaling parameter $|c'|$. Although this is in principle a good property, it will imply that the spectral properties and required coercivities will have to be proved for vector-valued operators with mixed orthogonality conditions.

\medskip

\noindent
{\it Mixed variables}. It turns out that the GB system has a natural mixed variable $u_2+cu_1$, where $u_1,u_2$ are perturbations and $c$ is the speed of the solitary wave (notice that this quantity becomes zero in the case of the solitary wave). This is a feature deeply related to the use of the energy space only, where there is no preferred direction of movement. Although this phenomenon had already been treated in \cite{KMMV20}, two important new consequences appear because of this twisted behavior: first, the previously mentioned orthogonality conditions become just ``linear'' conditions in terms of this new variable, meaning that their ``orthogonal'' behavior is lost; and second, every virial estimate must be necessarily placed in these new variables. Naturally, good primal estimates are completely out of reach if one only has just  linear conditions, but the case of dual virial estimates can be saved by a new almost orthogonality condition that appears for the transformed system, precisely as mentioned below. 

\medskip

\noindent
{\it Transformed system with a related vector-valued coercivity estimate}. It is usually common to find that Darboux transformations of linear dynamics lead to the study of simpler evolution systems. This is the case of scalar fields \cite{KMMV20}. In the GB setting the transformed problem conserves essentially all the bad properties of the original one except by one: perturbative terms issued from shifts are now not present. As a consequence of the previous steps, the classical transformed problem (Section \ref{Sec:4}) leads to the study of a vector valued virial bilinear form 
\[
\int \left( 3\eta _{2,y}^2 +(1-c^2)\eta_2^2 -f'(Q_c)\eta_2^2 +(\eta_1+c\eta_2)^2 \right) ,
\]
under (roughly speaking) almost orthogonality conditions ($\varepsilon$ is a small parameter)
\[
\begin{aligned}
	&\left| \langle {\bf Q}_c' , \bd{\eta} \rangle \right|+ \left| \langle   \Lambda \QQc , \bd{\eta} \rangle \right| \lesssim \varepsilon  \|\bd{\eta}\|_{L^2\times H^1}. 
	\end{aligned}
\]
The second orthogonality condition is new to the system and it is a consequence of the chosen transformed problem. Since the bilinear operator above has mixed variables, and the orthogonalities consider both coordinates at similar sizes, an alternative strategy to obtain a coercivity estimate is needed. In this case, we first relate the problem to a well-known one appearing in the orbital stability proof, to then generalize the corresponding coercivity estimate by comparing the obtained result with the more involved case above described. This is the only part in the proof where we require speeds larger than the expected threshold $\sqrt{(p-1)/4}$ ($p$ below a number above but very close to 3), since to ensure a suitable coercivity estimate requires sufficiently large speeds. The existence of a more involved dynamics for smaller speeds is not discarded at all, mainly because the lower the nonlinearity the more involved the dynamics (possibly internal modes or similar objects), and since other numerically motivated methods tested by us before publishing this paper reveal a more complicated setting than the expected one. In particular, the classical method developed by Weinstein to get coercivity estimates (numerically computing the inner product $(\mathcal L^{-1} \cdot , \cdot )$ and check the good negative sign) seems to give less conclusive evidence than the one present here.

\medskip

\noindent
{\it Inverted regularities. } The transformation performed into the linearized GB system inverts regularities, leading to strange results where good estimates are obtained for the second coordinate of the perturbation, but bad estimates are obtained for the first one. This is somehow standard in certain virial estimates, but the GB case is more involved and a new virial estimate (for the functional $\mathcal N$) is needed. Compared with the previous works \cite{Mau,MaMu}, in this case we strongly need two important facts: a miraculous cancellation of shifts error terms in the transformed problem, and the quadratic behavior of the scaling parameter, that can be considered as another almost cancellation of bad directions. This was not needed in the previous versions \cite{Mau,MaMu}, either because there were no shifts nor scalings, or scalings were not considered.

\subsection{Organization of this work} This paper is organized as follows. In Section \ref{Sec:2} we recall several important facts for the considered model and the key modulation estimates. Later, Section \ref{Sec:3} we perform first virial estimates. Section \ref{Sec:4} is devoted to the second virial estimates. Section \ref{Sec:5} deals with a third set of virial identities, and Section \ref{Sec:6} will contain the main coercivity estimates necessary for the final proof. Finally, in Section \ref{Sec:7} we prove Theorem \ref{MT}.

\subsection*{Acknowledgments} Ch. Ma. would like to thank the CMM and DIM at University of Chile, for their support and hospitality during research stays while this work was written. Cl. Mu. would like to thank the Erwin Schr\"odinger Institute ESI (Vienna) and INRIA Lille France, where part of this work was written.

\section{Spectral analysis, modulation and orbital stability}\label{Sec:2}

In this section we recall some standard facts of the Boussinesq model, including its solitary waves. Several facts are taken from the previous work \cite{Mau}, see this reference for full details and proofs.

\subsection{Rewriting of the problem} Let $c=c(t) \in (-1,1)$ and $\rho=\rho(t) \in \mathbb R$  be a scaling and shift parameters to be defined later. We define
\begin{equation}\label{v_u}
\phi_j(t,x):= \tilde \phi_j(t,y), \quad  j=1,2,\quad y := x-\rho(t).
\end{equation}
Then \eqref{eq:gGB} reads now in centered variables $(\tilde \phi_1, \tilde \phi_2)$
\begin{align}\label{eq:gGB_new}
\begin{cases}
\partial_t \tilde \phi_1=\partial_y (\tilde \phi_2 +c \tilde \phi_1) +(\dot\rho-c) \partial_y \tilde \phi_1 \\
\partial_t \tilde \phi_2=c \partial_y  (\tilde \phi_2 +c \tilde \phi_1)+ \partial_y (-\partial_y^2 \tilde \phi_1+ (1-c^2) \tilde \phi_1-f(\tilde \phi_1)) +(\dot\rho-c) \partial_y \tilde \phi_2.
\end{cases}
\end{align}
We will drop the tildes if no confusion is present. \eqref{eq:gGB_new} will be the model to study in this paper. Recall $Q$ and $Q_c$ in \eqref{eq:scaling_Q} and $\Q_c$ in \eqref{eq:QQ}. Notice first that $Q$ and $Q_c$ satisfy the equations
	\begin{equation}\label{eq:soliton_eq}
	Q'' -Q +f(Q)=0, \quad \hbox{and} \quad Q_c''-(1-c^2)Q_c+f(Q_c)=0,
	\end{equation}
respectively. Consider the modulated solitary wave 
	\begin{equation}\label{Mod_soliton}
	\begin{aligned}
	\Q_c :=  \begin{pmatrix} Q_c  \\ -cQ_c \end{pmatrix}(y), \qquad {\bf \phi} = \Q_c + \bd{v}.
	\end{aligned}
	\end{equation}
One can see that the perturbation $\bd{v} =(v_1,v_2)^T$ satisfies 
		\begin{align} \label{eq:eq_lin}
		\begin{cases}
		\partial_t v_1  = \partial_y (v_2+c v_1) +(\rho'-c)\partial_y (\Qc+v_1)-c'~\Lambda (\Qc)\\
		\partial_t v_2 = \partial_y \LL v_1 +c\partial_y (v_2+cv_1)+ N+(\rho'-c)\partial_y(-c\Qc+v_2)+c'~\Lambda (c\Qc), 
		\end{cases}
		\end{align}
	where $\Lambda f_c := \partial_c f_c$,
	\begin{equation}\label{eq:L_N}
	\begin{aligned}
	\LL:=&~{}-\partial_y^2+1-c^2-f'(\Qc), 
	\quad 
	N=N(v_1): =  -\partial_y(f (\Qc+v_1) - f (\Qc)- f'(\Qc)v_1).
	\end{aligned}
	\end{equation}
Additionally, from \eqref{Mod_soliton} consider $\Lambda \Q_c:= \partial_c \Q_c$,
 \begin{equation}\label{eq:TDGF}
 \begin{aligned}
 {\bf T}_c = \Q_c', \qquad  {\bf D}_c =\Lambda \Q_c,  \qquad {\bf G}_c=\JJ\partial_y^{-1}{\bf T}_c =\JJ\Q_c.
 \end{aligned}
 \end{equation}
Let also $L=-\partial_y^2+1- f'(Q_{c})$,
	\begin{equation}\label{eq:LL}
	\JJ= \begin{pmatrix}
	0&1\\ 1&0
	\end{pmatrix}, \quad \bd{L}=\begin{pmatrix} L &~{}&c \\ c &~{}& 1\end{pmatrix}, \quad  \NN(\bd{v})=\begin{pmatrix}0 \\ N(v_1)
\end{pmatrix}.
\end{equation}
Then using \eqref{eq:TDGF} we rewrite the system as follows:
	\begin{align}\label{eq:u_TD}
	\partial_t \bd{v}= \partial_y \JJ \bd{L} \bd{v} +(\rho'-c) ({\bf T}_c +\partial_y \bd{v})-c' {\bf D}_c +\NN (v_1).
	\end{align}
In order to fix the scaling and shift parameters, later we will choose the natural orthogonality conditions 
\[
\Jap{\bd{T}_c}{\bd{v}}=\Jap{\bd{J}\QQc}{\bd{v}}=0.
\]
Since $|c|>0$ by hypothesis, rewriting in terms of the variables $(cv_1+v_2,v_2)$, we get
\begin{equation}\label{eq:orto}
\begin{aligned}
& 0=\Jap{\bd{T}_c}{\bd{v}} =\int \Qc'(v_1- cv_2) = (1+c^2)\int Q_c'v_1 - \int Q_c'(v_2+cv_1) ,\\
& 0=\Jap{\bd{J}\QQc}{\bd{v}} =\int \Qc(v_2-cv_1) = -2c\int Q_cv_1  + \int Q_c(v_2+cv_1) .
\end{aligned}
\end{equation}
	\begin{lem}\label{lem2p2}
	Let $\bd{L}$ be as in \eqref{eq:LL}. Then it holds
	\begin{enumerate}
	\item[$(i)$]  $\bd{L}$ is a self-adjoint operator in $(L^2)^2$ with dense domain $H^2\times L^2.$
	\item[$(ii)$] $\bd{L}$ has a unique negative eigenvalue $-\nu_0<0$, associated to an eigenfunction $\bd{\Psi}_{-}$ such that $\|\bd{\Psi}_{-}\|_{L^2}=1$.
	\item[$(iii)$] $ \hbox{ker}(\bd{L})=\hbox{span}\{\bd{Q}_c'\}$.
	\item[$(iv)$]  $\mathcal L \Lambda \Qc  = 2c \Qc$ and $\bd{L}[{\bf \Lambda \Qc}]=-\JJ {\bf\Qc}$. 
	\item[$(v)$] Under \eqref{eq:orto},  there exists $c_0>0$ such that 
	\begin{equation}\label{coer0}
	\langle \bd{L} \bd{v} ,\bd{v} \rangle \geq c_0  \| \bd{v}\|_{H^1\times L^2}^2.
	\end{equation}
	\end{enumerate}
	\end{lem}
	\begin{proof}
	Proof of $(i)$: this is clear from \cite{Bona-Sachs}. 
	
	Proof of  $(ii)$ and $(iii)$. Recall that $c\neq 0$. The eigenvalue problem for $\bd{L}$ can be written as 
	\[
	L v_1 + cv_2 = \lambda v_1, \quad cv_1 + v_2 = \lambda v_2.
	\]
	The case $\lambda=1$ is simple: one has $v_1=0$ and therefore $v_2=0$, then this case is discarded. Now, if $\lambda \neq 1$, $v_2=\frac{cv_1}{\lambda-1}$. Replacing in the first equation we arrive to $\mathcal L v_1 = \left( \lambda -c^2 -\frac{c^2}{\lambda-1}\right) v_1$. Performing the change of variables $v_1(x)=\tilde v_1(y)$, $y=\gamma x$, one arrives to $L_0 \tilde v_1 = \mu \tilde v_1$, where $\mu:= \gamma^{-2}\left( \lambda -c^2 -\frac{c^2}{\lambda-1}\right)= \frac{\lambda(\lambda-1-c^2)}{(\lambda-1)\gamma^2}$ and $L_0:= -\partial_y^2 +1 -f'(Q)$. Now we invoke \cite[Theorem 3.1]{Chang} which characterizes the eigenvalues and eigenvectors of $L_0$: let $p_1=+\infty$, $p_m=\frac{m+1}{m-1}$, $m>1$. Define $\lambda_m:= 1-k_m^2$, $k_m= \frac{m-1}{2}(p_m-p)$. Then, if $p\in[p_{j+1},p_j)$ for some $j$, $L_0$ has eigenvalues $\lambda_m$, $m=0,1,\ldots,j$. Then, one has the equation $\lambda(\lambda-1-c^2) =(\lambda-1)\gamma^2 \lambda_m$, for some valid $m$. If $\theta:= \lambda -1$, one obtains $\theta^2 +\gamma^2 k_m^2 \theta -c^2 =0$. This returns 
	\[
	\theta_\pm = \frac12 \left( -\gamma^2 k_m^2 \pm \sqrt{\gamma^4 k_m^4 +4c^2} \right).
	\] 
Now we proceed by cases. In the case $\lambda=0$ ($\theta=-1$) one easily has $\lambda_m=0$ corresponding to the kernel of $L_0$. One completes $(iii)$ noticing that 
	\[
	\begin{aligned}
	\bd{L}\bd{Q}_c'
	= \begin{pmatrix}L\Qc' -c^2\Qc' \\0\end{pmatrix} 
	= \begin{pmatrix}0 \\0\end{pmatrix}.
	\end{aligned}
	\]
	Negative eigenvalues are obtained when $\theta_-<-1$. In this case, assume that $k_m^2=1$. Then $\theta_- =-1$. Consequently, $\theta_-<-1$ only in the cases where $k_m>1$. Since $k_0=\frac{p+1}{2}>1$, and $k_1=1$, the only possibility is to get $k_0$. In that case  $\theta_-<-1$, leading to the unique negative eigenvalue for $\bd{L}$, which is 
	\[
	\lambda^0 = \theta_-+1 = 1 - \frac12 \gamma^2 \left( \left( \frac{p+1}{2}\right)^2 + \sqrt{ \left( \frac{p+1}{2}\right)^4 + \frac{4c^2}{ \gamma^4}}\right).
	\]
	The associated eigenfunction is easily constructed from the eigenfunction of $L_0$ associated to $\lambda_0=1-\frac{(p+1)^2}{4}$, which is $Q^{(p+1)/2}$. One gets
	\begin{equation}\label{Psi_-}
	\bd{\Psi}_{-}(x) = c_0 Q_c^{\frac{p+1}2}(x) \begin{pmatrix} 1 \\ \frac{c}{\lambda^0-1}\end{pmatrix}.
	\end{equation}
	where $c_0>0$ is chosen such that $\|\bd{\Psi}_{-}\|_{L^2}=1$. This proves $(ii)$.
	
	Now we show $(iv)$. By \eqref{eq:soliton_eq}, we have $\Lambda\Qc''-(1-c^2) \Lambda \Qc +2c\Qc +f'(\Qc)\Lambda \Qc=0$. Therefore, $\mathcal L \Lambda \Qc  = 2c \Qc.$ Also,
	\[\begin{aligned}
	\bd{L}[\bd{\Lambda} \QQc]
	= \begin{pmatrix}L(\Lambda \Qc) -c\Lambda (c\Qc) \\c\Lambda \Qc-\Lambda (c \Qc) \end{pmatrix} 
	= \begin{pmatrix}L(\Lambda \Qc) -c\Lambda (c\Qc) \\ - \Qc \end{pmatrix} = \begin{pmatrix} L  (\Lambda \Qc) -c^2\Lambda \Qc-c\Qc \\ - \Qc \end{pmatrix} = \begin{pmatrix}\LL (\Lambda \Qc) -c\Qc \\ - \Qc \end{pmatrix} .
	\end{aligned}\]
	We conclude that
	\[\begin{aligned}
	\bd{L}[\bd{\Lambda} \QQc]
	=&~{} \begin{pmatrix} c\Qc \\ - \Qc \end{pmatrix} 
	=-\JJ \QQc.
	\end{aligned}
	\]
	This finally shows $(iv)$. 
	
	Now we prove $(v)$. Now we follow \cite[Lemma E.1]{Wei}. Since $\bd{L}$ is clearly self-adjoint, has a unique negative eigenvalue, $\bd{\Lambda Q}_c$ is even, and $\bd{L} \bd{\Lambda Q}_c =-\bd{JQ}_c$,  one has from \eqref{calculo_base} that $\langle \bd{L} \bd{\Lambda Q}_c, \bd{\Lambda Q}_c\rangle = - \langle \bd{JQ}_c, \bd{\Lambda Q}_c\rangle = - \frac{\gamma^{\frac{7-3p}{p-1}}}{p-1} \left(p-1-4c^2 \right)  \|Q\|_{L^2}^2 <0$. Additionally, from \eqref{Psi_-}
\[
\langle \bd{\Psi}_-, \bd{\Lambda Q}_c \rangle = \frac1{\lambda^0} \langle \bd{\Psi}_-, \bd{JQ}_c \rangle = \hbox{cst} \left( 1 - \frac{1}{\lambda^0+1}\right)c \int Q_c^{\frac{p+3}2}\neq 0.
\]
In conclusion, under $0=\Jap{\bd{J}\QQc}{\bd{v}}$ one has $\langle \bd{L} \bd{v} ,\bd{v} \rangle \geq	 0$.

Now we prove \eqref{coer0}. Arguing by contradiction as in \cite{CM}, there exists a sequence $(\bd{v}_n)_{n\in\mathbb N}$ of unit $H^1\times L^2$ norm and such that $0\leq  \langle \bd{L} \bd{v}_n ,\bd{v}_n \rangle <\frac1n$. Therefore, $\int v_{n,1} \mathcal L v_{n,1} \to 0$ and $ \int (v_{n,2} +cv_{n,1})^2\to 0$. Passing to a subsequence if necessary, $\bd{v}_n \to \bd{\bar v}$ in $H^1\times L^2$ weak and locally on compact sets. By weak convergence, one has $\langle \mathcal L \bar v_1,\bar v_1\rangle =\int ( \bar v_{1,x}^2 + (1-c^2) \bar v_1^2 -f'(Q_c) \bar v_1^2 ) \leq 0$. Additionally,  \eqref{eq:orto}  passes to the limit and $\langle \bd{T}_c,\bd{\bar v}\rangle=\langle \bd{JQ}_c,\bd{\bar v}\rangle=0$. Finally, $v_{2,n} +cv_{1,n}$ weakly $L^2$ converges to $\bar v_2+ c\bar v_1$. Since  the $L^2$ norm $ \int (v_{n,2} +cv_{n,1})^2$ also converges to 0, we get $ \int (\bar v_{2} +c\bar v_{1})^2 \leq \liminf_{n\to +\infty}  \int (v_{n,2} +cv_{n,1})^2 =0.$ Therefore $\bar v_2 = -c\bar v_1$ a.e. 

Therefore, \eqref{eq:orto} now reads $\int Q_c' \bar v_1 =\int Q_c \bar v_1 =0$. Under these orthogonalities, and the conditions $1<p< 5$, $\frac{p-1}{4} < c^2 <1$, there exists $c_0>0$ such that $ \int v_1 \mathcal L \bar v_1 \geq c_0 \| \bar v_1\|_{H^1}^2 \sim  \| \bd{\bar v}\|_{H^1\times L^2}^2$. Therefore $\bd{\bar v} \equiv 0$, a contradiction. This proves \eqref{coer0}.
\end{proof}

Now we consider the self-adjoint operator $\bd{\tilde L}$ given by
\begin{equation}\label{eq:tildeL}
 \bd{\tilde L}=\begin{pmatrix} \tilde L &~{}&c \\ c &~{}& 1 \end{pmatrix}, 
 \end{equation}
where $ \tilde L:= -3\partial_{y}^2  +1 -f'(Q_c)$. Let $\mathcal{\tilde L}:=-3\partial_{y}^2  +(1-c^2) -f'(Q_c) $.
	
	\begin{lem}\label{lem2p3}
	Let $\bd{\tilde L}$ be as in \eqref{eq:tildeL}. Then it holds
	\begin{enumerate}
	\item[$(i)$]  $\bd{\tilde L}$ is a self-adjoint operator in $(L^2)^2$ with dense domain $H^2\times L^2.$
	\item[$(ii)$]  For any $p\in(1,5)$,  $\mathcal{\tilde L}$ has an explicit unique negative eigenvalue associated to an eigenfunction $Q^\alpha$, $\alpha=\alpha(p)>0$ given. Moreover, there exists  $c_0>0$ such that in the orthogonal to $Q^\alpha$ one has
	\[
	\int \eta_2 \mathcal{\tilde L} \eta_2 \geq c_0 \gamma \|\eta_2\|_{H^1}^2.
	\] 
	\item[$(iii)$] $\bd{\tilde L}$ has a unique negative eigenvalue. 
	\item[$(iv)$] For each $1<p<5$, there exists $c_+(p) > 0$ such that, if $|c|>c_+$, the following holds: there exists $c_0(c,p)>0$ such that, under $\langle \JJ\bd{\Lambda Q}_c ,\bd{\eta} \rangle=0$,  
	\begin{equation}\label{coer0_new}
	\langle \bd{\tilde L} \bd{\eta} ,\bd{\eta} \rangle \geq c_0  \gamma \| \bd{\eta}\|_{H^1\times L^2}^2.
	\end{equation}
	\end{enumerate}
	\end{lem}
	\begin{proof}
	Item $(i)$ is direct. 
	
	Let us prove $(ii)$. We must study the equation $\mathcal{\tilde L} \eta_2 = -\lambda \eta_2$. It is enough to consider $\mathcal{\tilde L}_0:= \mathcal{\tilde L}_{c=0}$ and study $\mathcal{\tilde L}_0\tilde\eta_2 =  - \gamma^{-2}\lambda \tilde \eta_2$. The general case is obtained by considering $\eta_2(y)=\tilde \eta_2(\gamma y)$, $\gamma=\sqrt{1-c^2}$ and $Q_c(y)=\gamma^{\frac{2}{p-1}}Q(\gamma y)$. The operator $\mathcal{\tilde L}_0$ has a unique negative eigenvalue $\lambda= -3\alpha^2+1<0$,  $\alpha=\frac{p+1}{12} \left( -3 +\frac{6}{p+1} +\left( \frac{9+6p+33p^2}{(1+p)^2}\right)^{1/2}\right)$, with associated eigenfunction $Q^\alpha$ (see Fig. \ref{Fig1} left). This is checked as follows: one has
\[
\begin{aligned}
(Q^m)'' = &~{} mQ^{m-1}Q''+m(m-1)Q^{m-2}Q'^2 = m^2 Q^{m} - m \left( 1+ \frac2{p+1}(m-1)\right)Q^{m+p-1}; \\
(Q^n Q' )''  = &~{}  n^2 Q^{n}Q' - n \left( 1+ \frac2{p+1}(n-1)\right)Q^{n+p-1} Q' +2nQ^{n-1}Q'(Q-Q^p) + Q^n (Q' -pQ^{p-1}Q')\\
=&~{} (n+1)^2 Q^{n}Q'  - \left( p+ 3n + \frac2{p+1}n(n-1) \right) Q^{n+p-1} Q' .
\end{aligned}
\]
\begin{figure}[htbp]
   \centering
   \includegraphics[width=0.4\textwidth]{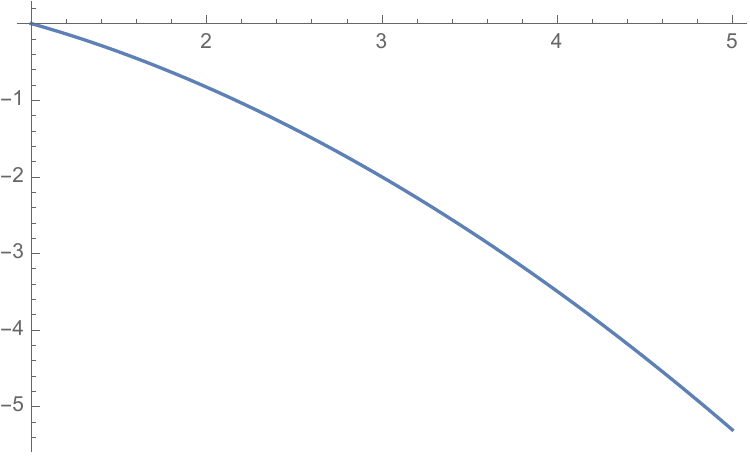} 
    \includegraphics[width=0.42\textwidth]{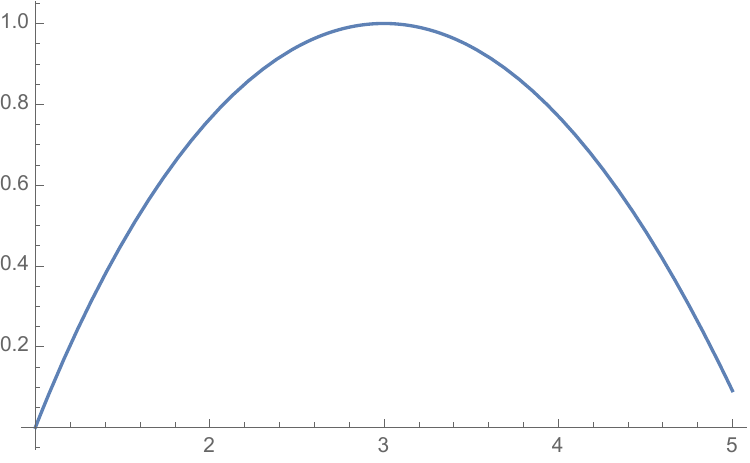} 
   \caption{Left. Negative eigenvalue of $\tilde L_0$ as a function of $p\in (1,5)$. Notice that at $p=1$ the corresponding value is zero. Right: Second eigenvalue of $\tilde L_0$, which is zero at $p=1$.}
   \label{Fig1}
\end{figure}
Therefore
\[
\begin{aligned}
\mathcal{\tilde L}_0 Q^m =&~{} -3(Q^m)'' +Q^m -pQ^{p+m-1} =(1 -3 m^2) Q^{m}  + \left(  3 m + \frac6{p+1}m (m-1) -p\right)Q^{m+p-1}; \\
\mathcal{\tilde L}_0 (Q^n Q' ) = &~{}  \left(1-3 (n+1)^2 \right)Q^{n}Q'  +3 \left( \frac23p+ 3n + \frac2{p+1}n(n-1) \right) Q^{n+p-1} Q' .
\end{aligned}
\]
The positive solution $m$ to $ 3 m + \frac6{p+1}m (m-1) -p =0$ leads to $\alpha$ previously defined, and the eigenvalue $\lambda= -3\alpha^2+1<0$. Notice that $Q^\alpha$ is positive. Also, the positive solution $n$ to  $ \frac23p+ 3n + \frac2{p+1}n(n-1) =0$ leads to a coefficient $\beta=\frac{p+1}{4} \left( -3 +\frac{2}{p+1} +\left( \frac{3+2p+11p^2}{3(1+p)^2}\right)^{1/2}\right)$ such that $Q^\beta Q'$ is the second eigenfunction (with only one zero) associated to the eigenvalue $1-3 (\beta+1)^2$. This eigenvalue is always positive in the range $1<p<5$, as it is easily checked (see Fig. \ref{Fig1} right).  Notice that
\[
\begin{aligned}
 \mathcal B[\eta_2] := &~{}\int \eta_2 \mathcal{\tilde L} \eta_2 = \gamma  \int \tilde \eta_2 \mathcal{\tilde L}_0 \tilde \eta_2 
= : \gamma \mathcal B_0[\tilde \eta_2].
 \end{aligned}
\]
It is enough to check the positivity of $ \mathcal B_0[\tilde \eta_2]$.  On the orthogonal to $Q^\alpha$, we know that there exists $c_0=c_0(p)>0$ such that $ \mathcal B_0[\tilde \eta_2] \geq c_0 \|\tilde \eta_2\|_{H^1}^2$. 
	
	Proof of $(iii)$. Notice that the action of $\bd{\tilde L}$ reads as follows:
	\[
	\langle  \bd{\tilde L} \bd{\eta},\bd{\eta} \rangle = \int \left( 3\eta _{1,y}^2 +(1-c^2)\eta_1^2 -f'(Q_c)\eta_1^2 +(\eta_2+c\eta_1)^2 \right).
	\]
	 Thanks to this fact we clearly see that $ \bd{\tilde L} \geq \bd{L}$, and therefore we have that $\bd{\tilde L}$ is an operator with at most one negative eigenvalue. Moreover, choosing $\eta_2$ such that $\eta_2+c\eta_1 =0$ and $\eta_1$ being an eigenfunction of $\mathcal{\tilde L}= -3\partial_y^2 +(1-c^2) -f'(Q_c)$ associated with the unique negative eigenvalue, we see that $ \bd{\tilde L}$ has also a unique negative eigenvalue.
	
Now we prove $(iv)$. First of all, notice that 
\[
\begin{aligned}
\Lambda Q_c = &~{} -\frac{c}{\gamma^2} \gamma^{\frac{2}{p-1}}  \left( \frac{2}{p-1} Q + yQ' \right) (\gamma y) =: -\frac{c}{\gamma^2}\gamma^{\frac{2}{p-1}} \Lambda_0Q(\gamma y),\\
2c\Lambda Q_c +Q_c = &~{} \gamma^{\frac{2}{p-1}} \left(  -\frac{2c^2}{\gamma^2} \Lambda_0Q + Q\right)(\gamma y) = : \gamma^{\frac{2}{p-1}} \left(    Q-m_c \Lambda_0Q\right)(\gamma y),\quad m_c:=\frac{2c^2}{\gamma^2}.
 \end{aligned}
\]
Recall that $c\neq 0$. 
Define $\eta(x)=\tilde \eta(y)$, with $x=\gamma  y$. We have $ 0=\langle \JJ \Lambda \QQc , \bd{\eta} \rangle  = \int (\Lambda Q_c (\eta_2+c\eta_1 )  -(2c\Lambda Q_c + Q_c)\eta_1)$. Thus,
\[
\begin{aligned}
0= &~{} \int (\Lambda Q_c (\eta_2+c\eta_1 )  -(2c\Lambda Q_c + Q_c)\eta_1)= \gamma^{\frac{2}{p-1}-1} \int \left( \frac{c}{\gamma^2} \Lambda_0Q (\tilde \eta_2 +c\tilde \eta_1) + \tilde\eta_1 (Q-m_c\Lambda_0 Q)\right).
\end{aligned}
\]
Similarly,
\[
 \int (3\eta_{1,y}^2 +(1-c^2)\eta_1^2 -f'(Q_c)\eta_1^2 +(\eta_2+c\eta_1)^2)   = \gamma  \int (3\tilde \eta_{1,z}^2 +\tilde \eta_1^2 -f'(Q)\tilde \eta_1^2 ) +\frac1{\gamma} \int (\tilde \eta_2+c\tilde \eta_1)^2.
\]
Then, by Cauchy-Schwarz, and using that $\tilde L_0:= \tilde L_{c=0}$,
\begin{equation}\label{chequeo}
\begin{aligned}
& \int (3\eta_{1,y}^2 +(1-c^2)\eta_1^2 -f'(Q_c)\eta_1^2 +(\eta_2+c\eta_1)^2)   \geq \gamma \int \tilde \eta_1 \tilde L_0 \tilde \eta_1  + \frac1{\gamma}\left( \int (\Lambda_0 Q)^2\right)^{-1} \left(\int \Lambda_0 Q (\tilde \eta_2 +c\tilde \eta_1) \right)^2 \\
& \quad  = \gamma \int \tilde \eta_1 \tilde L_0 \tilde \eta_1 
+ \frac{\gamma^3}{c^2} \left(  \int (\Lambda_0 Q )^2 \right)^{-1} \left( \int (-m_c \Lambda_0 Q + Q)\tilde \eta_1 \right)^2.  
\end{aligned}
\end{equation}
An additional simplification gives
\[
\begin{aligned}
&\eqref{chequeo} \\
&~{} = \gamma \left(\int (\Lambda_0 Q )^2\right)^{-1} \left( \int (\Lambda_0 Q )^2 \int \tilde \eta_1 \tilde L_0 \tilde \eta_1  +\frac{\gamma^2}{c^2} \left( \int \tilde \eta_1 Q\right)^2 -4 \int \tilde\eta_1 Q \int\tilde \eta_1 \Lambda_0Q +2m_c \left(\int\tilde\eta_1 \Lambda_0Q \right)^2 \right).
\end{aligned}
\]
It is enough to check the positivity of this last term.  From $(ii)$, the operator $\tilde L_0$ has a unique negative eigenvalue $\lambda= -3\alpha^2+1<0$,  $\alpha=\frac{p+1}{12} \left( -3 +\frac{6}{p+1} +\left( \frac{9+6p+33p^2}{(1+p)^2}\right)^{1/2}\right)$, with associated eigenfunction $Q^\alpha$. Orthogonal to $Q^\alpha$, we know that there exists $c_0=c_0(p)>0$ such that \eqref{chequeo} satisfies $\geq c_0 \|\tilde \eta_1\|_{L^2}^2$. Therefore we need to check the positivity of \eqref{chequeo} in the case $\tilde \eta_1 =Q^\alpha.$ In this situation, notice that
\[
\int Q^\alpha \Lambda_0Q = \left( \frac{1}{p-1} -\frac1{2(\alpha+1)}\right)\int Q^{\alpha+1}.
\]
Therefore,
\[
\begin{aligned}
&\eqref{chequeo} \\
 &~{}= \frac{\gamma^3}{c^2}  \left(  \int (\Lambda_0 Q )^2\right)^{-1} \left(\frac{c^2}{\gamma^2}  \int (\Lambda_0 Q )^2 \int \tilde \eta_1 \tilde L_0 \tilde \eta_1  + \left( \int \tilde \eta_1 Q\right)^2 -4  \frac{c^2}{\gamma^2}  \int \tilde\eta_1 Q \int\tilde \eta_1 \Lambda_0Q + m_c^2 \left(\int\tilde\eta_1 \Lambda_0Q \right)^2 \right)\\
 &~{}= \frac{\gamma^3}{c^2}  \left(  \int (\Lambda_0 Q )^2\right)^{-1} \\
 &~{} \qquad \qquad \left( \frac12m_c (1-3\alpha^2)  \int (\Lambda_0 Q )^2 \int Q^{2\alpha}  + \left( \int  Q^{\alpha+1}\right)^2 - 2m_c  \int  Q^{\alpha+1} \int Q^{\alpha} \Lambda_0Q + m_c^2 \left(\int Q^{\alpha} \Lambda_0Q \right)^2 \right)\\
 &~{} =\frac{\gamma^3}{c^2} \left(\int (\Lambda_0 Q )^2\right)^{-1}  \left(  \int Q^{\alpha+1} \right)^2 \left( 1+ a_1 m_c +a_2 m_c^2 \right),
\end{aligned}
\]
with
\[
a_1=\frac12(1-3\alpha^2)\left(  \int Q^{\alpha+1} \right)^{-2} \int Q^{2\alpha}\int(\Lambda_0 Q )^2  -2 \left( \frac{1}{p-1} -\frac1{2(\alpha+1)}\right),\quad a_2:= \left( \frac{1}{p-1} -\frac1{2(\alpha+1)}\right)^2.
\]
A simple numerical computation reveals that $a_2>0$ for all $1<p<6.6$ and $a_1<0$ (see Fig. \ref{positividad}). Therefore, for each $1<p<5$, there exists $c_+(p) > 0$ such that, if $|c|>c_+$, \eqref{chequeo} is positive. Indeed, one has $1+ a_1 m_c +a_2 m_c^2>0$ for $m>m_{c,+}(p)$, with $m_{c,+}(p):=\frac12(|a_1| +\sqrt{a_1^2-4a_2})>0$. Therefore, $c_+(p):=\sqrt{m_+(p)/(1+m_+(p))}$.

\begin{figure}[htbp]
   \centering
   \includegraphics[width=0.45\linewidth]{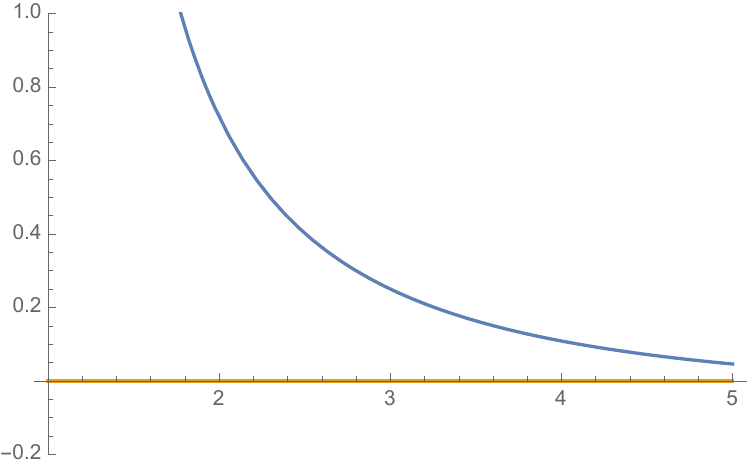} \hfill 
   \includegraphics[width=0.45\linewidth]{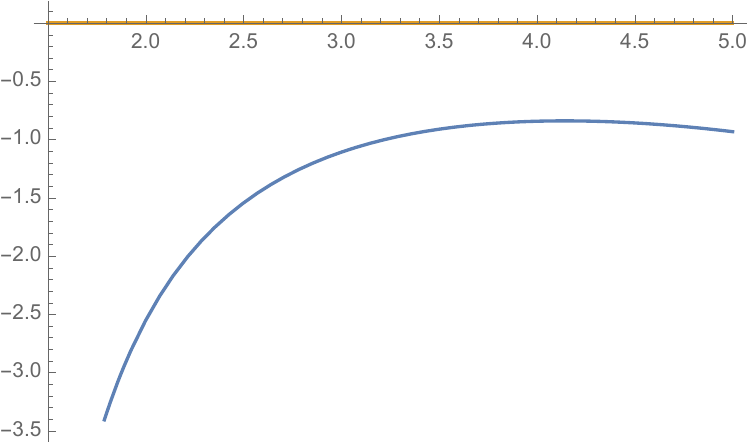} \vfill
      \includegraphics[width=0.45\linewidth]{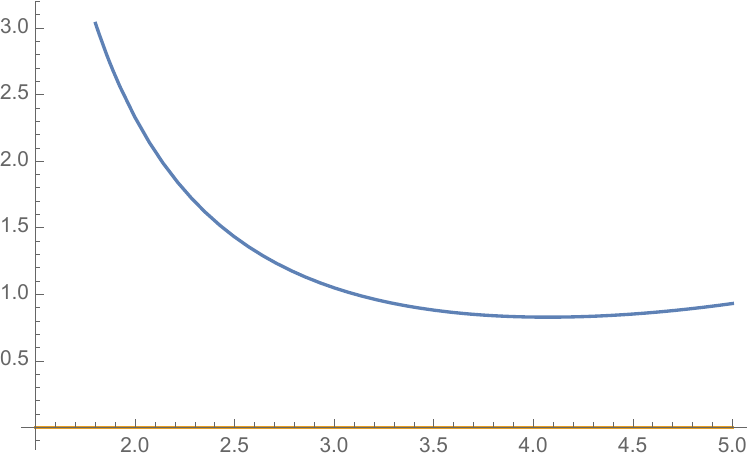} \hfill 
   \includegraphics[width=0.45\linewidth]{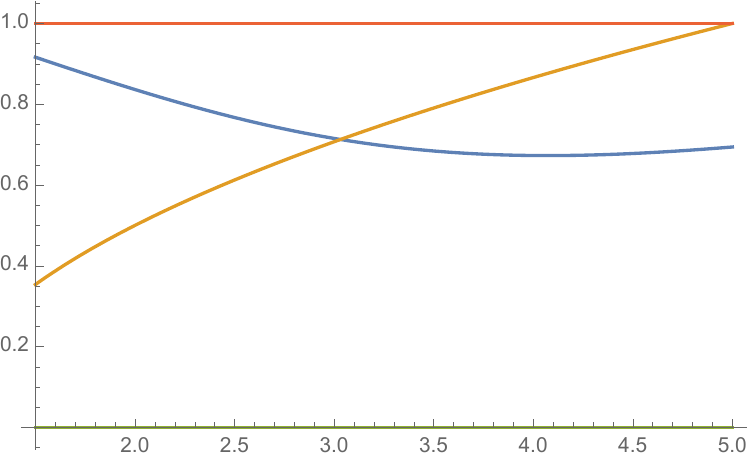}
   \caption{Above left: the coefficient $a_2(p)$ (blue) for $1<p<5$. The yellow line represents the zero value. Notice that $a_2=0$ for $p\sim 6.6$. Above right: the coefficient $a_1(p)<0$ (blue), in yellow, the zero value. Below left: $m_{c,+}(p)$ (blue), in yellow the zero value. Below right: $c_+(p)$ (blue), $p\mapsto \sqrt{\frac{p-1}{4}}$ (yellow). In green the zero value, in red the value 1.}
   \label{positividad}
\end{figure}

In the particular case $p=3$ we can compute all terms. First, for $p=3$, $\alpha=1$, $a_2=\frac1{16}$, and after a computation using Mathematica, $a_1=-\frac1{36}(12+\pi^2) -\frac12$. Therefore, $m_{c,+}(3) =\frac1{72}(30+\pi^2 +\sqrt{576 +60\pi^2 +\pi^4})\sim 1.048$ and $c_+(3)\sim 0.7153$. On the other hand, $\sqrt{(p-1)/4}\sim 0.7071$. One can check that for $p$ slightly above 3, one has $c_+(p)<\sqrt{\frac{p-1}4}$. 

We have proved that 
\[
\langle \bd{\tilde L} \bd{\eta} ,\bd{\eta} \rangle = \int \left( 3\eta _{2,y}^2 +(1-c^2)\eta_2^2 -f'(Q_c)\eta_2^2 +(\eta_1+c\eta_2)^2 \right) \geq 0.
\]	
We conclude \eqref{coercoer_new}. Let $\bd{\eta}_n$ be a sequence of unit $H^1\times L^2$ norm such that  
\[
(n+1) \int (3\eta_{2,n,y}^2 +(1-c^2)\eta_{2,n}^2 -f'(Q_c)\eta_{2,n}^2 +(\eta_{1,n}+c\eta_{2,n})^2)   + \left|  \langle  \Lambda \QQc , \bd{\eta}_n \rangle \right|^2 <\frac1{n+1}.
\]
Passing to a subsequence, one has $\bd{\eta}_n \to \bd{\eta}\neq 0$ weak in $H^1\times L^2$ and strong on compact sets. Moreover, $ \langle  \Lambda \QQc , \bd{\eta} \rangle=0 $, and $\int f'(Q_c)\eta_{2,n}^2 \to \int f'(Q_c)\eta_{2}^2$. Using weak convergence limit,
\[
\begin{aligned}
&  \int (3\eta_{2,y}^2 +(1-c^2)\eta_{2}^2 -f'(Q_c)\eta_{2}^2 +(\eta_{1,n}+c\eta_{2})^2) \\
& \qquad \leq \liminf_n  \int (3\eta_{2,n,y}^2 +(1-c^2)\eta_{2,n}^2 -f'(Q_c)\eta_{2,n}^2 +(\eta_{1,n}+c\eta_{2,n})^2)  =0,
\end{aligned}
\]
a contradiction with the fact that if $\bd{\eta}\neq 0$, under $ \langle  \Lambda \QQc , \bd{\eta} \rangle  =0$ one has $ \int (3\eta_{2,y}^2 +(1-c^2)\eta_2^2 -f'(Q_c)\eta_2^2 +(\eta_1+c\eta_2)^2)>0$.
	\end{proof}

\subsection{Modulation}
During this subsection we will come back to the static variable $\uu(t,x)=\bd{v}(t,x-\rho(t))$ (see \eqref{v_u}), which from \eqref{eq:u_TD} is the solution to 
\begin{align}\label{eq:v_TD}
\partial_t \bd{u}= \partial_y \JJ( \bd{L}-c\bd{J}) \bd{u} +(\rho'-c) \T
-c' \D +\NN (u_1),
\end{align}
where $(\T,\D,{\bf Q}_{c,\rho})(x)=({\bf T}_c,{\bf D}_c,{\bf Q}_c)(x-\rho(t))$. This is done to prove the following modulation result:
\begin{lem}
There exists $C,\delta_{1}>0$ such that for any $c_0\in (-1,1)$, $\rho_0\in \R$ , $\delta\in (0,\delta_1)$,
if $\bd \phi (t)$ satisfies
\begin{equation}\label{expectation}
\sup_{t\in [T_1,T_2]} \|{\bd \phi}(t)-\Q_{c_0,\rho_0}\|_{H^1\times L^2}\leq \delta,
\end{equation}
there exist unique functions $c:[T_1,T_2]\to  (-1,1)$ and $\rho:[T_1,T_2]\to  \R$ of class $C^1$, such that one has the decomposition
\begin{equation}\label{revelation}
{\bd \phi}(t)=\QQrt+\uu(t),
\end{equation}
and the following hold:
\begin{enumerate}
\item Smallness. $\|\uu\|_{H^1\times L^2}\leq C\delta$.
\item Orthogonality. One has $\Jap{\bd{T}_{c,\rho}}{\bd{u}}=\Jap{\JJ \QQr}{\uu}= 0.$
\item Estimates on the modulation terms: 
\begin{equation} \label{eq:c'} 
\begin{aligned}
 |\rho'-c|  \lesssim&~{}  \left\|\Qr^{3/4}u_1\right\|_{L^2}+\left\|\Qr^{3/4}(u_2+cu_1)\right\|_{L^2},\\
	|c' |  \lesssim&~{}    \left\| \Qr^{(p-1)/2} u_1\right\|_{L^2}^2 +\left\| \Qr^{3/4}u_1  \right\|_{L^2}^2+ \left\| \Qr^{3/4}(u_2+cu_1)  \right\|_{L^2}^2.
\end{aligned}
\end{equation}
\end{enumerate}
\end{lem}

\begin{proof}
The existence of modulated parameters is a standard fact and we only sketch the main ideas. First of all we prove that there exists $C,\delta_{0}>0$ such that for any $c_0 \in (-1,1)$, $\rho_0\in \R$, $\delta\in (0,\delta_1)$, and for any
\begin{equation}\label{eq:small}
\| {\bd \phi}-\Q_{c_0,\rho_0}\|\leq \delta,
\end{equation}
there exists $c=c({\bd \phi}) \in (-1,1)$ and $\rho$ such that $\uu={\bd \phi}-\QQr$ satisfies $\|\uu\|_{H^1\times L^2}\leq C\delta$ and $\Jap{\bd{T}_{c,\rho}}{\bd{u}}=\Jap{\JJ \QQr}{\uu}= 0$. Later, we 
 use \eqref{expectation} and the ideas in \cite{KMMV20} to conclude \eqref{revelation}, $\|\uu\|_{H^1\times L^2}\leq C\delta$ and $\Jap{\bd{T}_{c,\rho}}{\bd{u}}=\Jap{\JJ \QQr}{\uu}= 0.$ In the case of fixed time perturbations \eqref{eq:small}, the orthogonality conditions $\Jap{\bd{T}_{c,\rho}}{\bd{u}}=\Jap{\JJ \QQr}{\uu}= 0$ are ensured since
\[
\det \begin{pmatrix} \Jap{\bd{T}_{c,\rho}}{ \partial_c \bd{Q}_{c,\rho}} & &  \Jap{\bd{T}_{c,\rho}}{ \partial_\rho \bd{Q}_{c,\rho}}  \\   \Jap{\JJ \QQr}{ \partial_c \bd{Q}_{c,\rho}} & &  \Jap{\JJ \QQr}{ \partial_\rho \bd{Q}_{c,\rho}}\end{pmatrix} =\det \begin{pmatrix} 0 & &  - \Jap{\bd{T}_{c,\rho}}{ \bd{T}_{c,\rho} }  \\   \Jap{\JJ \QQr}{ \Lambda \bd{Q}_{c,\rho}} & &  0 \end{pmatrix} \neq 0.
\]
The fact that 
\begin{equation}\label{stab}
\Jap{\JJ \QQr}{ \Lambda \bd{Q}_{c,\rho}}\neq 0
\end{equation}
is proved as follows. Recall \eqref{eq:TDGF}. One has $\Jap{ \D }{\G} = \partial_c \left( c\gamma^{\frac{5-p}{p-1}}\right)\|Q\|_{L^2}^2$.
Indeed, we get that 
	\[\begin{aligned}
	\Jap{\D}{\JJ\QQr}
	=&
	- \int \big( c\Qr \Lambda \Qr+\Qr  \Lambda (c\Qr) \big)
	 =-\partial_c \left( c\int   \Qr^2 \right),
	 \end{aligned}
	\]
	and taking in mind \eqref{eq:scaling_Q} joint to $\partial_c (\gamma)=-c\gamma^{-1}$, one obtains	\begin{equation}\label{calculo_base}
	\begin{aligned}
	\Jap{\D}{\JJ\QQr}
	=& -\partial_c \left( c \gamma^{\frac{4}{p-1}-1} \int   Q^2(\gamma y) \gamma dy \right) = -\partial_c \left( c \gamma^{\frac{4}{p-1}-1}\right)  \|Q\|_{L^2}^2 \\
	=& -\partial_c \left( c \gamma^{\frac{5-p}{p-1}}\right)  \|Q\|_{L^2}^2 
	 = - \frac{\gamma^{\frac{7-3p}{p-1}}}{p-1} \left(p-1-4c^2 \right)  \|Q\|_{L^2}^2.
	\end{aligned}
	\end{equation}
Since the solitary wave is stable, we have $c^2>\frac{p-1}4$ and \eqref{stab}. This last fact concludes the proof of existence of modulated parameters.

\medskip

Let ${\bf H}_{c,\rho}$ be any modulated vector valued in $H^1\times L^2$. It holds
\begin{equation}\label{aux00}
	\begin{aligned}
	\frac{d}{dt} \Jap{\uu}{\bd{H}_{c,\rho}}  
	   =&  -\Jap{ \uu }{ \bd{L} \JJ \partial_x \bd{H}_{c,\rho}}
  	+ (\rho'-c)\left( \Jap{ \T }{\bd{H}_{c,\rho}}  - \Jap{\uu}{\bd{H}'_{c,\rho}}\right)
	\\&
	 +c' \left( \Jap{\uu}{ \bd{\Lambda H}_{c,\rho}} - \Jap{ \D }{\bd{H}_{c,\rho}}\right)
	   + \Jap{\NN (u_1)}{\bd{H}_{c,\rho}}  .
\end{aligned}
\end{equation}
Indeed,
\[
\begin{aligned}
\frac{d}{dt} \Jap{\uu}{\bd{H}_{c,\rho}}
 =&  \Jap{\partial_t \uu}{\bd{H}_{c,\rho}}-\rho' \Jap{\uu}{\bd{H}'_{c,\rho}}+c'  \Jap{\uu}{\bd{ \Lambda H}_{c,\rho}}.
\end{aligned}
\]
Now, replacing \eqref{eq:v_TD}, one gets
\[
\begin{aligned}
\frac{d}{dt} \Jap{\uu}{\bd{H}_{c,\rho}}
 =&  \Jap{ \partial_x \JJ (\bd{L}-c\JJ )\uu +(\rho'-c) \T-c' \D +\NN (u_1)}{\bd{H}_{c,\rho}}
 -\rho' \Jap{\uu}{\bd{H}'_{c,\rho}}+c'  \Jap{\uu}{\bd{\Lambda H}_{c,\rho}}
 \\
  =&  \Jap{ \partial_x \JJ (\bd{L}-c\JJ )\uu }{\bd{H}_{c,\rho}}
  	+ (\rho'-c) \Jap{ \T }{\bd{H}_{c,\rho}}
	  -c'\Jap{ \D }{\bd{H}_{c,\rho}}\\
& -\rho' \Jap{\uu}{\bd{H}'_{c,\rho}}+c'  \Jap{\uu}{\bd{ \Lambda H}_{c,\rho}}
 + \Jap{\NN (u_1)}{\bd{H}_{c,\rho}}.
\end{aligned}
\]
Rearranging the terms and integrating by parts, we obtain \eqref{aux00}. Now we shall use this identity with $\bd{H}_{c,\rho} = \bd{T}_{c,\rho}$. We have
\[
\begin{aligned}
& \Jap{  \bd{T}_{c,\rho} }{ \bd{T}_{c,\rho} }=(1+c^2)\int Q_c'^2,\qquad
\Jap{ \D }{ \bd{T}_{c,\rho} }=0;
\\
& |\Jap{\uu}{\bd{\Lambda}  \bd{T}_{c,\rho}}|+|\Jap{ \uu }{ \bd{L} \JJ \partial_x  \bd{T}_{c,\rho} }|
+|\Jap{\uu}{ \bd{T}'_{c,\rho} }|\lesssim C \| \uu\|_{H^1\times L^2}.
\end{aligned}
\]
Then, using that $\|\uu\|_{H^1\times L^2}\leq \delta$ and $\bd{H}_{c,\rho} = \bd{T}_{c,\rho}$, we obtain
\begin{equation}\label{rhorho}
	\begin{aligned}
  	|\rho'-c|\lesssim&
	 |\rho'-c|\left| \Jap{ \T }{\bd{T}_{c,\rho}}  +\Jap{\uu}{\bd{T}'_{c,\rho}}\right|
	\\
	\lesssim&~{}
	| \Jap{ \uu }{ \bd{L} \JJ \partial_x \bd{T}_{c,\rho}}|
	 +|c'| \left|  \Jap{ \D }{\bd{T}_{c,\rho}}-\Jap{\uu}{\bd{\Lambda T}_{c,\rho}} \right|
	   +\|\Qr u_1\|_{L^2}^2 
	   \\	
	   \lesssim& ~{} (1+|c'|) \left(\left\| \Qr^{3/4}u_1\right\|_{L^2}+\left\|\Qr^{3/4}(u_2+cu_1)\right\|_{L^2} \right).
\end{aligned}
\end{equation}
The exponent $3/4$ comes from the fact that $\bd{\Lambda}\bd{T}_{c,\rho}$ contains terms of the form $|y| Q_c(y)$. Now we use again \eqref{aux00} with $\bd{H}_{c,\rho} = \bd{G}_{c,\rho}=\bd{J Q}_{c,\rho}$. First, we observe now that
\begin{equation}\label{eq:u_G}
\begin{aligned}
& \Jap{ \uu }{ \bd{L} \JJ \partial_x \G }=0, 
 \quad 
\Jap{ \T }{\G} =0, 
\quad	|\Jap{\NN (u_1)}{\G} |\lesssim \int  f''(\Qr) \Qr
	 u_1^2 \lesssim \int \Qr^{p-1} u_1^2. 
\end{aligned}
\end{equation}
The previous identities are direct. Recalling that $\Jap{\uu}{\G}=0$, applying \eqref{aux00} and \eqref{eq:u_G}, one has 
\[
\begin{aligned}
c' \left( \Jap{\uu}{\bd{\Lambda}\G}  - \Jap{ \D }{\G}\right)
  	+ (\rho'  -c)  \left( \Jap{ \T }{\G}  -\Jap{\uu}{\G'}\right)	   = - \Jap{\NN (u_1)}{\G},
\end{aligned}
\]
 using that $\|\uu\|_{H^1\times L^2}\leq \delta$, we conclude
\[
\begin{aligned}
	|c' | \lesssim&~{} |c'| ~{}\big| \Jap{\uu}{\bd{\Lambda}\G} - \Jap{ \D }{\G}\big|
	   \lesssim  |\Jap{\NN (u_1)}{\G} |+ |\rho'-c| |\Jap{\uu}{\G'}|
	   \\
	    \lesssim&~{}   \int \Qr^{p-1} u_1^2 + |\rho'-c| \left( \left\| \Qr^{3/4}u_1  \right\|_{L^2}+ \left\| \Qr^{3/4}(u_2+cu_1)  \right\|_{L^2}\right).
\end{aligned}
\]
Gathering \eqref{rhorho} and the previous estimate, we obtain \eqref{eq:c'}.
\end{proof}

\subsection{Orbital stability} In this subsection, and for the sake of completeness, we provide some orbital stability quantitative estimates in the case of the orthogonality conditions \eqref{eq:orto}. For a similar proof, see \cite{Bona-Sachs}.
\begin{thm}\label{MT1}
Let $1<p< 5$, $\frac{p-1}{4} < c^2 <1$, and $x_0\in\mathbb R$ be fixed parameters and $\Q_c(t,x)$ as in \eqref{eq:QQ}. 
There exists $\delta_0,C_0>0$ such that for all $0<\delta<\delta_0$, the following holds. Assume that $(\phi_{1,0},\phi_{2,0})\in H^1\times L^2$ are such that $\| (\phi_{1,0},\phi_{2,0})- \Q_c(0,\cdot)\|_{H^1\times L^2}<\delta$. Then the corresponding solution $(\phi_1,\phi_2)(t)$ to \eqref{eq:gGB} with initial data $(\phi_{1,0},\phi_{2,0})$ at time $t=0$ is globally defined in $H^1\times L^2$ for all $t\geq 0$ and there is a shift $\rho(t)$ such that 
\begin{equation}\label{OS}
\sup_{t\geq 0} \| (\phi_1,\phi_2)(t,\cdot + \rho(t))- \Q_{c}\|_{H^1\times L^2} \leq C_0\delta.
\end{equation}
Moreover, one has $|c(t)-c(0)| \lesssim C_0^2\delta^2.$ 
\end{thm}

\begin{proof}
We follow the Bona-Sach's argument, but now orthogonality conditions \eqref{eq:orto} differ. From \eqref{eq:energy} and \eqref{eq:soliton_eq}, after modulation, one has
\[
\begin{aligned}
& E[\bd{Q}_c + \bd{v}] + cP[\bd{Q}_c + \bd{v}]  \\
 &~{} =  E[\bd{Q}_c] + c P[\bd{Q}_c ] + \int \left( -cQ_c v_2 + Q_c v_1 + Q_c' \partial_y v_1 -f(Q_c) v_1 + cQ_c v_2 -c^2Q_c v_1 \right)\\
&~{}  \quad + \frac12 \int \left( (v_2 +cv_1)^2 +(\partial_x v_1)^2 + (1-c^2)v_1^2  -f'(Q_c) v_1^2\right) \\
&~{} \quad - \int \left(  F(Q_c+v_1)-F(Q_c) -f(Q_c)v_1- \frac12 f'(Q_c)v_1^2\right) \\
&~{} = E[\bd{Q}_c] + c P[\bd{Q}_c ] + \frac12\langle \bd{L} \bd{v} ,\bd{v} \rangle +\bd{N}_2(\bd{v}) .
\end{aligned}
\]	 
The term $\bd{N}_2(\bd{v})$ naturally obeys the bound $|\bd{N}_2(\bd{v})| \lesssim \|\bd{v}\|_{H^1\times L^2}^3$. Additionally,
\[
\langle \bd{L} \bd{v} ,\bd{v} \rangle = \int (v_1Lv_1 +2cv_1v_2 +v_2^2) = \int v_1 \mathcal L v_1 + \int (v_2 +cv_1)^2 . 
\]
Thanks to Lemma \ref{lem2p2}, and \eqref{coer0}, we have  that there exists $c_0>0$ such that $\langle \bd{L} \bd{v} ,\bd{v} \rangle \geq c_0 \| \bd{v} \|_{H^1\times L^2}$. Consequently, since the energy and the momentum are conserved,
\[
\begin{aligned}
& c_0 \| \bd{v}\|_{H^1 \times L^2}^2 \\
& \le - (E[\bd{Q}_c] -E[\bd{Q}_{c(0)}]) - (c P[\bd{Q}_c ]  -c(0) P[\bd{Q}_{c(0)} ] ) +  (c(t)-c(0)) P[\bd{Q}_c + \bd{v}]  + C \| \bd{v}(0)\|_{H^1 \times L^2}^2 +C\| \bd{v}\|_{H^1 \times L^2}^3 \\
& \le - (E[\bd{Q}_c] -E[\bd{Q}_{c(0)}])  -c(0) ( P[\bd{Q}_c ]  -  P[\bd{Q}_{c(0)} ] ) + C|c(t)-c(0)| \|\bd{v}\|_{H^1\times L^2} +C\delta^2 + CC_0^3 \delta^3.
\end{aligned}
\]
Recall that $\int Q_c^2 =  4\gamma^{\frac{5-p}{p-1}}$, $\int Q_c^{p+1} = 2\left( \frac{p+1}{p+3}\right)(1-c^2) \int Q_c^2$, $\int Q_c'^2 = \frac{p-1}{2(p+1)}\int Q_c^{p+1}=(1-c^2)\frac{p-1}{p+3} \int Q_c^2$, $E[\bd{Q}_c] =\frac12 \int ((1+c^2)Q_c^2 + Q_c'^2 -\frac2{p+1} Q_c^{p+1}) = \frac{p-1 +4c^2}{p+3} \int Q_c^2= \left( 1 -\frac{4\gamma^2}{p+3}\right) \int Q_c^2$ and $P[\bd{Q}_c] =-c\int Q_c^2$. Therefore, since $\partial_c E[\bd{Q}_c] + c \partial_c P[\bd{Q}_c]=0$, one has
\[
-(E[\bd{Q}_c] -E[\bd{Q}_{c(0)}])  -c(0) ( P[\bd{Q}_c ]  -  P[\bd{Q}_{c(0)} ] ) \lesssim |c(t)-c(0)|^2.
\]
Finally, since $P[\bd{Q}_c + \bd{v}]  = -c\int Q_c^2 + \int Q_c v_2 -c\int Q_c v_1 +\int v_1 v_2 = -c\int Q_c^2 +\int v_1 v_2$, one has $|c(t)-c(0)| \lesssim \| \bd{v}\|_{H^1 \times L^2}^2$. We conclude that 
\[
\begin{aligned}
  \| \bd{v}\|_{H^1 \times L^2}^2  \le C\delta^2 + CC_0^3 \delta^3 \le C \delta^2.
\end{aligned}
\]
This ends the proof of orbital stability.
\end{proof}

\section{First virial estimates}\label{Sec:3}

In what follows, we assume Theorem \ref{MT1} and the boundedness property \eqref{OS}.

\subsection{Notation for virial argument}
Now recall the weights needed for virial arguments \cite{KMM}. Consider a smooth even function $\chi:\R\to \R$ satisfying
\[
\chi=1 \mbox{ on } [-1,1], \quad \chi=0 \mbox{ on } (-\infty,2]\cup [2,\infty), \quad \chi'\leq 0 \mbox{ on } [0,\infty).
\]
For $A,B>0$, define the functions
\[
\begin{aligned}
& \zeta_A(x)=\exp\left( -\dfrac{1}{A}(1-\chi(x))|x|\right),
\quad \varphi_A(x)=\int_{0}^{x} \zeta_A^2(y)dy,\\
& \zB(x)=\exp\left( -\dfrac{1}{B}(1-\chi(x))|x|\right),
\quad \varphi_B(x)=\int_{0}^{x} \zB^2(y)dy.
\end{aligned}
\]
Consider the function $\psi_{A,B}$ defined as
\begin{equation}\label{eq:psi_eti}
\psi_{A,B}(x)=\chi^2_A(x)\varphi_B(x)\quad  \mbox{ where }\quad \chi_A(x)=\chi\left(\dfrac{x}{A}\right), \ \  x \in \R.
\end{equation}
These functions will be used in two distinct virial arguments with different scales
\begin{equation}\label{eq:scales}
	1\ll B \ll B^{10}  \ll A.
\end{equation}
The following technical estimates on functions $\vA$, $\zeta_K$ and $\chi_{A}$ will be useful  (for the proofs, see \cite{Mau}). Notice that for the function $\vA$, it holds
	\begin{equation}\label{eq:vA_prop}
	\begin{aligned}
	0<\vA'(x) \leq 1, \quad |\vA(x)|\leq |x|, \quad |\vA(x)|\leq CA,
	\quad |\vA(x)\Qc(x) |+|\vA (x) \Lambda \Qc(x)|\leq C\sech\left(\frac{3\gamma}{4}x\right).
	\end{aligned}
	\end{equation}
	Similarly, one can see that $\vB$ and $\psi_{A,B}$ satisfies the following estimates
	\begin{equation}\label{eq:vB_psi_prop}
	\begin{aligned}
	|\vB(x)|\leq CB,\quad |\psi_{A,B}|\leq CB \chi_A^2(x).
	\end{aligned}
	\end{equation}
	Finally, one simple but fundamental inequality is: 
	for each function $v$
	\begin{equation}\label{eq:chiA_zA}
		\int \chi_A^2 v^2\leq\int_{|x|\leq 2A} v^2 \leq C\int_{|x|\leq 2A} e^{-4|x|/A}v^2 \lesssim 	\int v^2 \zeta_A^4 \leq \| \zeta_A^2 v\|^2_{L^2} .
	\end{equation}

We also observe that for any $K>0$ sufficiently large,
		\begin{equation}\label{eq:z11}
			\left|\dfrac{ \zeta_K''}{\zeta_K} -2\left(\dfrac{\zeta_K'}{\zeta_K} \right)^{2}\right|\lesssim \frac{1}{K},\quad 
\left| \frac{\zeta_K'}{\zeta_K}\right|\lesssim K^{-1} \textbf{1}_{\{|x|>1\}}(x),
		\quad 
				\left| \frac{\zeta_K''}{\zeta_K}\right|
				\lesssim   K^{-2}+K^{-1} \sech(x),
		\end{equation}
		and
		\begin{equation}\label{eq:z'/z}
		\begin{aligned}
			\left| \frac{\zeta_K'''}{\zeta_K}\right|
			\lesssim & ~{}  K^{-3}+K^{-1} \sech(x),\quad 
			\left| \frac{\zeta_K^{(4)}}{\zeta_K}\right|
			\lesssim   K^{-4}+K^{-1} \sech(x), \quad \left| \frac{\zeta_K''}{\zeta_K}\right| +
	\left| \frac{\zeta_K'''}{\zeta_K}\right|+
	\left|  \frac{\zeta_K^{(4)}}{\zeta_K}\right|\lesssim   K^{-1}.
		\end{aligned}
	\end{equation}
	In particular, for $A\geq K$ large enough, the following estimate holds:	 
	\begin{equation}\label{eq:C*z'/z}
	\begin{aligned}
			\left| \textbf{1}_{\{A<|x|<2A\}} \frac{\zeta_K^{(n)}}{\zeta_K}\right|\lesssim   \frac1{K^{n}}, \mbox{ for } n\in\mathbb{N}.
	\end{aligned}
	\end{equation}

\subsection{First virial estimate}

	Let $\varphi$ a bounded smooth function and consider the functional
	\begin{equation}\label{eq:I}
	\mathcal{I}=\int \varphi_A(y) v_1(t,y) v_2(t,y) dy,
	\end{equation}
and
\begin{equation} \label{eq:wi}
w_i=\zeta_A v_i, \quad  i=1,2.
\end{equation}

\begin{prop}\label{prop:dtI}
There exists $C_1>0$, and $\delta_1$ such that for all $\delta \in [0,\delta_1]$ the following holds. Let
\begin{equation}\label{eq:A}
A=\delta^{-1},
\end{equation}
and assume that for all $t\geq0$, \eqref{eq:small}  holds. Then
\begin{equation}\label{cota_I}
\begin{aligned}
\frac{d}{dt}\I 
\leq&~{}  -\frac12 \int  \left( (w_2 + c   w_1)^2+ 3 (\partial_x w_1)^2 +\bigg(1-c^2-f'(\Qc)+\frac{\zA''}{\zA}- 2\left(\frac{\zA'}{\zA}\right)^2 \bigg) w_1^2 \right)
\\&
	+C_1 |\rho'-c| \left( \|w_2+cw_1\|_{L^2}  +\|w_1\|_{L^2}  \right) 
	+ C_1 |c'| \left(\frac{c}{\gamma^2}\|w_2+c w_1\|_{L^2}+\|w_1\|_{L^2} \right)  
	\\&
	+C_1  \int\sech^2\left(\frac{\gamma}{4}y\right) v_1^2 
	.
	\end{aligned}
\end{equation}
\end{prop}
In order to prove this estimate, we shall need the following identity, proved in the appendix.
\begin{lem}
Let $(v_1,v_2)\in H^1(\R)\times L^2(\R)$ a solution of \eqref{eq:eq_lin}. Consider $\varphi_A =\varphi_A (y)$ a smooth bounded function to be chosen later. Then
\begin{equation}\label{eq:dtI}
	\begin{aligned}
   \frac{d}{dt}\I
	=&-\frac12 \int \vA' \left[ (v_2 + c   v_1)^2+ 3 (\partial_x v_1)^2 +(1-c^2)v_1^2 - f'(\Qc) v_1^2 \right]
	 -(\rho'-c)  \int \vA'  v_1 v_2
	\\
	&+\frac12 \int \vA (1-f'(\Qc))'  v_1^2
	+\frac12 \int \vA''' v_1^2 
	+(\rho'-c) \int \vA \Qc' \left(  v_2 -c  v_1\right)
	\\&
	- c' \int  \vA  \left( v_2  ~\Lambda (\Qc) +  v_1~\Lambda (c\Qc) \right)
	-\int \vA  v_1 N.
   \end{aligned}
\end{equation}
\end{lem}
\begin{proof}
See Appendix \ref{virial2}. 
\end{proof}

\subsection{Proof of Proposition \ref{prop:dtI}}
Now we rewrite  the main part of the virial identity, obtained in \eqref{eq:dtI}, using the localized variables $(w_1,w_2)$, defined in \eqref{eq:wi}.

\begin{lem}\label{lem:var_v_w}
It holds
\[
	\begin{aligned}
\int \vA'  (\partial_x v_1)^2 
=& \int (\partial_x w_1)^2 + \int   \frac{\zA''}{\zA}  w_1^2,\quad 
	\frac12 \int (\zA^2)'' v_1^2 
	=\int \bigg(\left(\frac{\zA'}{\zA}\right)^2+\frac{\zA''}{\zA} \bigg)w_1^2,
   \end{aligned}
\]
and
\[
\begin{aligned}
-\frac12 \int & \vA'  \left( (v_2 + c   v_1)^2+ 3 (\partial_x v_1)^2 +(1-f'(\Qc)-c^2) v_1^2 \right)
	 -(\rho'-c)  \int \vA'  v_1 v_2
	+\frac12 \int \vA''' v_1^2 
	\\
=&
-\frac12 \int  \left( (w_2 + c   w_1)^2+ 3 (\partial_x w_1)^2 +\bigg(1-c^2+\frac{\zA''}{\zA}- 2\left(\frac{\zA'}{\zA}\right)^2 -f'(\Qc) \bigg) w_1^2 \right)
-(\rho'-c)  \int   w_1 w_2 .
\end{aligned}
\]
\end{lem}

\begin{proof}
Since  $\vA'=\zA^2$ and $w_i=\zA v_i$, then from the definition one gets $\zA \partial_x v_1= \partial_x w_1- \frac{\zA'}{\zA} w_1$. Replacing the previous identity and integrating by parts, we obtain
\[ 
	\begin{aligned}
 \int \zA^2  (\partial_x v_1)^2 
=&~{}\int \bigg( (\partial_x w_1)^2+ \left(\frac{\zA'}{\zA} w_1\right)^2-   \frac{\zA'}{\zA}\partial_x (w_1)^2 \bigg)
= \int (\partial_x w_1)^2 + \int   \frac{\zA''}{\zA}  w_1^2.
   \end{aligned}
\] 
For the second identity, expanding the derivatives and using the definition of $w_1$, one gets the simplified expression $\frac12 \int (\zA^2)'' v_1^2 = \int \left(\left(\frac{\zA'}{\zA}\right)^2+\frac{\zA''}{\zA} \right) w_1^2 .$ Now, collecting the previous identities and by the definition of $w_i$ (see \eqref{eq:wi}), we get

\[
\begin{aligned}
-\frac12 \int & \vA'  \left( (v_2 + c   v_1)^2+ 3 (\partial_x v_1)^2 +(1-f'(\Qc)-c^2) v_1^2 \right)
	 -(\rho'-c)  \int \vA'  v_1 v_2
	+\frac12 \int \vA''' v_1^2 
	\\
=&
-\frac12 \int  \left( (w_2 + c   w_1)^2+ 3 (\partial_x w_1)^2 +\bigg(1-c^2-f'(\Qc)+\frac{\zA''}{\zA}- 2\left(\frac{\zA'}{\zA}\right)^2 \bigg) w_1^2 \right)
	-(\rho'-c)  \int   w_1 w_2	.
\end{aligned}
\]
This concludes the proof of the lemma.
\end{proof}

Lemma \ref{lem:var_v_w} returns the first part in \eqref{cota_I}. Consider \eqref{eq:dtI}. Now, we deal with the terms
\[
\begin{aligned}
I_{1}=& (\rho'-c) \int \vA \Qc' \left(  v_2 -c  v_1\right),
\quad I_{2}=- c' \int  \vA  \left( v_2  ~\Lambda (\Qc) +  v_1~\Lambda (c\Qc) \right),\\ 
I_3=	&~{} -\int \vA  v_1 N- \frac12 \int \vA (f'(\Qc))'  v_1^2.
\end{aligned}
\]

\subsection*{Step 1} Control of $I_1$. Recall \eqref{eq:vA_prop}. Since $|\varphi_A|\leq x$,
\begin{equation}\label{eq:vA_Qc}
|\vA \Qc'|\lesssim|x \Qc'|\lesssim \exp\left(-\frac12\gamma|x|\right).
\end{equation}
Using the above inequality we obtain
\[
|I_1|\leq |\rho'-c| \int \exp\left( -\frac12\gamma|x|\right) \zA^{-1} |w_2-cw_1|.
\] 
Recall that by Theorem \ref{MT1} one has $\gamma$ uniformly far from 0.  Since  $B^{1/4}\gamma > 2 $ if $\delta$ is chosen large, and using  the H\"older inequality, we  conclude
\begin{equation}\label{eq_I1}
\begin{aligned}
|I_1|
\lesssim&~{} |\rho'-c| \left(\int \exp\left( -\gamma|x|\right) \zA^{-2}\right)^{1/2} \|w_2-cw_1\|_{L^2}
\\
\lesssim&~{} |\rho'-c| \left(\int \exp\left( -\gamma|x|\right) \zA^{-2}\right)^{1/2}\left( \|w_2+cw_1\|_{L^2}  +\|w_1\|_{L^2}  \right) \lesssim |\rho'-c| \left( \|w_2+cw_1\|_{L^2}  +\|w_1\|_{L^2}  \right).
\end{aligned}
\end{equation}
\subsection*{Step 2} Control of $I_2$.
Now, notice that
\[
\begin{aligned}
I_{2}=&	- c' \int  \vA  \left( v_2  ~\Lambda (\Qc) +  v_1~\Lambda (c\Qc) \right)
\\
=&	- c' \int  \vA  \left( v_2  ~\Lambda (\Qc) +  v_1~(c\Lambda \Qc+\Qc) \right)
=	- c' \int  \vA  \left(( v_2+cv_1)  ~\Lambda \Qc +  v_1\Qc \right)
=:I_{21}+I_{22}.
\end{aligned}
\]
Similarly as in $I_1$, using \eqref{eq:vA_Qc}, we obtain 
\begin{equation}\label{I22}
|I_{22}|\lesssim |c'| \|w_1\|_{L^2}.
\end{equation}
Now, for $I_{21}$, we recall that $\Lambda \Qc= \frac{-c}{\gamma^2}\left[ \frac{2}{p-1} \Qc+x \gamma \Qc'\right]$. Then, we obtain 
\[ 
\begin{aligned}
\left|\frac{\vA}{\zA} \Lambda \Qc \right| 
\lesssim& ~{} \frac{c}{\gamma^2} \left|x\zA^{-1} \left(\Qc+x \gamma \Qc'\right)\right|
\lesssim \frac{c}{\gamma^2}\left( \left|x\zA^{-1}\Qc\right|+ \left| x^2\zA^{-1}  \Qc' \right|\right) \lesssim  \frac{c}{\gamma^2} \exp\left(-\frac14\gamma|x|\right).
\end{aligned}
\] 
By Theorem \ref{MT1} one has $\gamma>0$ uniformly in time and $\delta$. Then, we obtain
\[
|I_{21}|\lesssim |c'| \left(\frac{c}{\gamma^2}\|w_2+c w_1\|_{L^2}+\|w_1\|_{L^2} \right).
\]
Finally, from \eqref{I22} and the previous estimate we conclude that $I_{2}$ satisfies
\begin{equation}\label{eq_I2}
\begin{aligned}
|I_2|\lesssim&~{}
  |c'| \left(\frac{c}{\gamma^2}\|w_2+c w_1\|_{L^2}+\|w_1\|_{L^2} \right).
\end{aligned}
\end{equation}

\subsection*{Step 3} Control of $I_3$. 
Let us split $I_3 = I_{31}+I_{32}$, where $I_{31}$ is the term that contains the term $N$.
Recalling $N$  \eqref{eq:L_N} and integrating by parts in $I_{31}$, we obtain 
\[\begin{aligned}
I_{31}
=&~{} \int \partial_x (\vA  v_1) (f (v_1+\Qc) - f (\Qc)- f'(\Qc)v_1)
\\
=&~{} \int (\zA^2v_1 +\vA  \partial_x  v_1) (f (v_1+\Qc) - f (\Qc)- f'(\Qc)v_1)
\\
=&~{} \int  \zA^2 v_1 (f (v_1+\Qc) - f (\Qc)- f'(\Qc)v_1)
-\int  \vA  \partial_x  v_1 f'(\Qc)v_1
+\int  \vA  \partial_x  v_1 (f (v_1+\Qc) - f (\Qc))
\\
=&~{} \int  \zA^2 v_1 (f (v_1+\Qc) - f (\Qc)- f'(\Qc)v_1)
\\&
+\frac12 \int  (\zA^2  f'(\Qc)+\vA  (f'(\Qc))') v_1^2
+\int  \vA  \partial_x  v_1 (f (v_1+\Qc) - f (\Qc))
\\
=&~{} \int  \zA^2 v_1 (f (v_1+\Qc) - f (\Qc)- f'(\Qc)v_1)
+\frac12 \int  \zA^2  f'(\Qc) v_1^2
\\&
+\frac12 \int  \vA  (f'(\Qc))' v_1^2
+\int  \vA  \partial_x  v_1 (f (v_1+\Qc) - f (\Qc))
.
\end{aligned}
\]
Rewriting the last integral on the RHS, we have
\[\begin{aligned}
\int  \vA  \partial_x  v_1 (f (v_1+\Qc) - f (\Qc))
=&~{} \int  \vA  \partial_x  \left( F (v_1+\Qc) - F (\Qc)-f(\Qc)v_1\right)
\\&
-\int \vA \Qc' \left( f(\Qc+v_1)-f(\Qc)-f'(\Qc)v_1\right).
\end{aligned}
\]
Collecting the above relation and integrating by parts, we obtain
\[
\begin{aligned}
I_3
 = &~{} \int  \zA^2 v_1 (f (v_1+\Qc) - f (\Qc)- f'(\Qc)v_1)
-\int  \zA^2  \left( F (v_1+\Qc) - F (\Qc)-f(\Qc)v_1\right)
\\
&
+\frac12 \int  \zA^2 f'(\Qc) v_1^2
-\int \vA \Qc' \left( f(\Qc+v_1)-f(\Qc)-f'(\Qc)v_1\right).
\end{aligned}
\]
Finally, since $\|v_1\|_{L^{\infty}}\leq \delta$ and $p\geq 2$, using the following Taylor expansion 
\begin{equation}\label{eq:taylor}
|F (v_1+\Qc) - F (\Qc)- f(\Qc)v_1|+
|f (v_1+\Qc) - f (\Qc)- f'(\Qc)v_1|\lesssim (f'(\Qc)+f''(\Qc) )|v_1|^2+|v_1|^{p},
\end{equation}
and \eqref{eq:vA_Qc}, we conclude that
\[
\begin{aligned}
|I_3| \lesssim& \int\sech^2\left(\frac{\gamma}{4}x\right) v_1^2.
\end{aligned}
\]
Gathering the previous estimate together with estimates \eqref{eq_I1} and \eqref{eq_I2} we get 
\[ \label{cota_I}
\begin{aligned}
\frac{d}{dt}\I 
\leq&~{}  -\frac12 \int  \left( (w_2 + c   w_1)^2+ 3 (\partial_x w_1)^2 +\bigg(1-c^2-f'(\Qc)+\frac{\zA''}{\zA}- 2\left(\frac{\zA'}{\zA}\right)^2 \bigg) w_1^2 \right)
\\&
	+C_1 |\rho'-c| \left( \|w_2+cw_1\|_{L^2}  +\|w_1\|_{L^2}  \right)
	+ C_1 |c'| \left(\frac{c}{\gamma^2}\|w_2+c w_1\|_{L^2}+\|w_1\|_{L^2} \right)  
	\\&
	+C_1  \int\sech^2\left(\frac{\gamma}{4}x\right) v_1^2 .
\end{aligned}
\] 
Now \eqref{eq:c'} reads in $(v_1,v_2)$ variables
\begin{equation} \label{eq:c'_new} 
\begin{aligned}
 |\rho'-c|  \lesssim&~{}   \left\|\Qc^{3/4}v_1\right\|_{L^2}+\left\|\Qc^{3/4}(v_2+cv_1)\right\|_{L^2},\\
	|c' |  \lesssim&~{}    \left\| \Qc^{(p-1)/2} v_1\right\|_{L^2}^2 +\left\| \Qc^{3/4}v_1  \right\|_{L^2}^2+ \left\| \Qc^{3/4}(v_2+cv_1)  \right\|_{L^2}^2.
\end{aligned}
\end{equation}
Using these estimates one gets \eqref{cota_I}.

\section{Transformed system}\label{Sec:4}
Estimate \eqref{cota_I} is not enough to conclude local decay of solitary wave perturbations. Let us fix $\varepsilon>0$. Let us consider the following change of variable \cite{Martel_linearKDV,CMPS}
	\begin{equation}\label{eq:z_cv}
	\bd{z} = \Opg \bd{ L}\bd{v} = \Opg \begin{pmatrix} \LL + c^2 & c \\  c & 1 \end{pmatrix} \bd{v},
	\end{equation}
	where $\bd{L}$ is given in \eqref{eq:LL}. The above change of variable is equivalent to
	\begin{equation}\label{def_z}
	\begin{aligned}
	z_1 =\Opg (\LL v_1 + c(cv_1 + v_2)), \qquad
	z_2 = \Opg (c v_1+v_2 ),
	\end{aligned}
	\end{equation}
	where $\LL$ is given by \eqref{eq:L_N}. We obtain that the transformed system 
	\begin{equation}\label{eq:z}
	\begin{aligned}
	\partial_t z_1 =& ~{}c\partial_x (z_1+cz_2)+\LL \partial_x z_2 +M_1+M_{1,c},
	\\
	\partial_t z_2 =& ~{} \partial_x (z_1+cz_2)+M_2+M_{2,c},
	\end{aligned}
	\end{equation}
	where $\LL$ is the linearized operator obtained in \eqref{eq:L_N}, and the nonlinear terms $M_1$ and $M_2$ are given by
	\begin{equation}\label{eq:M}
	\begin{aligned}
	M_1:=&~{} 
	c\Opg  N(u_1)- \gar \Opg \left( 2\partial_x[\partial_x(f'(\Qc)) \partial_x z_2]- \partial_x^2(f'(\Qc)) \partial_x z_2\right)
	,
	\\
	M_2 :=&~{}\Opg  N(u_1), 
	\end{aligned}
	\end{equation}
	as well as, the modulation terms $M_{1,c}$ and $M_{2,c}$ are given by
		\begin{equation}\label{eq:Mc}
	\begin{aligned}
	M_{1,c}:=&~{}
	 c' \Opg ~{} (c\Qc+ f''(\Qc)\Lambda \Qc v_1+ 2cv_1-v_2)+(\rho'-c) \left( \partial_x z_1 -\Opg (\partial_x(f'(\Qc)) v_1) \right)
	,\\
	M_{2,c}:=&- c' \Opg (\Qc+v_1) +(\rho'-c)  \partial_x z_2 .
	\end{aligned}
	\end{equation}
Recall that $\psi_{A,B}=\chi_A^2\varphi_B$. Set now
\begin{equation}\label{eq:def_J}
\J=\int   \psi_{A,B}(y) z_1(t,y) z_2(t,y)dy ,\quad \hbox{and}\quad  \et_i =\chi_A \zB z_i, \quad i=1,2.
\end{equation}

\begin{prop}\label{prop:J} 
Under the conditions \eqref{eq:small} and  \eqref{eq:A}, there exists $C_2>0$ and $\delta_2>0$ such that for any $0<\delta\leq  \delta_2$, the following holds. 
Fix 
\begin{equation}\label{eq:B_gar}
B=A^{1/10}=\delta^{-1/10},\quad \gar=B^{-4}=\delta^{-2/5}.
\end{equation}
Then for all $t\geq 0$,
\begin{equation}\label{dt_J}
\begin{aligned}
\frac{d}{dt} \J \leq&
-\frac12  \int  \left( (\et_1+c \et_2)^2+3 (\partial_x \et_2)^2+\big(1-c^2\big)\et_2^2 -f'(\Qc) \et_2^2\right)
\\
&+ C_2 B^{-1} \left( \|\et_1+c\eta_2\|_{L^2}^2+  \| \et_2\|_{L^2}^2+\| \partial_x \et_2\|_{L^2}^2 \right)
+C_2 B^{-1} \big( \|w_1\|_{L^2}^2+\|\partial_x w_1\|_{L^2}^2+\|cw_1+w_2\|_{L^2}^2 \big)
\\&
+C_2 |c'| \big(\|  \et_2\|_{L^2}+ \|  \et_1+c\eta_2\|_{L^2} \big).
\end{aligned}
\end{equation}
\end{prop}
The proof of this result requires the following identity:
\begin{lem}
Let $(z_1,z_2)\in H^1\times H^2$ a solution to \eqref{eq:z}. Then
\begin{equation}\label{eq:dJ_lema}
\begin{aligned}
\frac{d}{dt} \J 
=&~{} -\frac12  \int  \psi_{A,B}' \left( (z_1+c z_2)^2+3 (\partial_x z_2)^2+(1-c^2 -f'(\Qc) )z_2^2\right)
\\&
		+\frac12 \int  \psi_{A,B} \partial_x(f'(\Qc))  z_2^2
		+\frac12 \int  \psi_{A,B}''' z_2^2	+ \int   \psi_{A,B} (M_1+M_{1,c}) z_2 + \int   \psi_{A,B} z_1 (M_2+M_{2,c}) ,
\end{aligned}
\end{equation}
here $M_1$, $M_{1,c}$,  $M_2$ and  $M_{2,c}$ are given in \eqref{eq:M} and \eqref{eq:Mc}, respectively, and the related terms to the nonlinearity, as well as, the modulation terms.
\end{lem}

\begin{proof}
See Appendix \ref{virial2}.
\end{proof}
The following identities were proved in \cite{Mau,MaMu}; they are useful in the next subsection.  First of all, by definition of $\psi_{A,B}$ it follows that
\begin{equation}\label{eq:psi_deriv}
\begin{aligned}
\psi_{A,B}' =&~{}\chi_A^2 \zB^2+(\chi_A^2)'\varphi_B,\quad \psi_{A,B}'' = \chi_A^2 (\zB^2)'+2(\chi_A^2)' \zB^2+(\chi_A^2)''\varphi_B,\\
\psi_{A,B}''' =& ~{}\chi_A^2 (\zB^2)''+3(\chi_A^2)' (\zB^2)'+3(\chi_A^2)'' \zB^2+(\chi_A^2)'''\varphi_B.
\end{aligned}
\end{equation}
Additionally, let $P\in W^{1,\infty}(\R)$, $P=P(y)$ and $\et_i$ be as in \eqref{eq:def_J}. Then 
		\begin{equation}\label{eq:P_CzB_zix}
	\begin{aligned}
	\int P\chi_A^2\zB^2(\partial_x z_i)^2 
	=&	\int P (\partial_x \et_i)^2 + \int   \left( P' \frac{\zB'}{\zB}+P\frac{\zB''}{\zB}\right) \et_i^2 +\int  \mathcal{E}_1(P)\zB^2 z_i^2,
	\end{aligned}
	\end{equation}
	\begin{equation}\label{eq:claimE1}
	\begin{aligned}
	\mathcal{E}_1(P)
	:= P  \left( \chi_A''\chi_A+(\chi_A^2)' \frac{\zB'}{\zB}\right)
+ \frac12 P'(\chi_A^2)' .
	\end{aligned}
	\end{equation}
	Note that
	\begin{equation}\label{cota_final_E1}
	|\mathcal{E}_1 (P) |\lesssim 
	A^{-1}\|P' \|_{L^\infty(A\leq|y|\leq2A)} 
	+ (AB)^{-1} \|P\|_{L^\infty (A\leq|y|\leq2A)}.
	\end{equation}
For further purposes, we need the following easy consequences: 
	for $P\equiv 1$, we get
	\begin{equation}\label{eq:CzB_zix}
	\begin{aligned}
	\int  \chi_A^2\zB^2(\partial_x z_i)^2 
	=&	\int  (\partial_x \et_i)^2 +
	\int     
	\frac{\zB''}{\zB} \et_i^2
	+\int  \mathcal{E}_1 \zB^2 z_i^2,
\qquad	\mathcal{E}_1 =\mathcal{E}_1(1) = \chi_A \bigg(\chi_A'' +2 \chi_A' \frac{\zB'}{\zB}\bigg).
	\end{aligned}
	 \end{equation}
	Additionally, from  \eqref{eq:CzB_zix}, \eqref{eq:z11} and \eqref{cota_final_E1}, one has the following estimate:
	\begin{equation}\label{eq:CZzx}
	\begin{aligned}
	\| \chi_A\zB \partial_x z_i\| ^2
	\lesssim & ~{}  \| \partial_x \et_i\|_{L^2}^2+B^{-1}\| \et_i\|_{L^2}^2
	+ (AB)^{-1}\| \zeta_{B}z_i\|_{L^2}^2.
	\end{aligned}
	\end{equation}
From \eqref{eq:dJ_lema}, define
\begin{equation}\label{defs_J}
\begin{aligned}
J_1  :=&~{} -\frac12  \int  \psi_{A,B}' \left[ (z_1+c z_2)^2+3 (\partial_x z_2)^2+(1-c^2 -f'(\Qc) )z_2^2\right] \\
J_2:= &~{} \frac12 \int  \psi_{A,B} \partial_x(f'(\Qc))  z_2^2+\frac12 \int  \psi_{A,B}''' z_2^2,	\\
J_3:= &~{} \int   \psi_{A,B} z_2 M_1 , \quad J_4:= \int   \psi_{A,B} z_2 M_{1,c}, \quad 
J_5:=  \int   \psi_{A,B} z_1 M_2, \quad J_6:= \int   \psi_{A,B} z_1 M_{2,c}.
\end{aligned}
\end{equation}

\subsection{Change of Variables} Recall \eqref{defs_J}. Now we will focus on $J_1+J_2$, which is the main part of the virial identity. Replacing \eqref{eq:psi_deriv} (see \cite{MaMu} for similar computations), we get the following separation
\[
\begin{aligned}
J_1+J_2
=&~{} -\frac12  \int  \chi_A^2 \zB^2  \left[  (z_1+c z_2)^2+3 (\partial_x z_2)^2+(1-c^2 -f'(\Qc) )z_2^2 \right]
\\&
+\frac12 \int  \frac{(\zB^2)''}{\zB^2}  \chi_A^2 \zB^2 z_2^2
		+\frac12 \int  ( 3(\chi_A^2)' (\zB^2)'+3(\chi_A^2)'' \zB^2+(\chi_A^2)'''\varphi_B) z_2^2
		\\&	
		-\frac12  \int  (\chi_A^2)'\vB  \left[ (z_1+c z_2)^2+3 (\partial_x z_2)^2+(1-c^2 -f'(\Qc) )z_2^2\right]
		+\frac12 \int  \chi_A^2 \vB \partial_x(f'(\Qc))  z_2^2.
\end{aligned}
\]
Now, applying the relation obtained in \eqref{eq:CzB_zix} and the definition of $\et_i$,  one gets 
\begin{equation}\label{eq:J1+J2}
\begin{aligned}
J_1+J_2
=&~{} -\frac12  \int  \bigg[ (\et_1+c \et_2)^2+3 (\partial_x \et_2)^2+\bigg(1-c^2 -f'(\Qc)-  \frac{\vB}{\zB^2} \partial_x(f'(\Qc)) \bigg)\et_2^2\bigg]	+J_{E}.
\end{aligned}
\end{equation}
where $J_{E}$ is related to the error terms appearing by the change of variable done, and is given by 
\begin{equation}\label{eq:JE}
\begin{aligned}
J_{E}
=&~{} \frac12 \int \bigg( 2 \bigg( \frac{\zB'}{\zB}\bigg)^2-\frac{\zB''}{\zB} \bigg)  \et_2^2
	-\frac32 \int \mathcal E_1 (1)  \zB^2 z_2^2
		+\frac12 \int  \left( 3(\chi_A^2)' (\zB^2)'+3(\chi_A^2)'' \zB^2+(\chi_A^2)'''\vB \right) z_2^2
		\\&	
		-\frac12  \int  (\chi_A^2)'\vB  \left[ (z_1+c z_2)^2+3 (\partial_x z_2)^2+(1-c^2 -f'(\Qc) )z_2^2\right]
		.
\end{aligned}
\end{equation}
In order to control the main part of the virial term, a lower bound for the potential $V$ is necessary. We have the following result:
\begin{lem}\label{lem:bad_term}
	Recall \eqref{eq:J1+J2}. One has
	 \[
	  -\frac12  \int \bigg(1-c^2 -f'(\Qc)-  \frac{\vB}{\zB^2} \partial_x(f'(\Qc) )\bigg)\et_2^2
	  \leq
	 -\frac12  \int \bigg(1-c^2 -f'(\Qc) \bigg)\et_2^2.
	 \]
\end{lem}
\begin{proof}
	Firstly,  we notice that $ \partial_x (f'(\Qc))=f''(\Qc) \Qc'<0$ for $y>0$. Moreover,
	  using that for $y\in [0,\infty)\mapsto \zeta_B(y)$ is a non-increasing function, we have for $y>0,$
	\[
	\dfrac{\vB}{\zB^2}=\dfrac{\int_{0}^y \zB^2 }{\zB^2} \geq y>0.
	\]
	 Then, from this estimate, and using a similar argument to the case $y\leq 0$, the proof follows.
	  This concludes the proof.
\end{proof}

\subsection{Technical estimates}
For $\gar>0$, let $(1-\gar \partial_x^2)^{-1}$
	be the bounded operator from $L^2$ to $H^2$ defined by its Fourier transform (denoted $\mathcal F$) as
	\[
	\mathcal F((1-\gar \partial_x^2)^{-1} g)(\xi)=\frac{\widehat{g}(\xi)}{1+\gar \xi^2},\quad \mbox{for any }  g\in L^2. 
	\]
	This operator satisfies:
\begin{lem}[\cite{KMM,KMMV20}]\label{lem:estimates_IOp}
	Let $f\in L^2(\R)$ and $0<\gar<1$ fixed. We have the following estimates:
	\vspace{0.1cm}
	\begin{enumerate}[$(i)$]
		\item $\| (1-\gar \partial_x^2)^{-1}f\|_{L^2(\R)}\leq \|f \|_{L^2(\R)}$,
		\vspace{0.1cm}
		\item $\| (1-\gar \partial_x^2)^{-1}\partial_x f\|_{L^2(\R)}\leq \gar^{-1/2}\|f \|_{L^2(\R)}$,
		\vspace{0.1cm}
		\item $\| (1-\gar \partial_x^2)^{-1}f\|_{H^2(\R)}\leq \gar^{-1}\|f \|_{L^2(\R)}$.
	\end{enumerate}
\end{lem}
We also enunciate the following results that appear in \cite{KMMV20,Mau}. Notice that now the variable in the weights is $y=x-\rho(t)$, but the shift does not affect the final outcome.
\begin{lem}
	There exist $\gar_1>0$ and $C>0$ such that for any $\gar\in (0,\gar_1)$, $0<K\leq 1$ and $g\in L^2$, 
	\vspace{0.1cm}
	\begin{align}
	\label{eq:sech_Opg}
	\left\| \sech\left( K y\right) (1-\gar\partial_x^2)^{-1} g\right\|_{L^2}
	\leq{}~ C &\left\| (1-\gar\partial_x^2)^{-1}\left[ \sech\left( Ky\right)  g\right]\right\|_{L^2},
	\\ \vspace{0.1cm}
	\label{eq:cosh_kink}
	\left\| \cosh\left( Ky\right) (1-\gar\partial_x^2)^{-1} \left[\sech\left( Ky\right) g\right]\right\|_{L^2}
	\leq& {}~ C \left\| (1-\gar\partial_x^2)^{-1} g\right\|_{L^2},\\
	\vspace{0.1cm}
\label{eq:sech_Opg_p}
	\left\| \sech(Ky)(1-\gar \partial_x^2)^{-1}\partial_x g \right\|_{L^2} 
	\leq&{}~ C \gar^{-1/2} \left\| \sech(Ky) g \right\|_{L^2},
	\end{align}
\begin{equation}\label{eq:sech_Opg_1pp}
	\left\| \sech(Ky)(1-\gar \partial_x^2)^{-1}(1-\partial_x^2)g \right\|_{L^2} \leq {}~ C \gar^{-1} \| \sech(Ky) g \|_{L^2}.
	\end{equation}
\end{lem}
Recall $z_1,z_2$ defined in \eqref{def_z}. The following lemma follows the spirit of the results presented on \cite{Mau,KMMV20,MaMu}, slightly different due to the new change of variable used in order to get the transformed system in the present work (see \eqref{eq:z_cv}).
\begin{lem}\label{lem:z_w}
Let  $1\leq K\leq A$ fixed. Then
	
	\begin{enumerate}[(a)]
		
		\item  Estimates on $z_1$ and $w_1$:
		\begin{equation}\label{eq:z1_w1}
		\begin{aligned}
		\|\zeta_K z_1\|\lesssim &~{}\gar^{-1} \|w_1 \|_{L^2}+\|c w_1+w_2\|_{L^2},\quad  \|\zeta_K \partial_x z_1\|\lesssim 
		\gar^{-1} \left(  \| w_1 \|_{L^2} + \|\partial_x w_1 \|_{L^2} \right)+ \gar^{-1/2} \|cw_1+w_2\|_{L^2}. 
		\end{aligned}
		\end{equation} 
		\item
		Estimates on $z_2$ and $w_2$:
		\begin{equation}\label{eq:z2_w2}
		\begin{gathered}
		\| \zeta_K z_2\|_{L^2}\lesssim  \| cw_1+w_2\|_{L^2}, \quad \quad 
		\|\zeta_K \partial_x z_2\|_{L^2} \lesssim \gar^{-1/2}\|cw_1+ w_2\|_{L^2},\quad \quad 
		\|\zeta_K \partial_x^2 z_2\|_{L^2}\lesssim \gar^{-1} \| cw_1+ w_2\|_{L^2}.
		\end{gathered}
		\end{equation}
		\item 
		Estimates on $v_1$ and $w_1$:
		\begin{equation}\label{eq:sech_v1_w1}
		\begin{aligned}
		\left\|\sech \left( \frac{y}{K}\right) v_1\right\|_{L^2} \lesssim &~{} \|w_1 \|_{L^2},\quad \left\|\sech \left( \frac{y}{K}\right) \partial_x v_1\right\|_{L^2}	\lesssim  \|\partial_x w_1 \|_{L^2}+ \| w_1 \|_{L^2}.
		\end{aligned}
		\end{equation} 
		\item 
		Estimates on $v_2+cv_1$ and $w_2+cw_1$: 
		\begin{equation}\label{eq:sech_v2_w2}
		\begin{gathered}
		\left\| \sech \left( \frac{y}{K}\right) (v_2+cv_1)\right\|_{L^2}\lesssim  \| w_2+c w_1\|_{L^2}.
		\end{gathered}
		\end{equation}
	\end{enumerate}
\end{lem}

For the purposes of this work, we also include a refined version of Lemma \ref{lem:z_w} $(a)$. 
\begin{lem} 
Let $1\leq K\leq A$ fixed. Then,
\begin{equation}\label{v1_mejorada}
\begin{aligned}
\|(1-\gar\partial_{x}^{2})^{-1}\LL g \|_{L^2}\lesssim &~{}\|g\|_{L^{2}} +\gar^{-1/2}\left\| \partial_{x} g\right\|_{L^{2}};\\
\|\zeta_{K} z_{1}\|_{L^{2}}\lesssim &~{}  \left(1+\frac{1}{K\gar^{1/2}}\right)  \left\|\zeta_{K} v_1\right\|_{L^{2}}
		  +\gar^{-1/2}\left\| \zeta_{K} \partial_{x} v_1\right\|_{L^{2}}+|c|\|\zeta_K(cv_1+v_2)\|_{L^2}.
\end{aligned}
\end{equation}
\end{lem}
\begin{proof}
Recalling \eqref{eq:L_N}, one gets
		\begin{equation*}
		\begin{aligned}
		\|(1-\gar\partial_{x}^{2})^{-1}\LL g\|_{L^{2}}
		=&~{} \left\|(1-\gar\partial_{x}^{2})^{-1}[(1-c^2-\partial_{x}^{2}) g-f'(\Qc) g]\right\|_{L^{2}} \lesssim \left\|(1-\gar\partial_{x}^{2})^{-1}(1-c^2-\partial_{x}^{2}) g\right\|_{L^{2}}+\left\|  f'(\Qc) g\right\|_{L^{2}}.
		\end{aligned}
		\end{equation*}
Using the the Plancherel's Theorem and that $|c|<1$, we get that the first term on the above inequality satisfies
		\[
		\begin{aligned}
		 & \left\|(1-\gar\partial_{x}^{2})^{-1}(1-c^2-\partial_{x}^{2}) g\right\|_{L^{2}}
		 \\
		  & ~{} = \left\| \frac{(1-c^2+\xi^{2})}{(1+\gar\xi^{2})} \widehat{g}\right\|_{L^{2}}
		  = \left\| \left(\frac{(1-c^2+\xi^{2})}{(1+\gar\xi^{2})} 1_{[-1,1]}+\frac{(1-c^2+\xi^{2})}{(1+\gar\xi^{2})} 1_{[-1,1]^{c}}\right)\widehat{g}\right\|_{L^{2}}\\
		    &~{} \leq \left\| \frac{2}{(1+\gar\xi^{2})} 1_{[-1,1]}\widehat{g}\right\|_{L^{2}}
		  +\left\|2\gar^{-1/2}\frac{\sgn(\xi)\sqrt{1+\gar|\xi|^{2}}}{(1+\gar\xi^{2})} 1_{[-1,1]^{c}} \xi\widehat{g}\right\|_{L^{2}}  \lesssim \left\|g\right\|_{L^{2}}
		  +\gar^{-1/2}\left\| \partial_{x} g\right\|_{L^{2}}.
		\end{aligned}
		\]
This concludes the proof of the first estimate in \eqref{v1_mejorada}. For the second inequality, since $\left\| \sech\left( K y\right) (1-\gamma\partial_x^2)^{-1} g\right\|_{L^2}
	\leq C \left\| (1-\gamma\partial_x^2)^{-1}\left[ \sech\left( K y\right)  g\right]\right\|_{L^2}$, using that $\zeta_{K}\LL(v_1)=\LL(\zeta_{K} v_1)+2\zeta_{K}'\partial_{x} v_1+\zeta_{K}'' v_1$, we obtain
		\[	
			\begin{aligned}
		\|\zeta_{K} z_{1}\|_{L^{2}}\lesssim& \| (1-\gar \partial_{x}^{2})^{-1} (\zeta_{K} \LL v_1)+c\zeta_{K}(c v_1+v_2)\|_{L^{2}}\\
		\lesssim& \left\| (1-\gar \partial_{x}^{2})^{-1}(\LL(\zeta_{K} v_1)+2\zeta_{K}'\partial_{x} v_1+\zeta_{K}'' v_1)\right\|_{L^{2}}+|c|\|\zeta_{K}(c v_1+v_2)\|_{L^{2}} \\
		\lesssim& \left\| (1-\gar \partial_{x}^{2})^{-1}\LL(\zeta_{K} v_1)\|_{L^{2}}+ \| 2\zeta_{K}'\partial_{x} v_1\right\|_{L^{2}}+ \| \zeta_{K}'' v_1\|_{L^{2}}+|c|\|\zeta_K (c v_1+v_2)\|.
		\end{aligned}
		\]
Therefore, from the first identity in \eqref{v1_mejorada}, \eqref{eq:z11} and \eqref{eq:z'/z},
		\[
		\begin{aligned}
		\|\zeta_{K} z_{1}\|_{L^{2}}
		\lesssim&~{}
		\left\|\zeta_{K} v_1 \right\|_{L^{2}}+\gar^{-1/2}\left\| \partial_{x} (\zeta_{K} v_1)\right\|_{L^{2}}+ \| 2\zeta_{K}'\partial_{x} v_1\|_{L^{2}}+\|f'(\Qc)\zeta_K v_1\|_{L^{2}}
		+ \| \zeta_{K}'' v_1\|_{L^{2}}+|c|\|\zeta_K (c v_1+v_2)\|_{L^{2}}.
		\end{aligned}
		\]
		This ends the proof.	
\end{proof}

\begin{rem}
In particular, for $K=A$ and using that $\zeta_{A}\partial_{x}v_1=\partial_{x} w_{1}-\frac{\zeta_{A}'}{\zeta_{A}} w_{1}$, we obtain
\begin{equation}\label{eq:zA_z1}
\begin{aligned}
		\|\zeta_{A} z_{1}\|_{L^{2}}
		\lesssim&~{}
		\left\| w_1 \right\|_{L^{2}}+\gar^{-1/2}\left\| \partial_{x} w_1\right\|_{L^{2}}+\frac{1}{A} \| \zeta_{A}'\partial_{x} v_1\|_{L^{2}}+\|w_1\|_{L^2}
		+ \| \zeta_{A}'' v_1\|_{L^{2}}+|c|\|c w_1+w_2\|_{L^2}
		\\
		\lesssim&~{} 
		\|w_1\|_{L^2}+\gar^{-1/2}\left\| \partial_{x} w_1\right\|_{L^{2}}
		+|c|\|c w_1+w_2\|_{L^2}.
\end{aligned}
\end{equation}
Additionally, using a similar procedure used on Lemma \ref{lem:var_v_w}, for $\zeta_K$ with $K\in[1,A]$, we obtain
\[
\begin{aligned}
	\int \zeta_{K}^{2} (\partial_x v_1)^2 
		=&\int \frac{\zeta_{K}^{2}}{\zeta_{A}^{2}}(\partial_{x} w_{1})^{2}
		+\int  w_{1}^{2} \left( 
						\frac{\zeta_{K}^{2}}{\zeta_{A}^{2}} \frac{\zeta_{A}''}{\zeta_{A}}
						+2 \frac{\zeta_{K}}{\zeta_{A}} \frac{\zeta_{A}'}{\zeta_{A}}\left(\frac{\zeta_{K}'}{\zeta_{A}}-\frac{\zeta_{K}}{\zeta_{A}}\frac{\zeta_{A}'}{\zeta_{A}} \right)
						\right).
\end{aligned}
\]
Hence, a crude estimate gives
\[
\|\zeta_{K} \partial_{x} v_{1}\|^{2}_{L^{2}} \lesssim \|\partial_{x} w_{1}\|_{L^{2}}^{2}+A^{-1}\|w_{1}\|^{2}_{L^{2}}.
\]
From \eqref{v1_mejorada}, and taking $K=A$ in the previous estimate, we obtain
\[ \label{eq:ineq_zK_u1x}
\|\zeta_{A} z_{1}\|_{L^{2}}\lesssim   \left(1+\frac{1}{(A\gar)^{1/2}}\right)  \left\| w_{1}\right\|_{L^{2}}
		  +\gar^{-1/2} \|\partial_{x} w_{1}\|_{L^{2}}+|c|\|c w_1+w_2\|_{L^2}, 
		  \quad
		   \|\zeta_{A} \partial_{x} v_{1}\|_{L^{2}} \lesssim \|\partial_{x} w_{1}\|_{L^{2}} +A^{-1/2}\|w_{1}\|_{L^{2}}.  
\] 
Finally, since $(A\gar)^{-1} = \delta^{3/5} \ll 1$,
\[ \label{eq:kv1_w1}
\begin{aligned}
\|\zeta_{A} z_{1}\|_{L^{2}}\lesssim &~{} \left\| w_{1}\right\|_{L^{2}}
		  +\gar^{-1/2} \|\partial_{x} w_{1}\|_{L^{2}}
		  +|c|\|c w_1+w_2\|_{L^2}.
  \end{aligned}
\] 
\end{rem}

\subsection{Control of the error term}
We split the term $J_E$ in \eqref{eq:JE} as follows: $J_E=J_{E,\et}+J_{E,z}+J_{E,zz}$, with
\begin{equation}\label{eq:JE_split}
\begin{aligned}
J_{E,\et}
=&~{} \frac12 \int \bigg( 2 \bigg( \frac{\zB'}{\zB}\bigg)^2-\frac{\zB''}{\zB} \bigg)  \et_2^2
	,\\
J_{E,z}
=&~{}	-\frac32 \int  \mathcal{E}_1 \zB^2 z_2^2
		+\frac12 \int  ( 3(\chi_A^2)' (\zB^2)'+3(\chi_A^2)'' \zB^2+(\chi_A^2)'''\vB) z_2^2
		,\\	
J_{E,zx}
=&~{}		-\frac12  \int  (\chi_A^2)'\vB  \left[ (z_1+c z_2)^2+3 (\partial_x z_2)^2+(1-c^2 -f'(\Qc) )z_2^2\right].
\end{aligned}
\end{equation}
We analyze each term in \eqref{eq:JE_split}. 

Bound on  $J_{E,\eta}$. By \eqref{eq:z11}, one gets
\begin{equation}\label{eq:JEet}
|J_{E,\et}| \lesssim  B^{-1}\|\et_2\|_{L^2}^2.
\end{equation}
On the other hand, for $J_{E,z}$, having in mind that $\mathcal{E}_1
	=\frac12  \chi_A''\chi_A+(\chi_A^2)' \frac{\zeta_B'}{\zeta_B}$, joint with \eqref{eq:z'/z} and \eqref{eq:vB_psi_prop}, we obtain 
\[ 
\begin{aligned}
|J_{E,z}|
\lesssim&~{}	\bigg|-\frac32 \int  \mathcal{E}_1 \zB^2 z_2^2
		+\frac12 \int  ( 3(\chi_A^2)' (\zB^2)'+3(\chi_A^2)'' \zB^2+(\chi_A^2)'''\vB) z_2^2
		\bigg|
		\\
\lesssim&~{}	(AB)^{-1} \|\zB z_2\|_{L^2}^2+ (AB)^{-1}\|z_2\|_{L^2(A<|x|<2A)}^2
\lesssim	(AB)^{-1} \left(\|\zB z_2\|_{L^2}^2+ \|\zA^2 z_2\|_{L^2}^2 \right).
\end{aligned}
\] 
After applying \eqref{eq:z2_w2}, we conclude that $J_{E,z}$ is bouded by
\begin{equation}\label{eq:JEz}
\begin{aligned}
|J_{E,z}|
\lesssim&~{}	(AB)^{-1} \|w_2+cw_1\|_{L^2}^2.
\end{aligned}
\end{equation}
Finally, for $J_{E,zx}$, applying \eqref{eq:vB_psi_prop} and \eqref{eq:chiA_zA}, we have
\[
\begin{aligned}
|J_{E,zx}|
\lesssim&~{}		A^{-1}B \left( \|z_1+c z_2\|_{L^2(|x|\leq A)}^2+\|\partial_x z_2\|_{L^2(|x|\leq A)}^2+\|z_2\|_{L^2(|x|\leq A)}^2\right)
\\
\lesssim&~{}		A^{-1}B \left(\|\zA(z_1+c z_2)\|_{L^2}^2+\|\zA \partial_x z_2\|_{L^2}^2+\|\zA z_2\|_{L^2}^2 \right)
.
\end{aligned}
\]
Using triangular inequality, \eqref{eq:z1_w1} and \eqref{eq:z2_w2}, finally we obtain
\begin{equation}\label{eq:JEzx}
\begin{aligned}
|J_{E,zx}|
\lesssim&~{} A^{-1}\gar^{-2} B( \|w_1\|_{L^2}^2+\|cw_1+ w_2\|_{L^2}^2).
\end{aligned}
\end{equation}
Then, by \eqref{eq:JEet}, \eqref{eq:JEz} and \eqref{eq:JEzx}, we conclude that
\begin{equation}\label{eq:JE_bound} 
\begin{aligned}
|J_{E}|
\lesssim&~{} 
B^{-1}\|\et_2\|_{L^2}^2
+(AB)^{-1} \|w_2+cw_1\|_{L^2}^2+A^{-1}\gar^{-2} B( \|w_1\|_{L^2}^2+\|cw_1+ w_2\|_{L^2}^2)
\\
\lesssim&~{} 
B^{-1}\|\et_2\|_{L^2}^2
+\big(A^{-1}\gar^{-2} B+(AB)^{-1}\big) \left( \|w_1\|_{L^2}^2+\|cw_1+ w_2\|_{L^2}^2\right).
\end{aligned}
\end{equation}

\subsection{Nonlinear and modulations terms} Notice that we have
\[
\begin{aligned}
 J_3= \int   \psi_{A,B} z_2 M_1 
 \quad \mbox{ and } \quad
J_5= \int   \psi_{A,B} z_1 M_2.
  \end{aligned}
  \]
We first deal with $J_3$. Replacing $M_1$ (see \eqref{eq:M}), we obtain
\[
\begin{aligned}
 \int   \psi_{A,B} z_2 M_1 
	=&~{}c \int   \psi_{A,B} z_2  \Opg  N(u_1) - \gar \int   \psi_{A,B} z_2 \Opg \left( 2\partial_x[\partial_x(f'(\Qc)) \partial_x z_2]- \partial_x^2(f'(\Qc)) \partial_x z_2\right)
\\	
	=&~{}c \int   \psi_{A,B} z_2  \Opg  N(u_1)
- 2\gar \int   \psi_{A,B} z_2 \Opg \partial_x[\partial_x(f'(\Qc)) \partial_x z_2]
	\\&+ \gar \int   \psi_{A,B} z_2 \Opg [\partial_x^2(f'(\Qc)) \partial_x z_2] =: J_{3,1}+J_{3,2}+J_{3,3}.
\end{aligned}
\]

For $J_{3,1}$, using H\"older's inequality and \eqref{eq:chiA_zA}, one gets
\[
\begin{aligned}
|J_{3,1}|
\lesssim &|c| \|  \chi_A \vB z_2 \|_{L^2} \|\chi_A \Opg  N(u_1)\|_{L^2} \lesssim  |c| \|  \chi_A \vB z_2 \|_{L^2} \|\zA^2 \Opg  N(u_1)\|_{L^2}
.
\end{aligned}
\]
Now, recalling \eqref{eq:L_N}; and applying \eqref{eq:vB_psi_prop} joint to \eqref{eq:sech_Opg_p}, we obtain
\[
\begin{aligned}
|J_{3,1}|
\lesssim &~{} |c| \|  \chi_A \vB z_2 \|_{L^2} \|\zA^2 \Opg  N(u_1)\|_{L^2}
\lesssim  |c| B\gar^{-1/2} \|  \zA^2  z_2 \left\|_{L^2} \|  \zA^2 (f(\Qc+v_1)-f(\Qc)-f'(\Qc)v_1)\right\|_{L^2}.
\end{aligned}
\]
Lastly, using \eqref{eq:taylor} and \eqref{eq:z2_w2}, we get
\begin{equation}\label{eq:J31}  
\begin{aligned}
|J_{3,1}|
\lesssim &~{}|c| B\gar^{-1/2} \|  c w_1+w_2 \|_{L^2} \|  (f''(\Qc)+|v_1|^{p-2})w_1^2\|_{L^2}
\lesssim 
|c| B\gar^{-1/2} \|u_1\|_{L^\infty} \|  c w_1+w_2 \|_{L^2} \| w_1\|_{L^2}.
\end{aligned}
\end{equation}
Now, for $J_{3,2}$, applying H\"olders inequality, \eqref{eq:vB_psi_prop} and \eqref{eq:chiA_zA}, we get 
\[\begin{aligned}
|J_{3,2}|
=
	&~{}  2\gar \left|\int    \chi_A ^2 \vB z_2 \Opg \partial_x[\partial_x(f'(\Qc)) \partial_x z_2]\right|
	\\
	\lesssim
	&~{}  \gar \| \chi_A \vB  z_2 \|_{L^2} \left\| \chi_A  \Opg \partial_x[\partial_x(f'(\Qc)) \partial_x z_2] \right\|_{L^2}
	\lesssim  \gar B \| \zA^2  z_2 \|_{L^2} \left\|\zA^2  \Opg \partial_x[\partial_x(f'(\Qc)) \partial_x z_2] \right\|_{L^2}
	.
\end{aligned}
\]
For the second factor, we apply \eqref{eq:sech_Opg_p}
\[\begin{aligned}
|J_{3,2}|
	\lesssim
	&~{}  \gar^{1/2} B \| \zA^2  z_2 \|_{L^2} \|\zA^2 \partial_x(f'(\Qc)) \partial_x z_2\|_{L^2}.
\end{aligned}
\]
To handle the second factor on the last expression, recall that $\partial_x (f'(\Qc))=\Qc^{p-2}|\Qc'|\sim \sech\left(\gamma (p-1) y\right) $. Since 
\[
\chi_A \zB \partial_x z_2=  \partial_x \et_2- \chi_A'\zB z_2 - \frac{\zB'}{\zB} \et_2,
\] 
 it follows that 
\begin{equation}\label{eq:sech_z2x}
\begin{aligned}
 \sech^2 \left(\gamma (p-1) y\right)  (\partial_x z_2)^2
 =&~{} \sech^2 \left(\gamma (p-1) y\right) \chi_A^2 (\partial_x z_2)^2+(1-\chi_A^2)\sech^2 \left(\gamma (p-1) y\right)  (\partial_x z_2)^2
 \\
 \lesssim&~{} 
 \sech^2 \left(\gamma (p-1) y\right)\zB^{-2}\bigg(  (\partial_x \et_2)^2+ (\chi_A'\zB)^2 z_2^2 + \left(\frac{\zB'}{\zB}\right)^2 \et_2^2 \bigg)
 \\&+\sech \left(\gamma (p-1) A\right) \sech \left(\gamma (p-1) y\right) (\partial_x z_2)^2
 .
\end{aligned}
\end{equation}
Then, by \eqref{eq:z2_w2}, we obtain the following inequality
\begin{equation}\label{eq:Q_z2x}
\begin{aligned}
 \|\zA \partial_x(f'(\Qc)) \partial_x z_2\|_{L^2}
 \lesssim &~{}   \| \et_2\|_{L^2}+\| \partial_x \et_2\|_{L^2}+\sech^{1/2} \left(\gamma (p-1) A\right) \gar^{-1/2} \| cw_1+w_2\|_{L^2},
\end{aligned}
\end{equation}
concluding that for $J_{3,2}$, after using \eqref{eq:z2_w2}, it holds 
\begin{equation}\label{eq:J32}  
\begin{aligned}
|J_{3,2}|
	\lesssim&~{} 
	 \gar^{1/2} B \| c w_1+w_2 \|_{L^2} \left(  \| \et_2\|_{L^2}+\| \partial_x \et_2\|_{L^2} \right)+\sech^{ 1/2}  \left(\gamma (p-1) A\right)  B \| c w_1+w_2 \|_{L^2}^2 .
\end{aligned}
\end{equation}
Lastly, for $J_{3,3}$, applying H\"older's inequality, \eqref{eq:vB_psi_prop} and \eqref{eq:chiA_zA}, we get 
\[\begin{aligned}
|J_{3,3}|
=&~{}	 \gar \left| \int   \psi_{A,B} z_2 \Opg \left[\partial_x^2(f'(\Qc)) \partial_x z_2\right]\right|
\\
\lesssim &~{}	 \gar \|   \chi_A  z_2\|_{L^2} \| \chi_A \vB\Opg [\partial_x^2(f'(\Qc)) \partial_x z_2]\|_{L^2}
\lesssim 	 \gar  B \|   \zA^2  z_2\|_{L^2} \left\| \zA^2 \Opg \left[\partial_x^2(f'(\Qc)) \partial_x z_2\right]\right\|_{L^2}.
\end{aligned}
\]
We will treat the second factor in a similar way used on $J_{3,2}$. Firstly, applying  \eqref{eq:sech_Opg}, and since $\partial_x^2(f'(\Qc))\sim \Qc^{p-1}|\Qc'|$, thanks to \eqref{eq:sech_z2x} we obtain
\[
\begin{aligned}
\| \partial_x^2(f'(\Qc)) \partial_x z_2\|_{L^2} 
\lesssim&~{}   \| \et_2\|_{L^2}+\| \partial_x \et_2\|_{L^2}+\sech^{1/2}  \left(\gamma (p-1) A\right) \gar^{-1/2} \| cw_1+w_2\|_{L^2}.
 \end{aligned}
\]
Then, by \eqref{eq:z2_w2}, we get
\begin{equation}\label{eq:J33}  
\begin{aligned}
|J_{3,3}|
\lesssim &~{}
 \gar  B \| cw_1+w_2\|_{L^2}  \left( \| \et_2\|_{L^2}+\| \partial_x \et_2\|_{L^2}\right)+\sech^{1/2}  \left(\gamma (p-1) A\right) \gar^{1/2}  B \| cw_1+w_2\|_{L^2}^2.
\end{aligned}
\end{equation}
Collecting \eqref{eq:J31} \eqref{eq:J32} and \eqref{eq:J33}, and using Cauchy--Schwarz inequality,  one gets that $J_3$ is bounded by
\[
\begin{aligned}
|J_3| \lesssim&~{} 
|c| B\gar^{-1/2} \|u_1\|_{L^\infty} \|  c w_1+w_2 \|_{L^2} \| w_1\|_{L^2}
+	 \gar^{1/2} B \| c w_1+w_2 \|_{L^2} \left(  \| \et_2\|_{L^2}+\| \partial_x \et_2\|_{L^2} \right)
 \\&+\sech^{ 1/2}  \left(\gamma (p-1) A\right)  B \| c w_1+w_2 \|_{L^2}^2
 +\gar  B \| cw_1+w_2\|_{L^2}  \left( \| \et_2\|_{L^2}+\| \partial_x \et_2\|_{L^2}\right)
 \\&+\sech^{ 1/2}  \left(\gamma (p-1) A\right) \gar^{1/2}  B \| cw_1+w_2\|_{L^2}^2.
 \end{aligned}
\]
We conclude
\begin{equation}\label{eq:J3} 
\begin{aligned}
|J_3| \lesssim&~{} \gar^{1/2} B \left(   \| \et_2\|_{L^2}^2+\| \partial_x \et_2\|_{L^2}^2 \right)+ \gar^{1/2} B \left( \| c w_1+w_2 \|_{L^2}^2+  \| w_1\|_{L^2}^2\right).
\end{aligned}
\end{equation}
Finally for $J_5$. Recalling $M_2$ in \eqref{eq:M}, similar as in $J_{3,1}$, by \eqref{eq:taylor} and using \eqref{eq:zA_z1} we obtain
\begin{equation}\label{eq:J5}
\begin{aligned}
|J_5| 
\lesssim &~{}|c| \|  \chi_A \vB z_1 \|_{L^2} \|\chi_A \Opg  N(u_1)\|_{L^2}
\\
\lesssim &~{}|c| B \gar^{-1} \|u_1\|_{L^\infty}\left( \|w_1\|_{L^2}+\|\partial_x w_1\|+|c|\|cw_1+w_2\|_{L^2} \right)  \| w_1\|_{L^2}
\\
\lesssim &~{} B \gar^{-1} \|u_1\|_{L^\infty} \left( \|w_1\|_{L^2}^2+\|\partial_x w_1\|_{L^2}^2+\|cw_1+w_2\|_{L^2}^2 \right).
\end{aligned}
\end{equation}

\subsection*{Modulations terms} In this case we have
\[
\begin{aligned}
  J_4=\int   \psi_{A,B}  z_2 M_{1,c}
  \quad \mbox{ and }\quad
 J_6= \int   \psi_{A,B} z_1 M_{2,c}
  \end{aligned}
\]
We focus on $J_4$. Replacing $M_{1,c}$ (see \eqref{eq:Mc}), we obtain
\[
\begin{aligned}
  J_4=&~{}\int   \psi_{A,B}  z_2 \left( 
  						c' \Opg(c\Qc+ f''(\Qc)\Lambda \Qc v_1 + 2cv_1-v_2) +(\rho'-c) \left( \partial_x z_1 -\Opg [\partial_x(f'(\Qc)) v_1] \right)
						\right)
\\
	=&~{}c'  \int   \psi_{A,B}  z_2 \Opg (c\Qc)
	+ c'  \int  \psi_{A,B}  z_2 \Opg [f''(\Qc)\Lambda \Qc v_1+2cv_1-v_2]
	\\&+(\rho'-c) \int   \psi_{A,B}  z_2  \partial_x z_1 
-(\rho'-c) \int   \psi_{A,B}  z_2  \Opg [\partial_x(f'(\Qc)) v_1]
	=: \sum_{j=1}^4 J_{4,j}.
\end{aligned}
\]

Now, we focus on each of the above terms. First, letting $\phi(y)=\sech\left(\frac{\gamma}{2}y \right)$ and using H\"older's inequality, one gets
\begin{equation}\label{eq:inicio_J41}
\begin{aligned}
|J_{4,1}|	=&~{} \left|c'  \int   \psi_{A,B}  z_2 \Opg(c\Qc)\right|
\\
=&~{} \left|c'  \int   \chi_A^2 \vB \phi  z_2  \phi^{-1}\Opg(c\Qc)\right|
\lesssim |c'| \|  \chi_A \vB \phi  z_2\|_{L^2} \|\chi_A  \phi^{-1}\Opg (c \Qc)\|_{L^2}.
\end{aligned}
\end{equation}
Using a similar inequality to \eqref{eq:vA_Qc} on the first factor, and \eqref{eq:cosh_kink} on the second factor, we obtain
\begin{equation}\label{eq:J41}
\begin{aligned}
|J_{4,1}|	\lesssim &~{}  |c'| \|  \et_2\|_{L^2} .
\end{aligned}
\end{equation}
We handle $J_{4,2}$ applying H\"older's inequality and  by \eqref{eq:chiA_zA}, one gets
\[
\begin{aligned}
|J_{4,2}| =&~{}
	|c'| \left| \int  \psi_{A,B}  z_2 \Opg (f''(\Qc)\Lambda \Qc v_1 + 2cv_1 -v_2)\right|
	\\
	\lesssim&~{} 	|c'| \|   \chi_A \vB  z_2\|_{L^2} \left\| \zA^2 \Opg (f''(\Qc)\Lambda \Qc v_1 +2cv_1 -v_2)\right\|_{L^2}
	.
\end{aligned}
\]
Lastly, applying  \eqref{eq:vB_psi_prop}, \eqref{eq:z2_w2}  and \eqref{eq:sech_Opg}, we conclude that $J_{4,2}$ is bounded by
\begin{equation}\label{eq:J42} 
\begin{aligned}
|J_{4,2}| 
	\lesssim&~{} 	|c'| B \|   \zA^2   z_2\|_{L^2} \left\|\zA^2 \Opg [f''(\Qc)\Lambda \Qc v_1 + 2cv_1-v_2]\right\|_{L^2}
	\\
	\lesssim&~{} 	|c'| B \|   cw_1+w_2\|_{L^2} \left\| f''(\Qc)\Lambda \Qc w_1 + 2cw_1-w_2\right\|_{L^2}
		\\
	\lesssim&~{} 	|c'| B \|   cw_1+w_2\|_{L^2} \big( \frac{c}{\gamma^2}\left\| w_1\right\|_{L^2}+\|cw_1+ w_2 \|_{L^2}\big)
	\\
	\lesssim&~{} 	|c'| B^{3/2} \|   cw_1+w_2\|_{L^2} \bigg( \left\| w_1\right\|_{L^2}+\|cw_1+ w_2 \|_{L^2}\bigg)
	.
\end{aligned}
\end{equation}
since $B^{1/4}\gamma>2$.
For $J_{4,4}$, we are going to follow a similar strategy to $J_{4,1}$. Let us consider
\[
\begin{aligned}
|J_{4,4}|
=&~{} |\rho'-c| \left| \int   \psi_{A,B}  z_2  \Opg [\partial_x(f'(\Qc)) v_1]\right|
\\
=&~{} |\rho'-c| \left| \int   \chi_A \vB \phi  z_2 \phi^{-1} \Opg [\partial_x(f'(\Qc)) v_1]\right| \\
\lesssim 	&~{}|\rho'-c| \|   \chi_A \vB \phi z_2 \|_{L^2} \| \zA^2 \phi^{-1} \Opg [\partial_x(f'(\Qc)) v_1]\|_{L^2}.
\end{aligned}
\]
Applying \eqref{eq:vA_Qc} and \eqref{eq:cosh_kink}, one concludes
\begin{equation}\label{eq:J44}  
\begin{aligned}
|J_{4,4}|
\lesssim& ~{} 	|\rho'-c| \|   \et_2 \|_{L^2} \|  w_1\|_{L^2}.
\end{aligned}
\end{equation}
We will treat $J_{4,3}$ jointly to some terms that belong to $J_6$.

\medskip

Now, for $J_6$, we have the following decomposition
\[
\begin{aligned}
 J_6=&~{} \int   \psi_{A,B} z_1 \big( 
 						- c' \Opg (\Qc+v_1) +(\rho'-c)  \partial_x z_2
						\big)
\\
	=&~{} 
 		- c'  \int   \psi_{A,B} z_1\Opg (\Qc+v_1) 
		+(\rho'-c)  \int   \psi_{A,B} z_1 \partial_x z_2
	=: J_{6,1}+J_{6,2}
	.
  \end{aligned}
\]
Let us focus on the first terms on the above expression. 
Using the same approach applied  in \eqref{eq:inicio_J41} for $J_{4,1}$; and using \eqref{eq:chiA_zA} joint to \eqref{eq:vB_psi_prop}, one gets
\[
\begin{aligned}
|J_{6,1}|=&~{} |c'| \left| \int   \psi_{A,B} z_1\Opg (\Qc+v_1) \right|
\\
\lesssim&~{} |c'| \left(\|  \et_1\|_{L^2} +\|  \chi_A \vB  z_1\|_{L^2} \|\chi_A \Opg v_1 \|_{L^2}\right)
\lesssim |c'| \left(\|  \et_1\|_{L^2} +B\|  \zA^2  z_1\|_{L^2} \|\zA^2 \Opg v_1 \|_{L^2}\right).
\end{aligned}
\]
Applying \eqref{eq:sech_Opg} and \eqref{eq:zA_z1}, we conclude
\begin{equation*}\label{eq:J61}
\begin{aligned}
|J_{6,1}|
\lesssim&~{} |c'| \left(\|  \et_1\|_{L^2} +B\|  \zA^2  z_1\|_{L^2} \|\zA^2 \Opg v_1 \|_{L^2}\right)
\\
\lesssim&~{} |c'| \left( \|  \et_1\|_{L^2} +B\gar^{-1/2}\|w_1 \|_{L^2} \big( \|w_1\|_{L^2}+\|\partial_x w_1\|+|c|\|cw_1+w_2\|_{L^2} \big)\right)
\\
\lesssim&~{} |c'| \left( \|  \et_1+c\eta_2\|_{L^2} +\|\eta_2\|_{L^2}+B\gar^{-1/2}\|w_1 \|_{L^2} \big( \|w_1\|_{L^2}+\|\partial_x w_1\|+|c|\|cw_1+w_2\|_{L^2} \big)\right).
\end{aligned}
\end{equation*}
Lastly, consider $J_{4,3}$ joint with $J_{6,2}$, integratig by parts and replacing \eqref{eq:psi_deriv}, we get
\[
\begin{aligned}
J_{4,3}+J_{6,2}
=&~{} -(\rho'-c)  \int  \left( \chi_A^2 \zB^2+(\chi_A^2)'\varphi_B \right) z_1 z_2
.
\end{aligned}
\] 
Then, by \eqref{eq:zA_z1} and \eqref{eq:z2_w2}, we obtain
\begin{equation}\label{eq:J43+J62}
\begin{aligned}
|J_{4,3}+J_{6,2}|
\lesssim~{}& |\rho'-c|  \|\et_1\|_{L^2} \|\et_2\|_{L^2}
+|\rho'-c| B\gar^{-1/2}A^{-1}\big(\|w_1\|_{L^2}+\|\partial_x w_1\|_{L^2}+\|cw_1+w_2\|_{L^2}\big) \|cw_1+w_2\|_{L^2}
\\
\lesssim~{}& |\rho'-c| \left( \|\et_1+c\eta_2\|_{L^2} \|\et_2\|_{L^2}+ \|\eta_2\|^2_{L^2}\right)
\\
&~{} +|\rho'-c| B\gar^{-1/2}A^{-1}\big(\|w_1\|_{L^2}+\|\partial_x w_1\|_{L^2}+\|cw_1+w_2\|_{L^2}\big) \|cw_1+w_2\|_{L^2}
.
\end{aligned}
\end{equation}
Finally, collecting  \eqref{eq:J41}, \eqref{eq:J42}, \eqref{eq:J44}, \eqref{eq:J61} ,\eqref{eq:J43+J62}, and recalling \eqref{eq:B_gar},
we obtain
\[
\begin{aligned}
|J_4+J_6|
\lesssim &~{} 
|c'| \|  \et_2\|_{L^2}
+|c'| B^{3/2} \|   cw_1+w_2\|_{L^2} \big( \left\| w_1\right\|_{L^2}+\|cw_1+ w_2 \|_{L^2}\big)
+|\rho'-c| \|   \et_2 \|_{L^2} \|  w_1\|_{L^2}
\\&
+|c'| \left( \|  \et_1+c\eta_2\|_{L^2} +\|\eta_2\|_{L^2}+B\gar^{-1/2}\|w_1 \|_{L^2} \big( \|w_1\|_{L^2}+\|\partial_x w_1\|+|c|\|cw_1+w_2\|_{L^2} \big)\right)
\\&
+ |\rho'-c| \left( \|\et_1+c\eta_2\|_{L^2} \|\et_2\|_{L^2}+ \|\eta_2\|^2_{L^2}\right)
\\
&+|\rho'-c| B\gar^{-1/2}A^{-1}\big(\|w_1\|_{L^2}+\|\partial_x w_1\|_{L^2}+\|cw_1+w_2\|_{L^2}\big) \|cw_1+w_2\|_{L^2}
\\
\lesssim&~{}  
|c'| \|  \et_2\|_{L^2}
+|c'| B^{3/2} \|   cw_1+w_2\|_{L^2} \big( \left\| w_1\right\|_{L^2}+\|cw_1+ w_2 \|_{L^2}\big)
+|\rho'-c| \|   \et_2 \|_{L^2} \|  w_1\|_{L^2}
\\&
+|c'| \left( \|  \et_1+c\eta_2\|_{L^2} +\|\eta_2\|_{L^2}+B^3\|w_1 \|_{L^2} \big( \|w_1\|_{L^2}+\|\partial_x w_1\|+|c|\|cw_1+w_2\|_{L^2} \big)\right)
\\&
+ |\rho'-c| \left( \|\et_1+c\eta_2\|_{L^2} \|\et_2\|_{L^2}+ \|\eta_2\|^2_{L^2}\right)\\
&
+|\rho'-c| B^{-7}\big(\|w_1\|_{L^2}+\|\partial_x w_1\|_{L^2}+\|cw_1+w_2\|_{L^2}\big) \|cw_1+w_2\|_{L^2}.
\end{aligned}
\]
Using Cauchy-Schwarz inequality jointly to $|c'|\leq \delta^2$ and $|\rho'-c|\leq \delta$; and rearranging terms, one gets
\begin{equation}\label{eq:J4+J6} 
\begin{aligned}
|J_4+J_6|
\lesssim&~{}  
|c'| \|  \et_2\|_{L^2}
+|c'|  B^{3/2}  \big( \left\| w_1\right\|_{L^2}^2+\|cw_1+ w_2 \|_{L^2}^2\big)
+|\rho'-c| \big(\|   \et_2 \|_{L^2}^2+ \|  w_1\|_{L^2}^2\big)
\\&
+|c'| \left( \|  \et_1+c\eta_2\|_{L^2} +\|\eta_2\|_{L^2}+B^3 \big( \|w_1\|_{L^2}^2+\|\partial_x w_1\|_{L^2}^2+\|cw_1+w_2\|_{L^2}^2 \big)\right)
\\&
+ |\rho'-c| \left( \|\et_1+c\eta_2\|_{L^2}^2 + \|\eta_2\|^2_{L^2}\right)
+|\rho'-c| B^{-7}\big(\|w_1\|_{L^2}^2+\|\partial_x w_1\|_{L^2}^2+\|cw_1+w_2\|_{L^2}^2\big)
\\
\lesssim&~{}   
|c'| \big(\|  \et_2\|_{L^2}+ \|  \et_1+c\eta_2\|_{L^2} \big)
+|c'|  B^{3/2}  \big( \left\| w_1\right\|_{L^2}^2+\|cw_1+ w_2 \|_{L^2}^2\big)
+|\rho'-c| \big(\|   \et_2 \|_{L^2}^2+ \|  w_1\|_{L^2}^2\big)
\\&
+|c'| B^3 \big( \|w_1\|_{L^2}^2+\|\partial_x w_1\|_{L^2}^2+\|cw_1+w_2\|_{L^2}^2 \big)
\\&
+ |\rho'-c| \left( \|\et_1+c\eta_2\|_{L^2}^2 + \|\eta_2\|^2_{L^2}\right)
+|\rho'-c| B^{-7}\big(\|w_1\|_{L^2}^2+\|\partial_x w_1\|_{L^2}^2+\|cw_1+w_2\|_{L^2}^2\big)
\\
\lesssim&~{} 
|c'| \big(\|  \et_2\|_{L^2}+ \|  \et_1+c\eta_2\|_{L^2} \big)
\\
&~{} +A^{-1} \big( \|w_1\|_{L^2}^2+\|\partial_x w_1\|_{L^2}^2+\|cw_1+w_2\|_{L^2}^2 \big)
+ A^{-1} \left( \|\et_1+c\eta_2\|_{L^2}^2 + \|\eta_2\|^2_{L^2}\right)
.
\end{aligned}
\end{equation}

\subsection{End proof Proposition \ref{prop:J}}
Collecting \eqref{eq:JE_bound}, \eqref{eq:J3} \eqref{eq:J5} and \eqref{eq:J4+J6}, and replacing \eqref{eq:B_gar}, we obtain
\[
\begin{aligned}
\frac{d}{dt} \J 
 \leq&
-\frac12  \int  \bigg[ (\et_1+c \et_2)^2+3 (\partial_x \et_2)^2+\bigg(1-c^2 -f'(\Qc)-  \frac{\vB}{\zB^2} \partial_x(f'(\Qc)) \bigg)\et_2^2\bigg]
\\
&+ C_2 B^{-1} \left( \|\et_1+c\eta_2\|_{L^2}^2+  \| \et_2\|_{L^2}^2+\| \partial_x \et_2\|_{L^2}^2 \right)
\\
&+C_2 B^{-1} \big( \|w_1\|_{L^2}^2+\|\partial_x w_1\|_{L^2}^2+\|cw_1+w_2\|_{L^2}^2 \big)
 \\
 &+ C_2 B^{5} \|u_1\|_{L^\infty} \left( \|w_1\|_{L^2}^2+\|\partial_x w_1\|_{L^2}^2+\|cw_1+w_2\|_{L^2}^2 \right)
\\&
+C_2 |c'| \big(\|  \et_2\|_{L^2}+ \|  \et_1+c\eta_2\|_{L^2} \big)
.
\end{aligned}
\]
Recalling that $\|u_1\|_{L^\infty}\leq \delta$ and \eqref{eq:A}; moreover by Lemma \ref{lem:bad_term}, one gets
\[
\begin{aligned}
\frac{d}{dt} \J \leq&
-\frac12  \int  \left( (\et_1+c \et_2)^2+3 (\partial_x \et_2)^2+\big(1-c^2\big)\et_2^2 -f'(\Qc) \et_2^2\right)
\\
&+ C_2 B^{-1} \left( \|\et_1+c\eta_2\|_{L^2}^2+  \| \et_2\|_{L^2}^2+\| \partial_x \et_2\|_{L^2}^2 \right)
+C_2 B^{-1} \big( \|w_1\|_{L^2}^2+\|\partial_x w_1\|_{L^2}^2+\|cw_1+w_2\|_{L^2}^2 \big)
\\&
+C_2 |c'| \big(\|  \et_2\|_{L^2}+ \|  \et_1+c\eta_2\|_{L^2} \big)
\end{aligned}
\] 
for some fix constant $C_2>0$. This concludes the proof of \eqref{dt_J}.

\section{Improved estimates}\label{Sec:5}

Estimate \eqref{dt_J} has an additional derivative in the variable $\eta_2$, which needs to be recovered using a new virial estimate. Recall that $\psi_{A,B}=\chi_A^2\varphi_B$. Set now
\begin{equation}\label{eq:virial_N}
\N=\int   \psi_{A,B}(y) \partial_x z_1(t,y) \partial_x z_2(t,y) dy ,
\end{equation}

\begin{prop}\label{prop:N}
Under the conditions \eqref{eq:small}, \eqref{eq:scales},  \eqref{eq:A} and \eqref{eq:B_gar}. 
There exists $C_3>0$ and $\delta_3>0$ such that for any $0<\delta\leq  \delta_3$, the following holds
\begin{equation}\label{Est_N}
\begin{aligned}
	\dfrac{d}{dt} \N
\leq~{}&
	-\frac14 \int \left( 
				(\partial_x \et_1+c\partial_x \et_2)^2 
				+3(\partial_x^2 \et_2)^2
			 \right)
+C_3 \big( \| \et_2\|_{L^2}^2+\| \partial_x \et_2\|_{L^2}^2\big)
\\&
+	C_3 B^{-1} \| \et_1+c\eta_2\|_{L^2}^2
	+ C_3B^{-1}\big( 
					\|  w_1\|_{L^2}^2 
					+\|  \partial_x w_1\|_{L^2}^2
					+ \|cw_1+w_2 \|_{L^2}^2
		\big)
	\\&
	+C_3| c'| B^{11/2} (\|w_1\|_{L^2}+\|\partial_x w_1\|_{L^2}+\|cw_1+w_2\|_{L^2}).
\end{aligned}
\end{equation}
\end{prop}
As in previous virial estimates, estimate \eqref{Est_N} will be consequence of a related identity.
\begin{lem}
Let $(z_1,z_2)\in H^1\times H^2$ a solution to \eqref{eq:z}. Then
\begin{equation}\label{eq:dt_N}
\begin{aligned}
\frac{d}{dt}\N 	=&-\frac12 \int   \psi_{A,B}' \left[(\partial_x z_1+c\partial_x z_2)^2+ 3(\partial_x^2 z_2)^2
+(1 -c^2)(\partial_x  z_2)^2 -f'(\Qc)(\partial_x  z_2)^2\right] 
	\\&
	-\frac12 \int   \psi_{A,B} \partial_x(f'(\Qc)) (\partial_x z_2)^2
	+\frac12 \int   \psi_{A,B}''' (\partial_x  z_2)^2  
	\\
		& +\int   \psi_{A,B} \partial_x z_2\partial_x (M_1+M_{1,c})  
	+ \int   \psi_{A,B} \partial_x z_1 \partial_x( M_2+M_{2,c}),
\end{aligned}
\end{equation}
where $M_1$, $M_{1,c}$,  $M_2$ and  $M_{2,c}$ are given in \eqref{eq:M} and \eqref{eq:Mc}, respectively.
\end{lem}
\begin{proof}
See Appendix \ref{virial2}.
\end{proof}
We shall decompose \eqref{eq:dt_N} as follows:
\begin{equation}\label{eq:dtN_N}
\begin{aligned}
	\dfrac{d}{dt} \N
	=&-\frac12 \int   \psi_{A,B}' \left[ (\partial_x z_1+c\partial_x z_2)^2+3 (\partial_x^2 z_2)^2 \right]
	-\frac12 \int   \psi_{A,B}'  (1 -c^2-f'(\Qc))(\partial_x  z_2)^2
	\\&
	-\frac12 \int   \psi_{A,B} \partial_x(f'(\Qc)) (\partial_x z_2)^2
	+\frac12 \int   \psi_{A,B}''' (\partial_x  z_2)^2  
+\int   \psi_{A,B} \partial_x z_2\partial_x M_1
	 +\int   \psi_{A,B} \partial_x z_2\partial_x M_{1,c}
	\\&
	+ \int   \psi_{A,B} \partial_x z_1 \partial_x M_2
	+ \int   \psi_{A,B} \partial_x z_1 \partial_x M_{2,c}  =:  \sum_{j=1}^8 N_{j}.
\end{aligned}
\end{equation}
In the next subsection, we shall estimate each term in \eqref{eq:dtN_N}. 

\subsection{Second set of identities related to the change of variable} Recall the following technical estimates on the variables $\zB$ and other related error terms. These estimates had been proved and used in \cite{Mau,MaMu},  therefore we only enunciate the main results.

\begin{claim}
	Let  $R$ be a $W^{2,\infty}(\R)$ function, $R=R(y)$, $z_i$ be as in \eqref{eq:z_cv}, and $\et_i$ be as in \eqref{eq:def_J}. Then
	\begin{equation*}
	\begin{aligned}
	\int R \chi_A^2\zB^2 (\partial_x^2 z_i)^2
	=&
	\int R(\partial_x^2 \et_i)^2+\int \tilde{R} \et_i^2
	+\int P_R (\partial_x \et_i)^2 +
	\int    \left( P_R' \frac{\zB'}{\zB}+P_R
	\frac{\zB''}{\zB}\right) \et_i^2 +\int \mathcal{E}_2 (R) \zB^2z_i^2\\
	&+\int  \mathcal{E}_1(P_R)\zB^2 z_i^2
	+\int \mathcal{E}_3 (R) \zB^2(\partial_x z_i)^2,
	\end{aligned}
	\end{equation*}
	where
	\begin{equation}\label{eq:claimtildeR}	
	\begin{aligned}
	\tilde{R}=\tilde{R}_R=&	  
-2R\left( \frac{\zB^{(4)}}{\zB}+\frac{\zB'''}{\zB}\frac{\zB'}{\zB}\right) -2	R'\frac{\zB'''}{\zB}
- R''\frac{\zB''}{\zB},
	\end{aligned}
	\end{equation}
	\begin{equation}\label{eq:claimPR}
	P_R= R \bigg(4  \frac{\zB''}{ \zB}  -2 \left(\frac{\zB'}{\zB} \right)^2\bigg)+2R' \frac{\zB'}{\zB},
	\end{equation}
	$\mathcal{E}_1$ is defined in \eqref{eq:claimE1}, 
\begin{equation}\label{eq:claimE2}	
\begin{aligned}
\mathcal{E}_2(R) =& -R \left(\chi_A^{(4)} \chi_A  +4\chi_A''' \chi_A \frac{\zB'}{\zB^2} +6 \chi_A''\chi_A\frac{\zB''}{\zB} +2(\chi_A^2)' \frac{\zB'''}{\zB}  \right)\\
& -R' \left(2\chi_A''' \chi_A +6\chi_A'' \chi_A  \frac{\zB'}{\zB} +6\chi_A' \chi_A \frac{\zB''}{\zB}\right) - R'' \left(\chi_A'' \chi_A +\frac12 (\chi_A^2)' \frac{\zB'}{\zB}\right), 
\end{aligned}
\end{equation}
and
\begin{equation}\label{eq:claimE3}
\begin{aligned}
\mathcal{E}_3(R)
= R\bigg( 
4\chi_A'' \chi_A -2(\chi_A' )^2+ 2\frac{\zB'}{\zB}(\chi_A^2)'  
\bigg)
+R'(\chi_A^2)' .
	\end{aligned}
	\end{equation}
Finally, $P_R$, $\mathcal{E}_2$ and $\mathcal{E}_3$ satisfy the following bounds
\begin{equation}\label{eq:PR_E23}
\begin{aligned}
	|P_R| \lesssim & ~{}  		B^{-1}\|R' \|_{L^{\infty}}  + B^{-1}\| R\|_{L^\infty},   \\
	|P_R'| \lesssim & ~{}  	B^{-1}\|R'' \|_{L^{\infty}}	+B^{-1}\|R' \|_{L^{\infty}}  + B^{-1}\| R\|_{L^\infty},   \\
	|\mathcal{E}_2| \lesssim & ~{} (AB)^{-1}\|R'' \|_{L^{\infty}(A\leq |y|\leq 2A)}+(AB^2)^{-1}\|R' \|_{L^{\infty}(A\leq |y|\leq 2A)} +(AB^3)^{-1}\|R \|_{L^{\infty}(A\leq |y|\leq 2A)},\\
	|\mathcal{E}_3| \lesssim & ~{}  A^{-1}\|R' \|_{L^{\infty}(A\leq |y|\leq 2A)}+(AB)^{-1}\|R \|_{L^{\infty}(A\leq |y|\leq 2A)}.
\end{aligned}
\end{equation}
\end{claim}

\begin{rem}
Estimates \eqref{eq:PR_E23} are of technical type. A proof of these estimates is present in \cite{Mau}, but without the localization terms. In \cite{MaMu0} we slightly improve these estimates by considering the region of space where these functions are supported.
\end{rem}

\begin{rem}
	For further purposes, we shall need the previous identities in the simplest case $R=1$.  We obtain
		\begin{equation}\label{eq:CZ_zxx}
	\begin{aligned}
	\int \chi_A^2\zB^2 (\partial_x^2 z_i)^2
	=&
	\int (\partial_x^2 \et_i)^2+\int \widetilde{R}_1\et_i^2
	+\int P_1 (\partial_x \et_i)^2 
	+
	\int    \left( P_1' \frac{\zB'}{\zB}+P_1 
	\frac{\zB''}{\zB}\right) \et_i^2 \\
	&+\int \mathcal{E}_2 \zB^2z_i^2
	+\int  \mathcal{E}_1(P_1)\zB^2 z_i^2
	+\int \mathcal{E}_3 \zB^2(\partial_x z_i)^2,
	\end{aligned}
	\end{equation}
	where (see \eqref{eq:claimtildeR} and \eqref{eq:claimPR}),
	\begin{equation}\label{eq:tildeR1_P1}	
		\widetilde{R}_1=	  
			-2\left(\frac{\zB^{(4)}}{\zB}+\frac{\zB'''}{\zB}\frac{\zB'}{\zB}\right)
	,\ \quad 
	P_1=
	4  \frac{\zB''}{ \zB}  
	-2 \left(\frac{\zB'}{\zB} \right)^2,
		\end{equation}
	$\mathcal{E}_1$ is defined in \eqref{eq:claimE1}, $\mathcal E_2$ in \eqref{eq:claimE2} becomes
\begin{equation}\label{eq:claimE2_1}	
\begin{aligned}
\mathcal{E}_2 =\mathcal{E}_2(1)
=&
- \chi_A
\left(
\chi_A^{(4)} 
+4\chi_A'''  \frac{\zB'}{\zB^2}
+6 \chi_A'' \frac{\zB''}{\zB}
+4\chi_A' \frac{\zB'''}{\zB}  
\right),
	\end{aligned}
	\end{equation}
	and \eqref{eq:claimE3} reads now
	\begin{equation}\label{eq:claimE3_1}
	\begin{aligned}
\mathcal{E}_3 =	\mathcal{E}_3 (1)
=
4\chi_A'' \chi_A -2(\chi_A' )^2+ 2\frac{\zB'}{\zB}(\chi_A^2)'  .
	\end{aligned}
	\end{equation}
	Finally, by \eqref{eq:z'/z} and \eqref{eq:C*z'/z}, we obtain the simplified estimate
	\begin{equation}\label{eq:CZzxx}
	\begin{aligned}
	\| \chi_A\zB \partial_x^2 z_i\| 
	\lesssim & ~{}  \| \partial_x^2 \et_i\|_{L^2}^2+B^{-1}\| \partial_x \et_i\|_{L^2}^2+B^{-1}\| \et_i\|_{L^2}^2 
	+(AB)^{-1} \left( \| \zeta_{B}  z_i\|_{L^2}^2 + \| \zeta_{B} \partial_{x} z_i\|_{L^2}^2 \right).
	\end{aligned}
	\end{equation}
\end{rem}

\subsection{Main part of third virial}	
	
	Here, we are going to focus on the main terms of \eqref{eq:dtN_N}. In particular, we will rewrite $N_1$, $N_2$ in terms of the localized variables $\bd \et$.
	
	\medskip
	
	First, we focus on $N_1$. Recalling \eqref{eq:psi_deriv},  we obtain
	\begin{equation}\label{eq:N1_split}
	\begin{aligned}
	N_1=&~{}
	-\frac12 \int   \big(\chi_A^2 \zB^2+(\chi_A^2)'\vB \big) \left( (\partial_x z_1+c\partial_x z_2)^2+3 (\partial_x^2 z_2)^2 \right)
	\\
	=&~{}
	-\frac12 \int  \chi_A^2 \zB^2 \left((\partial_x z_1+c\partial_x z_2)^2+3 (\partial_x^2 z_2)^2 \right)
	-\frac12 \int  (\chi_A^2)' \vB \left( (\partial_x z_1+c\partial_x z_2)^2+3 (\partial_x^2 z_2)^2 \right)
	\\
	=&~{}
	-\frac12 \int  \chi_A^2 \zB^2 (\partial_x z_1+c\partial_x z_2)^2
	-\frac32 \int  \chi_A^2 \zB^2  (\partial_x^2 z_2)^2
	-\frac12 \int  (\chi_A^2)' \vB \left( (\partial_x z_1+c\partial_x z_2)^2+3 (\partial_x^2 z_2)^2 \right)
	\\
	=&~{} N_{1,1}+ N_{1,2}+ N_{1,3}.
	\end{aligned}
	\end{equation}
	Then, for $N_{1,1}$, applying \eqref{eq:CzB_zix}, we obtain

	\[
	\begin{aligned}
	N_{1,1}
	=&~{}
	 -\frac12 \int  (\partial_x \et_1+c\partial_x \et_2)^2 
	 -\frac12 \int    \frac{\zB''}{\zB} (\et_1+c \et_2)^2
	 -\frac12 \int  \mathcal{E}_1 \zB^2 (z_1+c z_2)^2.
	\end{aligned}
	\]
	Now, for $N_{1,2}$, just replacing the identity \eqref{eq:CZ_zxx}, one gets
	\[
	\begin{aligned}
	N_{1,2} 
	=&~{} -\frac32 \int  \chi_A^2 \zB^2  (\partial_x^2 z_2)^2
	\\
	=&~{}
	-\frac32\int (\partial_x^2 \et_2)^2
	-\frac32\int \widetilde{R}_1\et_2^2
	-\frac32 \int P_1 (\partial_x \et_2)^2 
	-\frac32 \int    \left( P_1' \frac{\zB'}{\zB} + P_1 \frac{\zB''}{\zB}\right) \et_2^2 
	\\&
	-\frac32 \int \mathcal{E}_2 \zB^2 z_2^2
	-\frac32 \int  \mathcal{E}_1(P_1)\zB^2 z_2^2
	-\frac32\int \mathcal{E}_3 \zB^2 (\partial_x z_2)^2.
	\end{aligned}
	\]
	Therefore, for $N_1$ it holds the following decomposition 
	\begin{equation}\label{eq:N1}
	\begin{aligned}
	N_1
	=&~{}
	-\frac12 \int\left( (\partial_x \et_1+c\partial_x \et_2)^2 +3(\partial_x^2 \et_2)^2\right)
	+N_{\et}+N_{z}+N_{1,3},
	\end{aligned}
	\end{equation}
	where the error terms are $N_{1,3}$ (in \eqref{eq:N1_split}) and the ones given by
	\begin{equation}\label{eq:N1_E}
	\begin{aligned}
	N_{\et}=&~{}
	-\frac12 \int    \frac{\zB''}{\zB} (\et_1+c \et_2)^2
	  -\frac32\int \widetilde{R}_1\et_2^2
	-\frac32 \int P_1 (\partial_x \et_2)^2 
	-\frac32 \int    \left( P_1' \frac{\zB'}{\zB} + P_1 \frac{\zB''}{\zB}\right) \et_2^2 ,
	 \\
	 N_{z}=&~{}
	 -\frac12 \int  \mathcal{E}_1 \zB^2 (z_1+c z_2)^2
	-\frac32 \int \mathcal{E}_2 \zB^2 z_2^2
	-\frac32 \int  \mathcal{E}_1(P_1)\zB^2 z_2^2
	-\frac32\int \mathcal{E}_3 \zB^2 (\partial_x z_2)^2
	.
	\end{aligned}
	\end{equation}
 	Here,  $\mathcal{E}_1,~\mathcal{E}_2 ,~ \mathcal{E}_3$, $P_1,~\widetilde{R}_1$ are defined in \eqref{eq:claimE1}, \eqref{eq:claimE2_1}, \eqref{eq:claimE3_1} and \eqref{eq:tildeR1_P1}, respectively. It follows directly from the definition of $P_1,~\widetilde{R}_1$ and  \eqref{eq:z'/z}, that
\begin{equation}\label{eq:N1et_E}
	\begin{aligned}
	|N_{\et}|
	\lesssim& ~{} B^{-1}\left(
						\|\et_1+c \et_2\|_{L^2}^2
						+ \|\et_2\|_{L^2}^2
						+ \|\partial_x \et_2\|_{L^2}^2
				\right).
	\end{aligned}
	\end{equation}
	For the second error term, related to the variable $z$, by the definition of $\mathcal{E}_1,~\mathcal{E}_2 ,~ \mathcal{E}_3$, $P_1$, \eqref{eq:C*z'/z} and \eqref{eq:z2_w2} with \eqref{eq:z1_w1}
	\begin{equation}\label{eq:N1z_E}
	\begin{aligned}
	 |N_{z}|
	\lesssim&~{} (AB)^{-1}\gar^{-2} \big(\| w_1 \|_{L^2}^2+ \| cw_1+w_2 \|_{L^2}^2\big). 
	\end{aligned}
	\end{equation}
	For $N_{1,3}$ in \eqref{eq:N1_split}, we apply \eqref{eq:vB_psi_prop} and \eqref{eq:chiA_zA}, then we obtain
	\[ 
	\begin{aligned}
	|N_{1,3}|
	\lesssim& ~{}
	A^{-1}B \big( \| \zA^2 (\partial_x z_1+c\partial_x z_2)\|_{L^2}^2+\|\zA^2 \partial_x^2 z_2\|_{L^2}^2 \big)
	.
	\end{aligned}
	\] 
	Using \eqref{eq:z2_w2} and \eqref{eq:z1_w1}, we obtain
	\begin{equation}\label{eq:N1z1_E}
	\begin{aligned}
	|N_{1,3}|
	\lesssim& ~{}
	A^{-1}B\gar^{-2} \big( \|w_1\|_{L^2}^2+\| \partial_x w_1\|_{L^2}^2+\|cw_1+w_2\|_{L^2}^2 \big)	
	.
	\end{aligned}
	\end{equation}

	Now, we focus on $N_2$. Letting  $P_c=1-c^2-f'(\Qc)$ and replacing \eqref{eq:P_CzB_zix} 
	\begin{equation}\label{eq:N2}
	\begin{aligned}
	N_2
	=&~{}-\frac12 \int P_c \chi_A^2\zB^2(\partial_x z_2)^2 
	\\
	=&	-\frac12 \int \big( 1-c^2-f'(\Qc) \big) (\partial_x \et_2)^2 -\frac12 \int   \left( P_c' \frac{\zB'}{\zB}+P_c\frac{\zB''}{\zB}\right) \et_2^2 
		-\frac12 \int  \mathcal{E}_1(P_c)\zB^2 z_2^2,
	\end{aligned}
	\end{equation}
	where $ \mathcal{E}_1(P_c)$ is defined in \eqref{eq:claimE1}.	
	Notice that, by \eqref{eq:z'/z}, \eqref{cota_final_E1} and \eqref{eq:z2_w2}, one gets
	\begin{equation}\label{eq:N2_E}
	\left| \int   \left( P_c' \frac{\zB'}{\zB}+P_c\frac{\zB''}{\zB}\right) \et_2^2 
		+ \int  \mathcal{E}_1(P_c)\zB^2 z_2^2\right|
	\lesssim B^{-1} \|\et_2\|_{L^2}^2+(AB)^{-1}\|c w_1+w _2\|_{L^2}^2.
	\end{equation}	
	For $N_3$, recalling \eqref{eq:vA_Qc} and joint to the inequality \eqref{eq:Q_z2x}, we have
	\begin{equation}\label{eq:N3} 
	\begin{aligned}
	|N_3| \lesssim&~{} 
	 \| \et_2\|_{L^2}^2+\| \partial_x \et_2\|_{L^2}^2+\sech \left(\gamma (p-1) A\right) \gar^{-1} \| cw_1+w_2\|_{L^2}^2.
 \end{aligned}
	\end{equation}
	Lastly, replacing \eqref{eq:psi_deriv} in  $N_4$, we obtain
	\[
	\begin{aligned}
	N_4=&~{} \frac12 \int  (\chi_A^2 (\zB^2)''+3(\chi_A^2)' (\zB^2)'+3(\chi_A^2)'' \zB^2+(\chi_A^2)'''\varphi_B) (\partial_x  z_2)^2 
	\\
	=&~{}
	 \frac12 \int  \frac{(\zB^2)''}{\zB^2}  \chi_A^2\zB^2(\partial_x  z_2)^2 
	+ \frac12 \int  (3(\chi_A^2)' (\zB^2)'+3(\chi_A^2)'' \zB^2+(\chi_A^2)'''\varphi_B) (\partial_x  z_2)^2. 
	\end{aligned}
	\]
	From \eqref{eq:z'/z}, \eqref{eq:C*z'/z}  and \eqref{eq:P_CzB_zix}, with $P= (\zB^2)''/\zB^2$, it follows 	
	\[
	\begin{aligned}
	|N_4|
	\lesssim &~{}
	 \frac12 \int  \frac{(\zB^2)''}{\zB^2}  \chi_A^2\zB^2(\partial_x  z_2)^2 
	+(AB)^{-1} \|\partial_x  z_2\|_{L^2(A<|x|<2A)}^2
	\\
	\lesssim &~{}
	B^{-1} \|\partial_x \et_2\|_{L^2}^2+B^{-2} \|\et_2\|_{L^2}^2+ (AB^3)^{-1} \|\zB z_2\|_{L^2}^2
	+(AB)^{-1} \|\zA \partial_x  z_2\|_{L^2}^2  .
	\end{aligned}
	\]
	By \eqref{eq:z2_w2}, we conclude
	\begin{equation}\label{eq:N4} 
	\begin{aligned}
	|N_4|
	\lesssim &~{}
	B^{-1} \|\partial_x \et_2\|_{L^2}^2+B^{-2} \|\et_2\|_{L^2}^2
	+(AB)^{-1} \gar^{-1} \|c w_1+w_2 \|_{L^2}^2 .
	\end{aligned}
	\end{equation}

	\subsection{Controlling nonlinear and modulation terms}
	Let us define the nonlinear term and error terms:
	\[
	\begin{aligned}
	 N_5= \int   \psi_{A,B} \partial_x z_2\partial_x M_1  
	 \quad \mbox{ and }\quad
	 N_7=\int   \psi_{A,B} \partial_x z_1 \partial_x M_2;
	\end{aligned}
	\]
	modulation terms:
	\[
	\begin{aligned}
	N_6= \int   \psi_{A,B} \partial_x z_2\partial_x M_{1,c}  
	\quad \mbox{ and } \quad 
	 N_8=\int   \psi_{A,B} \partial_x z_1 \partial_x M_{2,c}.
	\end{aligned}
	\]
	
	\subsubsection*{Controlling nonlinear terms} Let us deal with $N_5$, then replacing $\psi_{A,B}$, and $M_1$ (see \eqref{eq:psi_eti} and \eqref{eq:M}), respectively; and then integrating by parts, we obtain the following decomposition
	\[ \label{eq:N5_decomp}
	\begin{aligned}
	 N_5
	 =&~{} c\int   \psi_{A,B} \partial_x z_2 \Opg \partial_x N(u_1)+ \gar \int   \psi_{A,B} \partial_x z_2 \Opg \partial_x \left( \partial_x^2(f'(\Qc)) \partial_x z_2-2 \partial_x(f'(\Qc)) \partial_x^2 z_2\right)
\\
 	=&~{} c\int   \psi_{A,B} \partial_x z_2 \Opg \partial_x N(u_1)
	\\
	&~{} - \gar \int  ( \psi_{A,B}' \partial_x z_2 + \psi_{A,B} \partial_x^2 z_2)\Opg  \left( \partial_x^2(f'(\Qc)) \partial_x z_2-2 \partial_x(f'(\Qc)) \partial_x^2 z_2\right)
	\\
	 =&~{} c\int   \psi_{A,B} \partial_x z_2 \Opg \partial_x N(u_1)- \gar \int  \psi_{A,B}' \partial_x z_2\Opg  \left( \partial_x^2(f'(\Qc)) \partial_x z_2-2 \partial_x(f'(\Qc)) \partial_x^2 z_2\right)
	\\&- \gar \int   \psi_{A,B} \partial_x^2 z_2 \Opg  \left( \partial_x^2(f'(\Qc)) \partial_x z_2-2 \partial_x(f'(\Qc)) \partial_x^2 z_2\right)
	 \\
	 =&~{} c\int   \chi_A^2 \vB \partial_x z_2 \Opg \partial_x N(u_1)
	\\&- \gar \int  (\chi_A^2 \zB^2+(\chi_A^2)'\vB)\partial_x z_2 \Opg  \left( \partial_x^2(f'(\Qc)) \partial_x z_2-2 \partial_x(f'(\Qc)) \partial_x^2 z_2\right)
	\\&- \gar \int  \chi_A^2 \vB \partial_x^2 z_2 \Opg  \left( \partial_x^2(f'(\Qc)) \partial_x z_2-2 \partial_x(f'(\Qc)) \partial_x^2 z_2\right)
	 \\
	 =:&~{} N_{5,1}+N_{5,2}+N_{5,3}.
	\end{aligned}
	\] 
	For $N_{5,1}$, by H\"older's inequality, \eqref{eq:vB_psi_prop} and \eqref{eq:chiA_zA}, one has
	\[
	\begin{aligned}
	|N_{5,1} | = &~{} \left| c \int   \chi_A^2 \vB \partial_x z_2 \Opg \partial_x  N(u_1) \right|
	\\
	\lesssim &~{} | c| \|  \chi_A \vB \partial_x z_2\|_{L^2} \|\chi_A \Opg \partial_x  N(u_1) \|_{L^2}
	\lesssim  | c| B \|  \zA  \partial_x z_2\|_{L^2} \|\zA^2 \Opg \partial_x  N(u_1) \|_{L^2}.
	\end{aligned}
	\]
	We conclude using \eqref{eq:z2_w2}, \eqref{eq:sech_Opg_1pp} joint to \eqref{eq:sech_Opg}, then
	\begin{equation}\label{eq:N51} 
	\begin{aligned}
	|N_{5,1} | 
	\lesssim &~{} | c| B \gar^{-3/2} \| cw_1+ w_2\|_{L^2} \| f'(\Qc) w_1^2 \|_{L^2}	\lesssim  B \gar^{-3/2}\|u_1\|_{L^\infty} \| cw_1+ w_2\|_{L^2} \| w_1 \|_{L^2}.
	\end{aligned}
	\end{equation}
In the case of $N_{5,3}$, using a similiar strategy to $J_{4,1}$ in  \eqref{eq:inicio_J41}, one gets
	\[ \label{eq:N53}
	\begin{aligned}
	|N_{5,3}| 
	=&~{} \left|
	 \gar \int  \chi_A^2 \vB\phi \partial_x^2 z_2 \phi^{-1}\Opg  \left( \partial_x^2(f'(\Qc)) \partial_x z_2-2 \partial_x(f'(\Qc)) \partial_x^2 z_2\right) \right|
	 \\
	 \lesssim&~{}
	  \gar \|  \chi_A ^2\vB\phi \partial_x^2 z_2\|_{L^2} \left\| \phi^{-1}\Opg  \left( \partial_x^2(f'(\Qc)) \partial_x z_2-2 \partial_x(f'(\Qc)) \partial_x^2 z_2 \right)\right\|_{L^2}
	  .
	\end{aligned}
	\] 
	Moreover, using \eqref{eq:cosh_kink} on the second factor, we obtain
	\[\begin{aligned}
	|N_{5,3}| 
	 \lesssim&~{}
	  \gar \|  \chi_A ^2\zB \partial_x^2 z_2\|_{L^2} \| \phi^{-1}(\partial_x^2(f'(\Qc)) \partial_x z_2-2 \partial_x(f'(\Qc)) \partial_x^2 z_2)\|_{L^2}
	  \\
	   \lesssim&~{}
	  \gar \|  \chi_A ^2\zB \partial_x^2 z_2\|_{L^2} \left( \| \phi^{-1}\partial_x^2(f'(\Qc)) \partial_x z_2\|_{L^2}+\|\phi^{-1} \partial_x(f'(\Qc)) \partial_x^2 z_2)\|_{L^2}\right)
	  .
	\end{aligned}
	\]
	Since
	\[
	\begin{aligned}
	  \|\phi^{-1} \partial_x(f'(\Qc)) \partial_x^2 z_2)\|_{L^2}
	  =&~{}  \|\phi^{-1} \partial_x(f'(\Qc))(1-\chi_A^2+\chi_A^2) \partial_x^2 z_2)\|_{L^2}
	  \lesssim  \|\chi_A^2 \zB^2 \partial_x^2 z_2\|_{L^2}+\sech(A\gamma/10) \| \phi \partial_x^2 z_2\|_{L^2},
	\end{aligned}
	\]
 by \eqref{eq:z2_w2} one obtains
	\begin{equation} \label{eq:fzx}
	\begin{aligned}
	\|\phi^{-1} \partial_x(f'(\Qc)) \partial_x^2 z_2\|_{L^2}
	  \lesssim&~{}  \|\chi_A^2 \zB^2 \partial_x^2 z_2\|_{L^2}+\sech^{1/2}(A\gamma/10) \gar^{-1}\| c w_1+w_2\|_{L^2}.
	\end{aligned}
	\end{equation}
	In a similar way, using \eqref{eq:z2_w2}, one can check that
	\begin{equation} \label{eq:fzxx}
	\begin{aligned}
	\|\phi^{-1} \partial_x(f'(\Qc)) \partial_x z_2\|_{L^2}
	  \lesssim&~{}  \|\chi_A^2 \zB^2 \partial_x z_2\|_{L^2}+\sech(A\gamma/10) \gar^{-1/2}\| c w_1+w_2\|_{L^2}.
	\end{aligned}
	\end{equation}
	Using the previous inequalities and the Cauchy-Schwarz inequality, we conclude that $N_{5,3}$ satisfies the following estimation
	\begin{equation}\label{eq:N53}
	\begin{aligned}
	|N_{5,3}| 
	\lesssim&~{}
	  \gar \|  \chi_A ^2\zB \partial_x^2 z_2\|_{L^2} \left(  \|\chi_A^2 \zB^2 \partial_x^2 z_2\|_{L^2}+\|\chi_A^2 \zB^2 \partial_x z_2\|_{L^2}+\sech(A\gamma/10) \gar^{-1}\| c w_1+w_2\|_{L^2}\right)
	  \\
	  \lesssim&~{}
	  \gar  \left(  \|\chi_A^2 \zB \partial_x^2 z_2\|_{L^2}^2+\|\chi_A^2 \zB \partial_x z_2\|_{L^2}^2+\sech(A\gamma/10) \gar^{-1}\| c w_1+w_2\|_{L^2}^2\right).
	\end{aligned}
	\end{equation}
	
	Finally, for $N_{5,2}$, by H\"older's inequality, we have
	\[\begin{aligned}
	N_{5,2} 
	=&~{}
		 \gar \int  (\chi_A^2 \zB^2+(\chi_A^2)'\vB)\partial_x z_2 \Opg  \left( -\partial_x^2(f'(\Qc)) \partial_x z_2+2 \partial_x(f'(\Qc)) \partial_x^2 z_2\right)
	\\
	=&~{}
	 \gar \int  (\chi_A^2 \zB^2+2\chi_A'\vB)\chi_A\partial_x z_2 \Opg  \left(- \partial_x^2(f'(\Qc)) \partial_x z_2+2 \partial_x(f'(\Qc)) \partial_x^2 z_2\right)
	\\
	\lesssim&{}~  \gar \|(\chi_A \zB^2+2\chi_A'\vB)\chi_A\partial_x z_2 \|_{L^2} \left\| \Opg  \left( \partial_x^2(f'(\Qc)) \partial_x z_2-2 \partial_x(f'(\Qc)) \partial_x^2 z_2\right)\right\|_{L^2}
	\\
	\lesssim&{} ~ \gar\big( \|\chi_A^2 \zB^2\partial_x z_2 \|_{L^2}+\|\chi_A'\vB \chi_A\partial_x z_2 \|_{L^2}\big) \left\| \Opg  \left( \partial_x^2(f'(\Qc)) \partial_x z_2-2 \partial_x(f'(\Qc)) \partial_x^2 z_2\right)\right\|_{L^2}
	.
	\end{aligned}
	\]
Now, using \eqref{eq:vB_psi_prop},\eqref{eq:chiA_zA}, Lemma  \ref{lem:estimates_IOp} and \eqref{eq:fzx} with \eqref{eq:fzxx}, 
	\[\begin{aligned}
	|N_{5,2} |
	\lesssim&{} ~ \gar\big( \|\chi_A^2 \zB^2\partial_x z_2 \|_{L^2}+BA^{-1}\| \chi_A\partial_x z_2 \|_{L^2}\big) \big( \| \partial_x^2(f'(\Qc)) \partial_x z_2 \|_{L^2} +\| \partial_x(f'(\Qc)) \partial_x^2 z_2 \|_{L^2}\big)
	\\
	\lesssim&{} ~ \gar\big( \|\chi_A^2 \zB^2\partial_x z_2 \|_{L^2}+BA^{-1}\| \zA^2 \partial_x z_2 \|_{L^2}\big) 
	\\&\times
	\big(  \|\chi_A^2 \zB^2 \partial_x z_2\|_{L^2}+\|\chi_A^2 \zB^2 \partial_x^2 z_2\|_{L^2}+\sech(A\gamma/10) \gar^{-1}\| c w_1+w_2\|_{L^2}\big).
	\end{aligned}
	\]
Now, we apply \eqref{eq:z2_w2} and Cauchy-Schwarz inequality, we conclude that $N_{5,2}$ is bounded by
	\begin{equation} \label{eq:N52} 
	\begin{aligned}
	|N_{5,2} |
	\lesssim&{} ~ \gar \big(  \|\chi_A^2 \zB^2 \partial_x z_2\|_{L^2}^2+\|\chi_A^2 \zB^2 \partial_x^2 z_2\|_{L^2}^2+\| c w_1+w_2\|_{L^2}^2\big).
	\end{aligned}
	\end{equation}
	Then, collecting \eqref{eq:N51}, \eqref{eq:N52} and \eqref{eq:N53}; and by Cauchy-Schwarz inequality, we obtain
	\[ 
	\begin{aligned}
	|N_5|
	\lesssim &~{}
	 B \gar^{-3/2}\|u_1\|_{L^\infty}\big( \| w_1 \|_{L^2}^2+\| cw_1+ w_2\|_{L^2}^2 \big)+ \gar  \big(  \|\chi_A^2 \zB \partial_x^2 z_2\|_{L^2}^2+\|\chi_A^2 \zB \partial_x z_2\|_{L^2}^2+\| c w_1+w_2\|_{L^2}^2\big).
	\end{aligned}
	\] 
	Now, we consider \eqref{eq:CZzxx} joint to \eqref{eq:CZzx} and \eqref{eq:z2_w2}, we conclude
	\begin{equation}\label{eq:N5}
	\begin{aligned}
	|N_5|
	\lesssim &~{}
	 B \gar^{-3/2}\|u_1\|_{L^\infty}\big( \| w_1 \|_{L^2}^2+\| cw_1+ w_2\|_{L^2}^2 \big)
	 \\&
	 + \gar  \big( 
	 			 \|\partial_x^2 \eta_2\|_{L^2}^2+B^{-1}\|\partial_x \eta_2\|_{L^2}^2+B^{-1}\| \eta_2\|_{L^2}^2+(AB)^{-1}(\|\zB z_2\|_{L^2}^2+\|\zB \partial_x z_2\|_{L^2}^2)
				 \\&
				 +\| \partial_x \eta_2\|_{L^2}^2 +B^{-1}\|\eta_2\|_{L^2}^2
				 +(AB)^{-1}\|\zB z_2\|_{L^2}^2
				 +\| c w_1+w_2\|_{L^2}^2
			\big)
	\\
	\lesssim &~{}
	 B \gar^{-3/2}\|u_1\|_{L^\infty}\big( \| w_1 \|_{L^2}^2+\| cw_1+ w_2\|_{L^2}^2 \big)
	 \\&
	 + \gar  \big( 
	 			 \|\partial_x^2 \eta_2\|_{L^2}^2
				  +\| \partial_x \eta_2\|_{L^2}^2 
				 +B^{-1}\| \eta_2\|_{L^2}^2
				 +(AB)^{-1}\gar^{-1}\| c w_1+w_2\|_{L^2}^2
			\big).
	\end{aligned}
	\end{equation}
	Now, we focus on $N_7$, replacing $M_2$ on \eqref{eq:M}, we obtain 
	\[
	\begin{aligned}
	N_7 =& \int   \psi_{A,B} \partial_x z_1 \partial_x \Opg  N(u_1).
	\end{aligned}
	\]
	Similar to $N_{5,1}$, by H\"older's inequality, \eqref{eq:vB_psi_prop} and \eqref{eq:chiA_zA}, , one gets
	\[ 
	\begin{aligned}
	|N_7|
	\lesssim&~{} | c| \|  \chi_A \vB \partial_x z_1\|_{L^2} \|\chi_A \Opg \partial_x  N(u_1) \|_{L^2}
	\lesssim | c| B\gar^{-1} \|u_1\|_{L^\infty} \|  \zA \partial_x z_1\|_{L^2} \|w_1 \|_{L^2}
	.
	\end{aligned}
	\]
	Then, by \eqref{eq:z1_w1}, one obtains
	\begin{equation}\label{eq:N7} 
	\begin{aligned}
	|N_7|
	\lesssim&~{} | c| B\gar^{-2} \|u_1\|_{L^\infty} \|  w_1\|_{L^2} \big( \|  w_1\|_{L^2} +\|  \partial_x w_1\|_{L^2}+\gar^{1/2}\|cw_1+w_2 \|_{L^2}\big)\\
	\lesssim  &~{} B\gar^{-2} \|u_1\|_{L^\infty} \big( \|  w_1\|_{L^2}^2 +\|  \partial_x w_1\|_{L^2}^2
	+\gar \|cw_1+w_2 \|_{L^2}^2\big) .
	\end{aligned}
	\end{equation}
	Collecting \eqref{eq:N5} and \eqref{eq:N7} and using \eqref{eq:B_gar}, we conclude the nonlinear terms holds the following bound
	\begin{equation}\label{eq:N5+N7}
	\begin{aligned}
	|N_5+N_7|
	\lesssim &~{}
		  \gar  \big( 
	 			 \|\partial_x^2 \eta_2\|_{L^2}^2
				  +\| \partial_x \eta_2\|_{L^2}^2 
				 +\| \eta_2\|_{L^2}^2
			\big)
	\\&
	+\big( (AB)^{-1}+B \gar^{-3/2}\|u_1\|_{L^\infty}+B\gar^{-2} \|u_1\|_{L^\infty} \big)\big( 
		\|  w_1\|_{L^2}^2 
		+\|  \partial_x w_1\|_{L^2}^2
		+ \|cw_1+w_2 \|_{L^2}^2
		\big)
	\\
		\lesssim &~{}
		  \gar  \big( 
	 			 \|\partial_x^2 \eta_2\|_{L^2}^2
				  +\| \partial_x \eta_2\|_{L^2}^2 
				 +\| \eta_2\|_{L^2}^2
			\big)
	+ B^{-1}\big( 
		\|  w_1\|_{L^2}^2 
		+\|  \partial_x w_1\|_{L^2}^2
		+ \|cw_1+w_2 \|_{L^2}^2
		\big)
	.
	\end{aligned}
	\end{equation}

	\subsubsection*{Controlling modulation terms}
	We focus on $N_6$, replacing $M_{1,c}$ in \eqref{eq:Mc}, we obtain the following separation
		\[
	\begin{aligned}
	N_6=&~{} \int   \psi_{A,B} \partial_x z_2\partial_x M_{1,c}  
	\\
	=&~{} \int   \psi_{A,B} \partial_x z_2\partial_x \left( c' \Opg(c\Qc+ f''(\Qc)\Lambda \Qc v_1+2cv_1-v_2)  +(\rho'-c) \left( \partial_x z_1 -\Opg (\partial_x(f'(\Qc)) v_1) \right)\right)
	\\
	=&~{} c' \int   \psi_{A,B} \partial_x z_2 \Opg\partial_x \left( c\Qc+ f''(\Qc)\Lambda \Qc v_1 +2cv_1-v_2 \right)
	\\&+(\rho'-c) \int   \psi_{A,B} \partial_x z_2\partial_x^2 z_1 
	-(\rho'-c) \int   \psi_{A,B} \partial_x z_2 \Opg \partial_x \left[\partial_x(f'(\Qc)) v_1 \right]
	\\
	=:&~{} N_{6,1}+N_{6,2}+N_{6,3}.
	\end{aligned}
	\]
	Now, for $N_{6,1}$, using  H\"older's inequality, one gets
	\[
	\begin{aligned}
	|N_{6,1} |
	\lesssim &~{} | c' | \|   \chi_A \vB \partial_x z_2\|_{L^2} \| \chi_A \Opg\partial_x (c\Qc+ f''(\Qc)\Lambda \Qc v_1-v_2)
	\|_{L^2}.
	\end{aligned}
	\]
	Applying \eqref{eq:vB_psi_prop}, \eqref{eq:sech_Opg_p}, \eqref{eq:chiA_zA}, \eqref{eq:z2_w2}  join to $B^{1/4}\gamma>2$, we obtain
		\begin{equation}\label{eq:N61}
	\begin{aligned}
	|N_{6,1} |
	\lesssim &~{} | c' | \gar^{-1/2} B \|   \zA  \partial_x z_2\|_{L^2} \| \zA (c\Qc+ f''(\Qc)\Lambda \Qc v_1+2cv_1 -v_2)
	\|_{L^2}
	\\
	\lesssim &~{} | c' | \gar^{-1/2} B \|   \zA  \partial_x z_2\|_{L^2} \| (c\zA \Qc+ f''(\Qc)\Lambda \Qc w_1 + 2cw_1-w_2)
	\|_{L^2}
	\\
	\lesssim &~{} | c' | \gar^{-1} B \|  cw_1+w_2\|_{L^2} \left( 1+\| f''(\Qc)\Lambda \Qc w_1 + 2cw_1\|_{L^2}+\|cw_1+w_2\|_{L^2}\right)
	\\
	\lesssim &~{} | c' | \gar^{-1} B \|  cw_1+w_2\|_{L^2} \left( 1+\frac{c}{\gamma^2}\|w_1\|_{L^2}+\|cw_1+w_2\|_{L^2}\right)
	\\
	\lesssim &~{} | c' | \gar^{-1}  B^{3/2}  \left( \|  cw_1+w_2\|_{L^2}+\|cw_1+w_2\|_{L^2}^2+\|w_1\|_{L^2}^2\right).
	\end{aligned}
	\end{equation}
	For $N_{6,3}$, we follow the same strategy consider for $N_{5,3}$ \eqref{eq:N53}. Let us consider   the auxiliary function $\phi=\sech(\gamma \cdot/2)$ and by H\"older's inequality, one get
	\[
	\begin{aligned}
	|N_{6,3} |
	=&~{} \left| (\rho'-c) \int   \psi_{A,B} \partial_x z_2 \Opg \partial_x \left[\partial_x(f'(\Qc)) v_1 \right]
	\right|
	\\
	\lesssim &~{}  | \rho'-c| \|   \chi_A \phi \vB \partial_x z_2\|_{L^2} 
	\|\phi^{-1}\chi_A \Opg \partial_x \left(\partial_x(f'(\Qc)) v_1 \right) \|_{L^2}.
	\end{aligned}
	\]
	Then, using a similar bound to \eqref{eq:vA_Qc}, applying \eqref{eq:cosh_kink}, \eqref{eq:sech_v1_w1} and \eqref{eq:CZzx},  we obtain
	\begin{equation}\label{eq:N63}
	\begin{aligned}
	|N_{6,3} |
	\lesssim &~{}  | \rho'-c| \|   \chi_A \zB \partial_x z_2\|_{L^2} 
	\|\phi^{-1}\partial_x^2(f'(\Qc)) v_1 + \phi^{-1} \partial_x(f'(\Qc)) \partial_x v_1 \|_{L^2}
	\\
	\lesssim &~{}  | \rho'-c| \|   \chi_A \zB \partial_x z_2\|_{L^2} 
	\big(
	\| w_1 \|_{L^2}+\| \partial_x w_1 \|_{L^2}\big)
	\\
	\lesssim &~{}  | \rho'-c| \big( \|  \partial_x \eta_2\|_{L^2}^2+B^{-1}\|\eta_2\|_{L^2}^2+(AB)^{-1}\|cw_1+w_2\|_{L^2}^2+\| w_1 \|_{L^2}^2+\| \partial_x w_1 \|_{L^2}^2\big).
	\end{aligned}
	\end{equation}
	Finally, in the case of $N_8$, we consider the following decomposition
	\[
	\begin{aligned}
	 N_8=&~{}\int   \psi_{A,B} \partial_x z_1 \partial_x M_{2,c} 
	 \\
	 =&~{}\int   \psi_{A,B} \partial_x z_1 \partial_x \bigg(- c' \Opg (\Qc+v_1) +(\rho'-c)  \partial_x z_2 \bigg)
	 \\
	 =&~{} -c' \int   \psi_{A,B} \partial_x z_1   \Opg \partial_x \left(\Qc+v_1\right)	+(\rho'-c)\int   \psi_{A,B} \partial_x z_1  \partial_x^2 z_2
	=: N_{8,1}+N_{8,2}.
	\end{aligned}
	\]
	For $N_{8,1}$, using H\"older's inequality jointly to \eqref{eq:chiA_zA} and \eqref{eq:sech_Opg}, one can see
	\[
	\begin{aligned}
	|N_{8,1} |
	=&~{} \left| 
	-c' \int   \psi_{A,B} \partial_x z_1   \Opg \partial_x \left(\Qc+v_1\right)
	\right|
	\\
	\lesssim &~{} | c'| \|  \chi_A \vB \partial_x z_1\|_{L^2}   \|\chi_A \Opg \partial_x \left(\Qc+v_1\right)\|_{L^2}
	\\
	\lesssim &~{} | c'| B \|  \zA \partial_x z_1\|_{L^2}   \|\zA \Opg \partial_x \left(\Qc+v_1\right)\|_{L^2}
		\\
	\lesssim &~{} | c'| B \|  \zA \partial_x z_1\|_{L^2}   (1+\|\zA \partial_x v_1\|_{L^2})	\lesssim  | c'| B \|  \zA \partial_x z_1\|_{L^2}   (1+\|\zA \partial_x v_1\|_{L^2}).
	\end{aligned}
	\]
	We conclude applying \eqref{eq:z1_w1} and \eqref{eq:sech_v1_w1}, 
		\begin{equation}\label{eq:N81}
	\begin{aligned}
	|N_{8,1} |
	\lesssim &~{} | c'| B\gar^{-1}(\|w_1\|_{L^2}+\|\partial_x w_1\|_{L^2})    (1+\|w_1\|_{L^2}+\|\partial_x w_1\|_{L^2})
	\\
	\lesssim &~{} | c'| B\gar^{-1}(\|w_1\|_{L^2}+\|\partial_x w_1\|_{L^2}+\|cw_1+w_2\|_{L^2}+\|w_1\|_{L^2}^2+\|\partial_x w_1\|_{L^2}^2+\|cw_1+w_2\|_{L^2}^2).
	\end{aligned}
	\end{equation}
	Lastly, we have that
	\[
	N_{6,2}+N_{8,2}= -(\rho'-c)\int   \psi_{A,B}' \partial_x z_1  \partial_x z_2,
	\]
	considering \eqref{eq:psi_deriv}, \eqref{eq:vB_psi_prop} and \eqref{eq:chiA_zA}; lastly applying \eqref{eq:z1_w1} with \eqref{eq:z2_w2}, one gets
	\[ \label{eq:N62+N82}
	\begin{aligned}	
	|N_{6,2}+N_{8,2}|
	\lesssim&~{} |\rho'-c|    \left( 
							\| \chi_A \zB\partial_x z_1 \|_{L^2} \| \chi_A \zB\partial_x z_2 \|_{L^2}
							+ A^{-1} B \| \zA^2\partial_x z_1\|_{L^2} \| \zA^2 \partial_x z_2\|_{L^2} \right)
	\\
	\lesssim&~{} |\rho'-c|  \| \chi_A \zB\partial_x z_1 \|_{L^2} \| \chi_A \zB\partial_x z_2 \|_{L^2}
			\\&	
			+ |\rho'-c|   A^{-1} B\gar^{-3/2} ( \| w_1\|_{L^2}+ \| \partial_x w_1\|_{L^2}+ \| cw_1+w_2\|_{L^2}) \| cw_1+w_2\|_{L^2}.				.
	\end{aligned}
	\] 
	Finally, by \eqref{eq:CZzx}, \eqref{eq:z1_w1}, \eqref{eq:z2_w2} and \eqref{eq:B_gar}, we conclude
	\begin{equation}\label{eq:N62+N82}
	\begin{aligned}	
	|N_{6,2}+N_{8,2}|
	\lesssim&~{} |\rho'-c|  \big( \| \partial_x \et_1\|_{L^2}^2+B^{-1}\| \et_1\|_{L^2}^2
	+ (AB)^{-1}\| \zeta_{B}z_1\|_{L^2}^2
						+\| \partial_x \et_2\|_{L^2}^2+B^{-1}\| \et_2\|_{L^2}^2
	+ (AB)^{-1}\| \zeta_{B}z_2\|_{L^2}^2
							\big)
			\\&	
			+ |\rho'-c|   A^{-1} B\gar^{-3/2} ( \| w_1\|_{L^2}^2+ \| \partial_x w_1\|_{L^2}^2+ \| cw_1+w_2\|_{L^2}^2)
	\\
		\lesssim&~{} |\rho'-c|  \big( \| \partial_x \et_1+c\partial_x \et_2\|_{L^2}^2
		+\| \partial_x \et_2\|_{L^2}^2
		+B^{-1}\| \et_1+c\eta_2\|_{L^2}^2+B^{-1}\| \et_2\|_{L^2}^2
							\big)
			\\&	
			+ |\rho'-c|   (A^{-1} B\gar^{-3/2}+ (AB)^{-1}\gar^{-1}) ( \| w_1\|_{L^2}^2+ \| \partial_x w_1\|_{L^2}^2+ \| cw_1+w_2\|_{L^2}^2)
	\\
	\lesssim&~{} |\rho'-c|  \big( \| \partial_x \et_1+c\partial_x \et_2\|_{L^2}^2
		+\| \partial_x \et_2\|_{L^2}^2
		+B^{-1}\| \et_1+c\eta_2\|_{L^2}^2+B^{-1}\| \et_2\|_{L^2}^2
							\big)
			\\&	
			+ |\rho'-c|  B^{-3} ( \| w_1\|_{L^2}^2+ \| \partial_x w_1\|_{L^2}^2+ \| cw_1+w_2\|_{L^2}^2).
	\end{aligned}
	\end{equation}
	Collecting \eqref{eq:N61}, \eqref{eq:N63}, \eqref{eq:N81} and \eqref{eq:N62+N82} and after rearrange the terms, we conclude
	\begin{equation}\label{eq:N6+N8}
	\begin{aligned}
	|N_6+N_8|
	\lesssim& ~{} 
	 | c' | \gar^{-1}  B^{3/2}  \left( \|  cw_1+w_2\|_{L^2}+\|cw_1+w_2\|_{L^2}^2+\|w_1\|_{L^2}^2\right)
	 \\&
	 + | \rho'-c| \big( \|  \partial_x \eta_2\|_{L^2}^2+B^{-1}\|\eta_2\|_{L^2}^2+(AB)^{-1}\|cw_1+w_2\|_{L^2}^2+\| w_1 \|_{L^2}^2+\| \partial_x w_1 \|_{L^2}^2\big)
	 \\&
	 +| c'| B\gar^{-1}(\|w_1\|_{L^2}+\|\partial_x w_1\|_{L^2}+\|cw_1+w_2\|_{L^2}+\|w_1\|_{L^2}^2+\|\partial_x w_1\|_{L^2}^2+\|cw_1+w_2\|_{L^2}^2)
	 \\&
	 +|\rho'-c|  \big( \| \partial_x \et_1+c\partial_x \et_2\|_{L^2}^2
		+\| \partial_x \et_2\|_{L^2}^2
		+B^{-1}\| \et_1+c\eta_2\|_{L^2}^2+B^{-1}\| \et_2\|_{L^2}^2
							\big)
	\\&	
			+ |\rho'-c|  B^{-3} ( \| w_1\|_{L^2}^2+ \| \partial_x w_1\|_{L^2}^2+ \| cw_1+w_2\|_{L^2}^2).
	\\
		\lesssim& ~{} 
	 |\rho'-c|  \big( \| \partial_x \et_1+c\partial_x \et_2\|_{L^2}^2
		+\| \partial_x \et_2\|_{L^2}^2
		+B^{-1}\| \et_1+c\eta_2\|_{L^2}^2+B^{-1}\| \et_2\|_{L^2}^2
							\big)
	\\&	
	+ |\rho'-c|  B^{-3} ( \| w_1\|_{L^2}^2+ \| \partial_x w_1\|_{L^2}^2+ \| cw_1+w_2\|_{L^2}^2)
	\\&
	+| c'| {B^{3/2}}\gar^{-1}(\|w_1\|_{L^2}+\|\partial_x w_1\|_{L^2}+\|cw_1+w_2\|_{L^2}+\|w_1\|_{L^2}^2+\|\partial_x w_1\|_{L^2}^2+\|cw_1+w_2\|_{L^2}^2).
	\end{aligned}
	\end{equation}
	
	\subsection{End of proof} We now end the proof of  Proposition \ref{prop:N}.
	Recalling that  $\frac{d}{dt}\N$ holds \eqref{eq:dtN_N}, and replacing $N_1$ and $N_2$ obtained in \eqref{eq:N1} and \eqref{eq:N2}, respectively. We get
	\begin{equation*}
\begin{aligned}
	\dfrac{d}{dt} \N
	=&
	-\frac12 \int \left( 
				(\partial_x \et_1+c\partial_x \et_2)^2 
				+3(\partial_x^2 \et_2)^2
				+\big( 1-c^2-f'(\Qc) \big) (\partial_x \et_2)^2 
			 \right)
	\\&
	-\frac12  \int   \left( P_c' \frac{\zB'}{\zB}+P_c\frac{\zB''}{\zB}\right) \et_2^2 
	-\frac12 \int  \mathcal{E}_1(P_c)\zB^2 z_2^2
	+N_{\et}+N_{z}+N_{z,1}+ \sum_{j=3}^8 N_{j}.
	\end{aligned}
	\end{equation*}
	recalling that $P_c=1-c^2-f'(\Qc)$ and where $N_{\et}, ~N_{z},~N_{1,3}$ are given in \eqref{eq:N1_E} and \eqref{eq:N1_split}.
	
	\medskip
	
	Now, gathering and rearranging the error terms of $N_1$ in \eqref{eq:N1et_E}, \eqref{eq:N1z_E}, and \eqref{eq:N1z1_E}, along with those from \eqref{eq:N2_E}, \eqref{eq:N3}, \eqref{eq:N4}, \eqref{eq:N5+N7}, and \eqref{eq:N6+N8}, into the previous expression, and recalling 
\eqref{eq:B_gar}, we obtain
\[
\begin{aligned}
	\dfrac{d}{dt} \N
\leq~{}&
	-\frac12 \int \left( 
				(\partial_x \et_1+c\partial_x \et_2)^2 
				+3(\partial_x^2 \et_2)^2
			 \right)
+C_3 B^{-4}  \|\partial_x^2 \eta_2\|_{L^2}^2
+C_3 \big( \| \et_2\|_{L^2}^2+\| \partial_x \et_2\|_{L^2}^2\big)
\\&
+C_3B^{-1}				\|\et_1+c \et_2\|_{L^2}^2
+	C_3 |\rho'-c|  \big( \| \partial_x \et_1+c\partial_x \et_2\|_{L^2}^2
		+\| \partial_x \et_2\|_{L^2}^2
		+B^{-1}\| \et_1+c\eta_2\|_{L^2}^2+B^{-1}\| \et_2\|_{L^2}^2
							\big)
\\&
	+ C_3( |\rho'-c|  B^{-3}+B^{-1})\big( 
					\|  w_1\|_{L^2}^2 
					+\|  \partial_x w_1\|_{L^2}^2
					+ \|cw_1+w_2 \|_{L^2}^2
		\big)
	\\&
	+C_3| c'| B^{11/2} (\|w_1\|_{L^2}+\|\partial_x w_1\|_{L^2}+\|cw_1+w_2\|_{L^2}+\|w_1\|_{L^2}^2+\|\partial_x w_1\|_{L^2}^2+\|cw_1+w_2\|_{L^2}^2).
\end{aligned}
\]

Recalling \eqref{eq:c'}, we observe that $\|\bd{w}\|_{H^1\times L^2}\leq \|\bd{u}\|_{H^1\times L^2}\leq \delta$, moreover $|c'|\leq \delta^2$ and $|\rho'-c|\leq \delta$, then
\[
\begin{aligned}
	\dfrac{d}{dt} \N
\leq~{}&
	-\frac14 \int \left( 
				(\partial_x \et_1+c\partial_x \et_2)^2 
				+3(\partial_x^2 \et_2)^2
			 \right)
+C_3 \big( \| \et_2\|_{L^2}^2+\| \partial_x \et_2\|_{L^2}^2\big)
\\&
+	C_3 B^{-1} \| \et_1+c\eta_2\|_{L^2}^2
	+ C_3B^{-1}\big( 
					\|  w_1\|_{L^2}^2 
					+\|  \partial_x w_1\|_{L^2}^2
					+ \|cw_1+w_2 \|_{L^2}^2
		\big)
	\\&
	+C_3| c'| B^{11/2} (\|w_1\|_{L^2}+\|\partial_x w_1\|_{L^2}+\|cw_1+w_2\|_{L^2}).
\end{aligned}
\]
	 for some fix constant $C_3>0$.

\section{Coercivity estimates}\label{Sec:6}
In this section, we deal with the uncontrolled terms in \eqref{cota_I}. 

\subsection{First Coercivity estimate}
First we present a coercivity estimate for the centered variable $(v_1,v_2)$ under  \eqref{eq:orto}.
\begin{lem}\label{lem6p1}
There exists $\mu_0>0$ such that if \eqref{eq:orto} holds and $\gamma>0$ sufficiently small, one has
	\[
	\mu_0 \int\sech(\gamma x)  [v_1^2+(\partial_x v_1)^2] \leq  \int  \sech(\gamma x)  \left((\partial_x v_1)^2+(1-c^2-f'(\Qc))v_1^2 +(v_2+cv_1)^2\right).
	\]	
\end{lem}

\begin{proof} This is a standard consequence of \eqref{coer0} and the fact that $\gamma$ is sufficiently small. For a similar proof, see e.g. \cite{CMPS}.
\end{proof} 

	\subsection{Second coercivity estimate}
	In this subsection, we will focus on the last term on the RHS of the previous coercivity.
	\begin{lem}\label{prop:coer_v2}
	One has
	\begin{equation}\label{lem6p2}
	\begin{aligned}
	 \int \sech^2\left(\frac{\gamma y}{4}\right) (v_2+cv_1)^2 
	 	\lesssim&~{}
	 \| w_2+c w_1\|_{L^2} \| \et_2\|_{L^2}+ B^{-4} \| w_2+c w_1\|_{L^2}^2
	+ B^{-4} \big( \| \partial_x^2 \et_2\|_{L^2}^2+\| \partial_x \et_2\|_{L^2}^2+\| \et_2\|_{L^2}^2 \big).
	\end{aligned}
	\end{equation}
	\end{lem}
	\begin{proof}
	Recalling the definition of $z_2$ on \eqref{eq:z_cv}, we have
	\[
	\begin{aligned}
	 \int \sech^2 \left( \frac{\gamma y}{4} \right) (v_2+c v_1)^2
	 =& \int \sech^2 \left( \frac{\gamma y}{4} \right) (v_2+c v_1) (1-\gar \partial_x^2) z_2
	 \\
	 =&~{}  \int \sech^2 \left( \frac{\gamma y}{4} \right)(1-\chi_A^2+\chi_A^2) (v_2+c v_1)  z_2
	 \\&
	 -\gar \int \sech^2 \left( \frac{\gamma y}{4} \right) (v_2+c v_1)(1-\chi_A^2+\chi_A^2)  \partial_x^2 z_2
	 =: L_1  +L_2.
	\end{aligned}
	\]
	For $L_1$, by Holder's inequality, \eqref{eq:chiA_zA} and \eqref{eq:z2_w2}, joint to the definition of $w_i$ and $\et_i$, we get
	 \begin{equation}\label{eq:L1}
	\begin{aligned}
	|L_1| 
	\leq &\left| \int \sech^2 \left( \frac{\gamma y}{4} \right)\zB^{-1} \chi_A (v_2+c v_1) \zB \chi_A  z_2\right|
	+ \left|\int \sech^2 \left( \frac{\gamma y}{4} \right)\zA^{-2} (1-\chi_A^2) \zA(v_2+c v_1)  \zA z_2 \right|
	\\
	\lesssim&~{}   \| w_2+c w_1\|_{L^2} \| \et_2\|_{L^2}
	+  \sech \left( \frac{\gamma A}{4} \right) \| w_2+c w_1\|_{L^2}^2 .
	\end{aligned}
	\end{equation}
	 Similarly for $L_2$, using H\"older's inequality and by the definition of $w_i$ and $\et_i$, we have
		\[
	\begin{aligned}
	|L_2|
	\leq&~{}	  
	 \gar \left| \int \sech^2 \left( \frac{\gamma y}{4} \right)\zB^{-1} \chi_A (v_2+c v_1) \chi_A \zB \partial_x^2 z_2 \right|
	 +\gar \left| \int \sech^2 \left( \frac{\gamma y}{4} \right) (v_2+c v_1)(1-\chi_A^2 ) \partial_x^2 z_2 \right| 
	 \\
	  \lesssim&~{} 
	 \gar \| \chi_A (v_2+c v_1)\|_{L^2} \|\chi_A \zB \partial_x^2 z_2\|_{L^2}
	  +\gar\sech \left( \frac{\gamma A}{4} \right) \| w_2+c w_1 \|_{L^2} \| \zA(1-\chi_A^2 ) \partial_x^2 z_2\|_{L^2}.
	\end{aligned}
	\]
	Applying \eqref{eq:chiA_zA}, \eqref{eq:z2_w2} together with \eqref{eq:sech_v2_w2}; Cauchy Schwarz inequality and \eqref{eq:CZzxx}, we conclude that
	\begin{equation}\label{eq:L2}
	\begin{aligned}
	|L_2|
	  \lesssim&~{} 
	 \gar \|  w_2+c w_1\|_{L^2} \|\chi_A \zB \partial_x^2 z_2\|_{L^2}+ \sech \left( \frac{\gamma A}{4} \right) \| w_2+c w_1\|_{L^2}^2
	 \\
	 \lesssim&~{} 
	 \gar \|\chi_A \zB \partial_x^2 z_2\|_{L^2}^2+ \left(\sech \left( \frac{\gamma A}{4} \right)+\gar\right) \| w_2+c w_1\|_{L^2}^2
	 \\
	  \lesssim&~{} 
	 \gar (  \| \partial_x^2 \et_2\|_{L^2}^2+B^{-1}\| \partial_x \et_2\|_{L^2}^2+B^{-1}\| \et_2\|_{L^2}^2 
	+(AB)^{-1}  \gar^{-1}\|w_2+c w_1\|_{L^2}^2)
	+\left( \sech \left( \frac{\gamma A}{4} \right) +\gar \right)\| w_2+c w_1\|_{L^2}^2.
	\end{aligned}
	\end{equation}
	Collecting $L_1$ in \eqref{eq:L1} and $L_2$  in \eqref{eq:L2}, and recalling \eqref{eq:B_gar}, we obtain
	\begin{equation*}
	\begin{aligned}
	\int \sech^2 \left( \frac{\gamma y}{4} \right) (v_2+c v_1)^2
	\lesssim&~{}
	 \| w_2+c w_1\|_{L^2} \| \et_2\|_{L^2}+ B^{-4} \| w_2+c w_1\|_{L^2}^2
	\\&
	+ B^{-4} \big( \| \partial_x^2 \et_2\|_{L^2}^2+\| \partial_x \et_2\|_{L^2}^2+\| \et_2\|_{L^2}^2 \big).
	\end{aligned}
	\end{equation*}
	This concludes the proof of \eqref{lem6p2}. 
	\end{proof}

\subsection{Third coercivity estimate}

Now we present a decomposition in terms of the localized variables $\bd w=(w_1,w_2)$ and $\bd \et =(\eta_1,\eta_2)$ for the term	$\int  \sech\left(\frac{\gamma}{2} x\right)  v_1^2$.

	\begin{lem}\label{prop:coer_v1}
	One has
	\[
	\begin{aligned}
	\int  \sech(\gamma x)  v_1^2 \lesssim &~{} 
	\| w_1\|_{L^2} \big(\|\et_1+c\eta_2\|_{L^2} + \| \et_2\|_{L^2} \big)
	+ \| w_2+c w_1\|_{L^2} \| \et_2\|_{L^2}
	\\&
	+B^{-4} \big( \|w_1 \|_{L^2}^2 + \| \partial_x w_1 \|_{L^2}^2 +\| cw_1+w_2\|_{L^2}^2\big)
	 \\&
	+B^{-4} \big(		\| \partial_x \et_1+c \partial_x \et_2\|_{L^2}^2+\| \et_1+c\eta_2 \|_{L^2}^2
				+\| \partial_x^2 \et_2\|_{L^2}^2
	 			+\| \partial_x \et_2\|_{L^2}^2+\| \et_2\|_{L^2}^2
				\big)	
	\end{aligned}
	\]
	\end{lem}
	\begin{proof}
	For $\ell$ small and applying Lemma \ref{lem6p1}, one gets
	\[
	\begin{aligned}
	 \int  \sech(\gamma x)  v_1^2
	 \leq&~{}  \int  \sech(\gamma \ell x)  v_1^2
	\leq ~{} \int\sech(\gamma\ell x)  [v_1^2+(\partial_x v_1)^2] \\
	\leq &~{} \mu_0^{-1} \int  \sech(\gamma\ell x)  \big((\partial_x v_1)^2+(1-c^2-f'(\Qc))v_1^2\big)
	+\int\sech(\gamma\ell x) (v_2+c v_1)^2.
	\end{aligned}
	\]
	Integrating by parts, we get
	\[
	\begin{aligned}
	\int  \sech(\gamma \ell x)  v_1^2
	 \leq&~{} \mu_0^{-1} \int  \sech(\gamma \ell x)  v_1 \LL v_1
			+\frac{1}{2\mu_0 } \int\partial_x^2 (\sech(\gamma\ell x)) v_1^2+\int\sech(\gamma \ell x) (v_2+c v_1)^2,
	\end{aligned}
	\]
	and we get
		\[
	\begin{aligned}
	\int  \sech(\gamma \ell x)  v_1^2
	\leq&~{} \mu_0^{-1} \int  \sech(\gamma \ell x)  v_1 \LL v_1+\int\sech(\gamma \ell x) (v_2+c v_1)^2 :=  L_3+L_4.
	\end{aligned}
	\]
	Notice that $L_4$ was treated in  \eqref{lem6p2}. 
	\\
	Next, we focus on $L_3$. From \eqref{eq:z_cv}, we can observe
	$z_1-c z_2 =\Opg \LL v_1$, then
	\[
	\begin{aligned}
	\mu_0 L_3=
	 &~{}  \int  \sech(\gamma \ell x)  v_1 (1-\gar \partial_x^2)(z_1 - c z_2)
	 \\
	 =&~{}  \int  \sech(\gamma \ell x)  v_1 (z_1 - c z_2)
	 -\gar \int  \sech(\gamma \ell x)  v_1 \partial_x^2(z_1 - c z_2)
	 =: \ell_1+\ell_2.
	\end{aligned}
	\]
	
	Firstly, we are going to deal with  $\ell_1$. Recalling \eqref{eq:wi}, \eqref{eq:psi_eti}, \eqref{eq:def_J}, joint to \eqref{eq:chiA_zA} and applying H\"older's inequality, we get
	\[
	\begin{aligned}
	|\ell_1| 
	=&  \left| \int  \chi_A^3 \sech(\gamma \ell x)  v_1 (z_1 - c z_2)
	+ \int  (1-\chi_A^3) \sech(\gamma \ell x)  v_1 (z_1 - c z_2) \right|
	\\
	\leq & \left| \int  \chi_A^2 \sech(\gamma \ell x)\zB^{-1}  v_1 (\et_1 - c \et_2)\right|
	+ \left|\int  (1-\chi_A^3) \sech(\gamma \ell x)\zA^{-2}  w_1 [\zA (z_1 - c z_2)] \right|
	\\
	\leq&~{}  \| \zA^2 v_1\|_{L^2} \|\et_1 - c \et_2\|_{L^2}
	+ \sech \left( \frac{\gamma \ell A}{4} \right)  \|w_1\|_{L^2} \|\zA (z_1 - c z_2)\|_{L^2}.
	\end{aligned}
	\]
	Therefore, applying Cauchy-Schwarz inequality together with \eqref{eq:sech_v1_w1}, \eqref{eq:z2_w2} and \eqref{eq:zA_z1},  we obtain that $\ell_1$ is bounded by
	 \begin{equation}\label{l1}
	 \begin{aligned}
	|\ell_1| 
	 \lesssim&~{}  \| w_1\|_{L^2} \big(\|\et_1+c\eta_2\|_{L^2} + \| \et_2\|_{L^2} \big)
	+ \sech \left( \frac{\gamma \ell A}{4} \right) \gar^{-1/2} \|w_1\|_{L^2} \big( \|w_1\|_{L^2}+\| \partial_x w_1\|_{L^2}+\|c w_1+w_2\|_{L^2} \big)
	\\
	 \lesssim&~{}  \| w_1\|_{L^2} \big(\|\et_1+c\eta_2\|_{L^2} + \| \et_2\|_{L^2} \big)
	+ \sech \left( \frac{\gamma \ell A}{4} \right) \gar^{-1/2} \big( \|w_1\|_{L^2}^2+\| \partial_x w_1\|_{L^2}^2+\|c w_1+w_2\|_{L^2}^2 \big).
	 \end{aligned}
	 \end{equation}
	 
	Now, we will deal with $\ell_2$. Integrating by parts, one gets
	 \[
	 \begin{aligned}
	\ell_2 =&~{}  -\gar \int  \sech(\gamma \ell x)  v_1 \partial_x^2(z_1 - c z_2)
	=  \gar \int  \partial_x[\sech(\gamma \ell x)  v_1 ](\partial_x z_1 - c \partial_x z_2)
	\\
	=&~{}  \gar \int  [  (\sech(\gamma \ell x) )' v_1+\sech(\gamma \ell x)  \partial_x v_1 ](\partial_x z_1 - c \partial_x z_2)
	\\
	=&~{}  \gar \int  [  (\sech(\gamma \ell x) )' v_1+\sech(\gamma \ell x)  \partial_x v_1 ] \chi_A^2 (\partial_x z_1 - c \partial_x z_2)\\
	&~{}+  \gar \int  [  (\sech(\gamma \ell x) )' v_1+\sech(\gamma \ell x)  \partial_x v_1 ] (1-\chi_A^2) (\partial_x z_1 - c \partial_x z_2) =: \ell_{2,1}+\ell_{2,2}.
	 \end{aligned}
	 \]
	 For the first integral in the RHS, using a similar decomposition that in $\ell_1$,  one gets
	 \[ 
	 \begin{aligned}
	| \ell_{2,1}| =&~{} \left| \gar \int  [  (\sech(\gamma \ell x) )' v_1+\sech(\gamma \ell x)  \partial_x v_1 ] \chi_A^2 (\partial_x z_1 - c \partial_x z_2)\right|
	 \\
	 \lesssim& ~{} \gar \big( \|\zA v_1 \|_{L^2} + \|\zA \partial_x v_1 \|_{L^2} \big)
	 \big(\|\chi_A \zB \partial_x z_1 \|_{L^2}+|c|\|\chi_A \zB   \partial_x z_2 \|_{L^2}\big).
	 \end{aligned}
	 \] 
	 Moreover, by \eqref{eq:sech_v1_w1}, \eqref{eq:CZzx}, \eqref{eq:zA_z1} and \eqref{eq:z2_w2}; and recalling \eqref{eq:B_gar}, we conclude
	 \begin{equation}\label{l21}
	 \begin{aligned}
	| \ell_{2,1}| 
	 \lesssim& ~{} B^{-4} \big( \|w_1 \|_{L^2}^2 + \| \partial_x w_1 \|_{L^2}^2 +\| cw_1+w_2\|_{L^2}^2\big)
	 \\&
	+B^{-4} \big(		\| \partial_x \et_1+c \partial_x \et_2\|_{L^2}^2+\| \et_1+c\eta_2 \|_{L^2}^2
	 			+\| \partial_x \et_2\|_{L^2}^2+\| \et_2\|_{L^2}^2
				\big)			
				.
	 \end{aligned}
	 \end{equation}
	 
	 Now, for $\ell_{2,2}$. Applying H\"older's inequality jointly to \eqref{eq:sech_v1_w1},  \eqref{eq:z2_w2}, and \eqref{eq:z1_w1}, we conclude
	 \begin{equation}\label{l22}
	 \begin{aligned}
	| \ell_{2,2}|
	 =& ~{} \left| \gar \int \zA [  (\sech(\gamma \ell x) )' v_1+\sech(\gamma \ell x)  \partial_x v_1 ] \zA^{-2}(1-\chi_A^2) \zA (\partial_x z_1 - c \partial_x z_2)\right|
	 \\
	 \lesssim&~{} \gar \sech \left( \frac{\gamma \ell A}{4} \right) \left( \|\zA v_1 \|_{L^2} + \|\zA \partial_x v_1 \|_{L^2}\right) 
	\|\zA \partial_x(z_1-c z_2 ) \|_{L^2} 
	\\
	 \lesssim&~{} \sech \left( \frac{\gamma \ell A}{4} \right) 
	  \left( \| w_1 \|_{L^2} + \| \partial_x w_1 \|_{L^2}\right) 
	 \left( \| w_1 \|_{L^2} + \| \partial_x w_1 \|_{L^2}+\|cw_1+w_2\|_{L^2}\right) .
	 \end{aligned}
	 \end{equation}
	Gathering estimates \eqref{l1}, \eqref{l21}, \eqref{l22}, \eqref{lem6p2}; and considering \eqref{eq:B_gar}, we get
	\[
	\begin{aligned}
	\int  \sech(\gamma \ell x)  v_1^2
		\lesssim~{}&
	\| w_1\|_{L^2} \big(\|\et_1+c\eta_2\|_{L^2} + \| \et_2\|_{L^2} \big)
	+ \| w_2+c w_1\|_{L^2} \| \et_2\|_{L^2}
	\\&
	+ \sech \left( \frac{\gamma \ell A}{4} \right)  \big( \|w_1\|_{L^2}^2+\| \partial_x w_1\|_{L^2}^2+\|c w_1+w_2\|_{L^2}^2 \big)
	\\&
	+B^{-4} \big( \|w_1 \|_{L^2}^2 + \| \partial_x w_1 \|_{L^2}^2 +\| cw_1+w_2\|_{L^2}^2\big)
	 \\&
	+B^{-4} \big(		\| \partial_x \et_1+c \partial_x \et_2\|_{L^2}^2+\| \et_1+c\eta_2 \|_{L^2}^2
	 			+\| \partial_x \et_2\|_{L^2}^2+\| \et_2\|_{L^2}^2
				\big)	
	\\&
	+\sech \left( \frac{\gamma \ell A}{4} \right) \left( \| w_1 \|_{L^2} + \| \partial_x w_1 \|_{L^2}\right) 
	 \left( \| w_1 \|_{L^2} + \| \partial_x w_1 \|_{L^2}+\|cw_1+w_2\|_{L^2}\right)
	 \\&
	 +B^{-4} \| w_2+c w_1\|_{L^2}^2
	+ B^{-4} \big( \| \partial_x^2 \et_2\|_{L^2}^2+\| \partial_x \et_2\|_{L^2}^2+\| \et_2\|_{L^2}^2 \big).
\end{aligned}
\]
Simplifying,
\[
\begin{aligned}
\int  \sech(\gamma \ell x)  v_1^2
	\lesssim~{}&
	\| w_1\|_{L^2} \big(\|\et_1+c\eta_2\|_{L^2} + \| \et_2\|_{L^2} \big)
	+ \| w_2+c w_1\|_{L^2} \| \et_2\|_{L^2}
	\\&
	+B^{-4} \big( \|w_1 \|_{L^2}^2 + \| \partial_x w_1 \|_{L^2}^2 +\| cw_1+w_2\|_{L^2}^2\big)
	 \\&
	+B^{-4} \big(		\| \partial_x \et_1+c \partial_x \et_2\|_{L^2}^2+\| \et_1+c\eta_2 \|_{L^2}^2
				+\| \partial_x^2 \et_2\|_{L^2}^2
	 			+\| \partial_x \et_2\|_{L^2}^2+\| \et_2\|_{L^2}^2
				\big)	.
	\end{aligned}
	\]
This concludes the proof of Lemma \ref{prop:coer_v1}.
	\end{proof}
	
From Lemmas \ref{prop:coer_v1} and \ref{prop:coer_v2}, we obtain that the modulation terms satisfies the following result
\begin{cor}
Let $c,\rho$ the modulation terms which holds \eqref{eq:c'}. It holds
\begin{equation}\label{eq:c'_wz}
\begin{aligned}
|c'|+|\rho'-c|^2 \lesssim&~{} 	 \big( \| w_1\|_{L^2} + \| w_2+c w_1\|_{L^2}\big)\big(\|\et_1+c\eta_2\|_{L^2} + \| \et_2\|_{L^2} \big)
	\\&
	+B^{-4} \big( \|w_1 \|_{L^2}^2 + \| \partial_x w_1 \|_{L^2}^2 +\| cw_1+w_2\|_{L^2}^2\big)
	 \\&
	+B^{-4} \big(		\| \partial_x \et_1+c \partial_x \et_2\|_{L^2}^2+\| \et_1+c\eta_2 \|_{L^2}^2
				+\| \partial_x^2 \et_2\|_{L^2}^2
	 			+\| \partial_x \et_2\|_{L^2}^2+\| \et_2\|_{L^2}^2
				\big)	.
\end{aligned}
\end{equation}
\end{cor}
\begin{proof}
Recalling \eqref{eq:c'}, \eqref{eq:c'_new} and using Lemmas \ref{prop:coer_v1} and \ref{prop:coer_v2},
\[
\begin{aligned}
|c'|+|\rho'-c|^2\lesssim ~{}&
 \left\| \Qc^{(p-1)/2} v_1\right\|_{L^2}^2 +\left\| \Qc^{3/4}v_1  \right\|_{L^2}^2+ \left\| \Qc^{3/4}(v_2+cv_1)  \right\|_{L^2}^2
 \\
 \lesssim~{} &
 \| w_1\|_{L^2} \big(\|\et_1+c\eta_2\|_{L^2} + \| \et_2\|_{L^2} \big)
	+ \| w_2+c w_1\|_{L^2} \| \et_2\|_{L^2}
	\\&
	+B^{-4} \big( \|w_1 \|_{L^2}^2 + \| \partial_x w_1 \|_{L^2}^2 +\| cw_1+w_2\|_{L^2}^2\big)
	 \\&
	+B^{-4} \big(		\| \partial_x \et_1+c \partial_x \et_2\|_{L^2}^2+\| \et_1+c\eta_2 \|_{L^2}^2
				+\| \partial_x^2 \et_2\|_{L^2}^2
	 			+\| \partial_x \et_2\|_{L^2}^2+\| \et_2\|_{L^2}^2
				\big)	
\\&
+ \| w_2+c w_1\|_{L^2} \| \et_2\|_{L^2}+ B^{-4} \| w_2+c w_1\|_{L^2}^2
	+ B^{-4} \big( \| \partial_x^2 \et_2\|_{L^2}^2+\| \partial_x \et_2\|_{L^2}^2+\| \et_2\|_{L^2}^2 \big)
\\
\lesssim~{}&	
	 \| w_1\|_{L^2} \big(\|\et_1+c\eta_2\|_{L^2} + \| \et_2\|_{L^2} \big)
	+ \| w_2+c w_1\|_{L^2} \| \et_2\|_{L^2}
	\\&
	+B^{-4} \big( \|w_1 \|_{L^2}^2 + \| \partial_x w_1 \|_{L^2}^2 +\| cw_1+w_2\|_{L^2}^2\big)
	 \\&
	+B^{-4} \big(		\| \partial_x \et_1+c \partial_x \et_2\|_{L^2}^2+\| \et_1+c\eta_2 \|_{L^2}^2
				+\| \partial_x^2 \et_2\|_{L^2}^2
	 			+\| \partial_x \et_2\|_{L^2}^2+\| \et_2\|_{L^2}^2
				\big)	
\\
\lesssim~{}&	
	 \big( \| w_1\|_{L^2} + \| w_2+c w_1\|_{L^2}\big)\big(\|\et_1+c\eta_2\|_{L^2} + \| \et_2\|_{L^2} \big)
	\\&
	+B^{-4} \big( \|w_1 \|_{L^2}^2 + \| \partial_x w_1 \|_{L^2}^2 +\| cw_1+w_2\|_{L^2}^2\big)
	 \\&
	+B^{-4} \big(		\| \partial_x \et_1+c \partial_x \et_2\|_{L^2}^2+\| \et_1+c\eta_2 \|_{L^2}^2
				+\| \partial_x^2 \et_2\|_{L^2}^2
	 			+\| \partial_x \et_2\|_{L^2}^2+\| \et_2\|_{L^2}^2
				\big)	
	. 
\end{aligned}
\]
This concludes the proof of the corollary.
\end{proof}	

	\subsection{Fourth coercivity estimate}
	Considering the orthogonality conditions from \eqref{eq:orto}, $z_i$ in \eqref{eq:z_cv} and $\eta_i$ from \eqref{eq:def_J} we obtain first
	\begin{equation}\label{ortoz1}
	\langle (1-\varepsilon \partial_x^{2}){\bf T}_c , \bd{z} \rangle =\langle (1-\varepsilon \partial_x^{2}){\bf T}_c , \Opg \bd{L}\bd{v} \rangle =\langle  \bd{L} \bd{Q}_c' , \bd{v} \rangle =0.
	\end{equation}
	Also, from Lemma \ref{lem2p2},
	\begin{equation}\label{ortoz2}
	 \langle (1-\varepsilon \partial_y^{2}) \Lambda \QQc , \bd{z} \rangle = \langle  \Lambda \QQc, \bd{L}\bd{v} \rangle = -\langle {\bf JQ}_c ,  \bd{v} \rangle =0.
	\end{equation}
	Recalling that $ \eta_i:=  \chi_A \zB z_i$, $i=1,2$, \eqref{ortoz1} and \eqref{ortoz2} are described as follows:
	\begin{equation}\label{ortoz3}
	\begin{aligned}
	&\left| \langle {\bf T}_c , \bd{\eta} \rangle \right|  \leq \varepsilon \left| \langle  \partial_y^{2}{\bf T}_c , \bd{\eta} \rangle \right| + \left|  \langle (1- \chi_A \zB) (1-\varepsilon \partial_y^{2}){\bf T}_c , \bd{z} \rangle \right| \lesssim \varepsilon  \|\bd{\eta}\|_{L^2\times H^1} +  e^{-m_0 B} \| \bd{z}\|_{L^2\times L^2},  
	\\ 
	& \left| \langle   \Lambda \QQc , \bd{\eta} \rangle \right| \lesssim \varepsilon  \left| \langle   \partial_y^{2}  \Lambda \QQc , \bd{\eta} \rangle \right| + \left|  \langle  (1- \chi_A \zB)(1-\varepsilon \partial_y^{2}) \Lambda \QQc , \bd{z} \rangle \right| \lesssim \varepsilon  \|\bd{\eta}\|_{L^2\times H^1} +e^{-m_0 B} \| \bd{z}\|_{L^2\times L^2}.
	\end{aligned}
	\end{equation}
Consider now the first line on the right of \eqref{dt_J}. For this line, we have the following coercivity estimate:
	
\begin{lem}
There exists $c_0>0$ such that 
\begin{equation}\label{coercoer}
\int \left( 3\eta _{2,y}^2 +(1-c^2)\eta_2^2 -f'(Q_c)\eta_2^2 +(\eta_1+c\eta_2)^2 \right)  \geq c_0 \| \bd{\eta} \|_{L^2 \times H^1}^2 - C e^{-2m_0 B} \| \bd{z}\|_{L^2\times L^2}^2 .
\end{equation}
\end{lem}

\begin{proof}

Without loss of generality, we assume that $\langle  \Lambda \QQc , \bd{\eta} \rangle   =0$ and prove that there exists $c_0>0$ such that
\begin{equation}\label{coercoer_new}
\int \left( 3\eta _{2,y}^2 +(1-c^2)\eta_2^2 -f'(Q_c)\eta_2^2 +(\eta_1+c\eta_2)^2 \right)  \geq c_0 \| \bd{\eta} \|_{L^2 \times H^1}^2. 
\end{equation}
By standard arguments, and \eqref{ortoz3}, one gets \eqref{coercoer} if $\varepsilon$ is fixed but sufficiently small.  Finally, \eqref{coercoer_new} is a consequence of Lemma \ref{lem2p3} and \eqref{coer0_new} with the role of $\eta_1$ and $\eta_2$ changed.

\end{proof}

\subsection{Improved bounds} Recall \eqref{cota_I}. Using the previous result, we are now ready to conclude a new bound on $\mathcal I'(t)$.
	\begin{cor}
	It holds
	\begin{equation}\label{eq:dtI_F}
	\begin{aligned}
\frac{d}{dt}\I
	\leq&~{}  -\frac14 \int  \left( (w_2 + c   w_1)^2+ 3 (\partial_x w_1)^2 +\big(1-c^2-A^{-1} \big) w_1^2 \right)
	 \\&
	+C_1B^{-2}  \big(		\| \partial_x \et_1+c \partial_x \et_2\|_{L^2}^2
				+\| \partial_x^2 \et_2\|_{L^2}^2
	 			+\| \partial_x \et_2\|_{L^2}^2
				\big)
	+C_1 B^{1/2}\big(\|\et_1+c\eta_2\|_{L^2}^2 + \| \et_2\|_{L^2}^2\big) 
.
	\end{aligned}
	\end{equation}
	\end{cor}
\begin{proof}
Recalling \eqref{cota_I}.  Firstly, we will focus on the terms
\[
\begin{aligned}
& I_{f1}= |\rho'-c| \left( \|w_2+cw_1\|_{L^2}  +\|w_1\|_{L^2}  \right) , \quad\quad
I_{f2}= |c'| \left(\frac{c}{\gamma^2}\|w_2+c w_1\|_{L^2}+\|w_1\|_{L^2} \right)  ,
\\
& I_{f3}=  \int \left(\sech^2\left(\frac{\gamma}{4}y\right) v_1^2 +\frac12f'(\Qc)w_1 \right).
\end{aligned}
\]
Let us focus on $I_{f1}$. Using \eqref{eq:c'_wz}, we obtain
\[
\begin{aligned}
I_{f1}\lesssim&~{} \left( \|w_2+cw_1\|_{L^2}  +\|w_1\|_{L^2}  \right) 
\bigg(
	 \big(\| w_1\|_{L^2}+\| w_2+c w_1\|_{L^2}\big) \big(\|\et_1+c\eta_2\|_{L^2} + \| \et_2\|_{L^2} \big)\bigg)^{1/2}
	\\&
	+B^{-2} \left( \|w_2+cw_1\|_{L^2}  +\|w_1\|_{L^2}  \right) \big( \|w_1 \|_{L^2} + \| \partial_x w_1 \|_{L^2} +\| cw_1+w_2\|_{L^2}\big)
	 \\&
	+B^{-2}\left( \|w_2+cw_1\|_{L^2}  +\|w_1\|_{L^2}  \right)  \big(		\| \partial_x \et_1+c \partial_x \et_2\|_{L^2}+\| \et_1+c\eta_2 \|_{L^2}
				+\| \partial_x^2 \et_2\|_{L^2}
	 			+\| \partial_x \et_2\|_{L^2}+\| \et_2\|_{L^2}
				\big),
\end{aligned}
\]
and by Young's inequality, we obtain
\[
\begin{aligned}
I_{f1}
\lesssim&~{} \left( \|w_2+cw_1\|_{L^2}  +\|w_1\|_{L^2}  \right) ^{3/2}
 \big(\|\et_1+c\eta_2\|_{L^2} + \| \et_2\|_{L^2} \big)^{1/2}
	\\&
	+B^{-2} \big( \|w_1 \|_{L^2}^2 + \| \partial_x w_1 \|_{L^2}^2 +\| cw_1+w_2\|_{L^2}^2\big)
	 \\&
	+B^{-2}  \big(		\| \partial_x \et_1+c \partial_x \et_2\|_{L^2}^2+\| \et_1+c\eta_2 \|_{L^2}^2
				+\| \partial_x^2 \et_2\|_{L^2}^2
	 			+\| \partial_x \et_2\|_{L^2}^2+\| \et_2\|_{L^2}^2
				\big)
\\
\lesssim&~{} B^{-1/2}\left( \|w_2+cw_1\|_{L^2}^2  +\|w_1\|_{L^2}^2  \right)
 +B^{1/2}\big(\|\et_1+c\eta_2\|_{L^2}^2 + \| \et_2\|_{L^2}^2\big) 
	\\&
	+B^{-2} \big( \|w_1 \|_{L^2}^2 + \| \partial_x w_1 \|_{L^2}^2 +\| cw_1+w_2\|_{L^2}^2\big)
	 \\&
	+B^{-2}  \big(		\| \partial_x \et_1+c \partial_x \et_2\|_{L^2}^2+\| \et_1+c\eta_2 \|_{L^2}^2
				+\| \partial_x^2 \et_2\|_{L^2}^2
	 			+\| \partial_x \et_2\|_{L^2}^2+\| \et_2\|_{L^2}^2
				\big).
\end{aligned}
\]
Therefore, $I_{f1}$ holds the following estimate
\begin{equation}\label{eq:If1}
\begin{aligned}
I_{f1}
\lesssim&~{} 
	B^{-1/2} \big( \|w_1 \|_{L^2}^2 + \| \partial_x w_1 \|_{L^2}^2 +\| cw_1+w_2\|_{L^2}^2\big)
	 \\&
	+B^{-2}  \big(		\| \partial_x \et_1+c \partial_x \et_2\|_{L^2}^2
				+\| \partial_x^2 \et_2\|_{L^2}^2
	 			+\| \partial_x \et_2\|_{L^2}^2
				\big)
	+B^{1/2}\big(\|\et_1+c\eta_2\|_{L^2}^2 + \| \et_2\|_{L^2}^2\big). 
\end{aligned}
\end{equation}

Let us focus on $I_{f2}$. Recalling \eqref{eq:c'_wz} and $B^{1/4}\gamma>2$ joint to $\|\bd{w}\|_{H^1\times L^2} \lesssim \|\bd{u} \|_{H^1\times L^2}\lesssim \delta$
\[
\begin{aligned}
I_{f2}\lesssim&~{}  |c'| \left(\frac{c}{\gamma^2}\|w_2+c w_1\|_{L^2}+\|w_1\|_{L^2} \right) \\
\lesssim&~{}  \delta B^{1/2} \big(
 \| w_1\|_{L^2}^2 + \| w_2+c w_1\|_{L^2}^2+ \|\et_1+c\eta_2\|_{L^2}^2 + \| \et_2\|_{L^2}^2 \big)
	\\&
	+\delta  B^{-3} \big( \|w_1 \|_{L^2}^2 + \| \partial_x w_1 \|_{L^2}^2 +\| cw_1+w_2\|_{L^2}^2\big)
	 \\&
	+\delta   B^{-3} \big(		\| \partial_x \et_1+c \partial_x \et_2\|_{L^2}^2+\| \et_1+c\eta_2 \|_{L^2}^2
				+\| \partial_x^2 \et_2\|_{L^2}^2
	 			+\| \partial_x \et_2\|_{L^2}^2+\| \et_2\|_{L^2}^2
				\big).
\end{aligned}
\]
We conclude that $I_{f2}$, satisfies the following estimate
\begin{equation}\label{eq:If2}
\begin{aligned}
I_{f2}
\lesssim&~{}  \delta  \big(
 \| w_1\|_{L^2}^2 + \| \partial_x w_1 \|_{L^2}^2 + \| w_2+c w_1\|_{L^2}^2+ \|\et_1+c\eta_2\|_{L^2}^2 +\| \partial_x \et_2\|_{L^2}^2+ \| \et_2\|_{L^2}^2 \big)
	 \\&
	+\delta   B^{-4} \big(		\| \partial_x \et_1+c \partial_x \et_2\|_{L^2}^2
				+\| \partial_x^2 \et_2\|_{L^2}^2
				\big)	.
\end{aligned}
\end{equation}

Noticing that the bound of $I_{f3}$ is obtained by direct application of Lemma \ref{prop:coer_v1}, replacing \eqref{eq:If1} and \eqref{eq:If2}; and using Cauchy-Schwarz inequality, we obtain
\[ 
\begin{aligned}
\frac{d}{dt}\I 
\leq&~{}  -\frac12 \int  \left( (w_2 + c   w_1)^2+ 3 (\partial_x w_1)^2 +\big(1-c^2-A^{-1} \big) w_1^2 \right)
\\&
	+ C_1B^{-1/2} \big( \|w_1 \|_{L^2}^2 + \| \partial_x w_1 \|_{L^2}^2 +\| cw_1+w_2\|_{L^2}^2\big)
	 \\&
	+ C_1B^{-2}  \big(		\| \partial_x \et_1+c \partial_x \et_2\|_{L^2}^2
				+\| \partial_x^2 \et_2\|_{L^2}^2
	 			+\| \partial_x \et_2\|_{L^2}^2
				\big)
	+ C_1B^{1/2}\big(\|\et_1+c\eta_2\|_{L^2}^2 + \| \et_2\|_{L^2}^2\big) 
\\&
	+ C_1 \delta  \big(
 \| w_1\|_{L^2}^2 + \| \partial_x w_1 \|_{L^2}^2 + \| w_2+c w_1\|_{L^2}^2+ \|\et_1+c\eta_2\|_{L^2}^2 +\| \partial_x \et_2\|_{L^2}^2+ \| \et_2\|_{L^2}^2 \big)
	 \\&
	+ C_1\delta   B^{-4} \big(		\| \partial_x \et_1+c \partial_x \et_2\|_{L^2}^2
				+\| \partial_x^2 \et_2\|_{L^2}^2
				\big)	
	\\&
	+C_1 \| w_1\|_{L^2} \big(\|\et_1+c\eta_2\|_{L^2} + \| \et_2\|_{L^2} \big)
	+C_1  \| w_2+c w_1\|_{L^2} \| \et_2\|_{L^2}
	\\&
	+C_1 B^{-4} \big( \|w_1 \|_{L^2}^2 + \| \partial_x w_1 \|_{L^2}^2 +\| cw_1+w_2\|_{L^2}^2\big)
	 \\&
	+C_1 B^{-4} \big(		\| \partial_x \et_1+c \partial_x \et_2\|_{L^2}^2+\| \et_1+c\eta_2 \|_{L^2}^2
				+\| \partial_x^2 \et_2\|_{L^2}^2
	 			+\| \partial_x \et_2\|_{L^2}^2+\| \et_2\|_{L^2}^2
				\big)	.
\end{aligned}
\]
Arranging terms,
\[
\begin{aligned}
\frac{d}{dt}\I
	\leq&~{}  -\frac12 \int  \left( (w_2 + c   w_1)^2+ 3 (\partial_x w_1)^2 +\big(1-c^2-A^{-1} \big) w_1^2 \right)
\\&
	+C_1 B^{-1/2} \big( \|w_1 \|_{L^2}^2 + \| \partial_x w_1 \|_{L^2}^2 +\| cw_1+w_2\|_{L^2}^2\big)
	 \\&
	+C_1B^{-2}  \big(		\| \partial_x \et_1+c \partial_x \et_2\|_{L^2}^2
				+\| \partial_x^2 \et_2\|_{L^2}^2
	 			+\| \partial_x \et_2\|_{L^2}^2
				\big)
	+C_1 B^{1/2}\big(\|\et_1+c\eta_2\|_{L^2}^2 + \| \et_2\|_{L^2}^2\big) 
.
	\end{aligned}
\] 
Finally, using that $B$ is large and $1-c^2>0$ uniformly in time, we conclude the proof.
\end{proof}

Recall \eqref{dt_J}. Using the previous result, we are now ready to conclude a new bound on $\mathcal J'(t)$.

	\begin{cor}
	It holds the improved estimate
	\begin{equation}\label{dt_J_F}
	\begin{aligned}
	\frac{d}{dt}\J
	\leq&~{}  
-\frac14  \int  \left( (\et_1+c \et_2)^2+3 (\partial_x \et_2)^2+\big(1-c^2\big)\et_2^2 -f'(\Qc) \et_2^2\right)
\\
&
+C_2 B^{-1} \big( \|w_1\|_{L^2}^2+\|\partial_x w_1\|_{L^2}^2+\|cw_1+w_2\|_{L^2}^2 \big)
+C_2 A^{-1} \big(		\| \partial_x \et_1+c \partial_x \et_2\|_{L^2}^2
				+\| \partial_x^2 \et_2\|_{L^2}^2
				\big)	.
	\end{aligned}
	\end{equation}
	\end{cor}
\begin{proof}
Recalling \eqref{eq:dJ_lema}. Firstly, we will focus on the last term on the RHS in \eqref{eq:dJ_lema} and replacing \eqref{eq:c'_wz}, one gets
\[
\begin{aligned}
 |c'| &\big(\|  \et_2\|_{L^2}+ \|  \et_1+c\eta_2\|_{L^2} \big)
 \\
 \lesssim&~{} 	\big( \| w_1\|_{L^2} \big(\|\et_1+c\eta_2\|_{L^2} + \| \et_2\|_{L^2} \big)
	+ \| w_2+c w_1\|_{L^2} \| \et_2\|_{L^2}\big)\big(\|  \et_2\|_{L^2}+ \|  \et_1+c\eta_2\|_{L^2} \big)
	\\&
	+B^{-4}\big(\|  \et_2\|_{L^2}+ \|  \et_1+c\eta_2\|_{L^2} \big) \big( \|w_1 \|_{L^2}^2 + \| \partial_x w_1 \|_{L^2}^2 +\| cw_1+w_2\|_{L^2}^2\big)
	 \\&
	+B^{-4}\big(\|  \et_2\|_{L^2}+ \|  \et_1+c\eta_2\|_{L^2} \big) \big(		\| \partial_x \et_1+c \partial_x \et_2\|_{L^2}^2+\| \et_1+c\eta_2 \|_{L^2}^2
				+\| \partial_x^2 \et_2\|_{L^2}^2
	 			+\| \partial_x \et_2\|_{L^2}^2+\| \et_2\|_{L^2}^2
				\big)	
\end{aligned}
\]
Noticing that $\|  \et_2\|_{L^2}+ \|  \et_1+c\eta_2\|_{L^2}\leq \delta\gar^{-1/2}$. Indeed,  recalling that that $\eta_2=\chi_A\zB z_2$ and $\eta_1+c\eta_2=\chi_A\zB (z_1+cz_2)$  jointly to \eqref{eq:z2_w2} and \eqref{eq:zA_z1}, we obtain
\[
\|  \et_2\|_{L^2}+ \|  \et_1+c\eta_2\|_{L^2}\lesssim  \|w_1 \|_{L^2} +\gar^{-1/2} \| \partial_x w_1 \|_{L^2} +\| cw_1+w_2\|_{L^2},
\]
we conclude by \eqref{OS} and  \eqref{eq:B_gar}.

Applying Cauchy-Schwarz inequality jointly the above estimates; and having in mind \eqref{eq:A} and \eqref{eq:B_gar}, we have
\[
\begin{aligned}
 |c'| \big(\|  \et_2\|_{L^2}+ \|  \et_1+c\eta_2\|_{L^2} \big)
 \lesssim&~{} 	\big( \| w_1\|_{L^2} +\|w_2+cw_1\|_{L^2}\big)\big(\|\et_1+c\eta_2\|_{L^2}^2 + \| \et_2\|_{L^2}^2 \big)
	\\&
	+B^{-2}A^{-1} \big( \|w_1 \|_{L^2}^2 + \| \partial_x w_1 \|_{L^2}^2 +\| cw_1+w_2\|_{L^2}^2\big)
	 \\&
	+B^{-2}A^{-1} \big(		\| \partial_x \et_1+c \partial_x \et_2\|_{L^2}^2+\| \et_1+c\eta_2 \|_{L^2}^2
				+\| \partial_x^2 \et_2\|_{L^2}^2
	 			+\| \partial_x \et_2\|_{L^2}^2+\| \et_2\|_{L^2}^2
				\big)	
\\
 \lesssim&~{} 	
	A^{-1} \big(		\| \partial_x \et_1+c \partial_x \et_2\|_{L^2}^2+\| \et_1+c\eta_2 \|_{L^2}^2
				+\| \partial_x^2 \et_2\|_{L^2}^2
	 			+\| \partial_x \et_2\|_{L^2}^2+\| \et_2\|_{L^2}^2
				\big)	
	\\&+B^{-2}A^{-1} \big( \|w_1 \|_{L^2}^2 + \| \partial_x w_1 \|_{L^2}^2 +\| cw_1+w_2\|_{L^2}^2\big).
\end{aligned}
\]
Replacing th estimates in \eqref{eq:dJ_lema},
\[
\begin{aligned}
\frac{d}{dt} \J \leq&
-\frac12  \int  \left( (\et_1+c \et_2)^2+3 (\partial_x \et_2)^2+\big(1-c^2\big)\et_2^2 -f'(\Qc) \et_2^2\right)
\\
&+ C_2 B^{-1} \left( \|\et_1+c\eta_2\|_{L^2}^2+  \| \et_2\|_{L^2}^2+\| \partial_x \et_2\|_{L^2}^2 \right)
+C_2 B^{-1} \big( \|w_1\|_{L^2}^2+\|\partial_x w_1\|_{L^2}^2+\|cw_1+w_2\|_{L^2}^2 \big)
\\&
+C_2 A^{-1} \big(		\| \partial_x \et_1+c \partial_x \et_2\|_{L^2}^2+\| \et_1+c\eta_2 \|_{L^2}^2
				+\| \partial_x^2 \et_2\|_{L^2}^2
	 			+\| \partial_x \et_2\|_{L^2}^2+\| \et_2\|_{L^2}^2
				\big)	
\\&+C_2 B^{-2}A^{-1} \big( \|w_1 \|_{L^2}^2 + \| \partial_x w_1 \|_{L^2}^2 +\| cw_1+w_2\|_{L^2}^2\big).
\end{aligned}
\]
Arranging terms,
\[
\begin{aligned}
\frac{d}{dt} \J \leq&
-\frac12  \int  \left( (\et_1+c \et_2)^2+3 (\partial_x \et_2)^2+\big(1-c^2\big)\et_2^2 -f'(\Qc) \et_2^2\right)
\\
&+ C_2 B^{-1} \left( \|\et_1+c\eta_2\|_{L^2}^2+  \| \et_2\|_{L^2}^2+\| \partial_x \et_2\|_{L^2}^2 \right)
+C_2 B^{-1} \big( \|w_1\|_{L^2}^2+\|\partial_x w_1\|_{L^2}^2+\|cw_1+w_2\|_{L^2}^2 \big)
\\&
+C_2 A^{-1} \big(		\| \partial_x \et_1+c \partial_x \et_2\|_{L^2}^2
				+\| \partial_x^2 \et_2\|_{L^2}^2
				\big)	
\\
 \leq&
-\frac14  \int  \left( (\et_1+c \et_2)^2+3 (\partial_x \et_2)^2+\big(1-c^2\big)\et_2^2 -f'(\Qc) \et_2^2\right)
\\
&+C_2 B^{-1} \big( \|w_1\|_{L^2}^2+\|\partial_x w_1\|_{L^2}^2+\|cw_1+w_2\|_{L^2}^2 \big)
+C_2 A^{-1} \big(		\| \partial_x \et_1+c \partial_x \et_2\|_{L^2}^2
				+\| \partial_x^2 \et_2\|_{L^2}^2
				\big)	.
\end{aligned}
\]
This concludes the proof of this corollary.
\end{proof}

Finally, recall \eqref{Est_N}. Using the previous result, we are now ready to conclude a new bound on $\mathcal N'(t)$.
	\begin{cor}
	It holds
	\begin{equation}\label{Est_N_F}
	\begin{aligned}
	\frac{d}{dt}\N
	\leq&~{}  
	-\frac18 \int \left( 
				(\partial_x \et_1+c\partial_x \et_2)^2 
				+3(\partial_x^2 \et_2)^2
			 \right)
+C_3 \big( \| \et_2\|_{L^2}^2+\| \partial_x \et_2\|_{L^2}^2\big)
\\&
+	C_3 B^{-1} \| \et_1+c\eta_2\|_{L^2}^2
	+ C_3B^{-1}\big( 
					\|  w_1\|_{L^2}^2 
					+\|  \partial_x w_1\|_{L^2}^2
					+ \|cw_1+w_2 \|_{L^2}^2
		\big).
	\end{aligned}
	\end{equation}
	\end{cor}
\begin{proof}
Firstly, we will focus on the las term on the RHS of \eqref{Est_N}. Replacing \eqref{eq:c'_wz} and using that $\|\bd{w}\|_{H^1\times L^2}\lesssim \|\bd{u}\|_{H^1\times L^2}\leq \delta$,  we obtain
\[
\begin{aligned}
| c'| &(\|w_1\|_{L^2}+\|\partial_x w_1\|_{L^2}+\|cw_1+w_2\|_{L^2}) \\
\lesssim &
~{}\delta \| w_1\|_{L^2} \big(\|\et_1+c\eta_2\|_{L^2} + \| \et_2\|_{L^2} \big)
	+ \| w_2+c w_1\|_{L^2} \| \et_2\|_{L^2}
	\\&
	+\delta B^{-4} \big( \|w_1 \|_{L^2}^2 + \| \partial_x w_1 \|_{L^2}^2 +\| cw_1+w_2\|_{L^2}^2\big)
	 \\&
	+\delta B^{-4} \big(		\| \partial_x \et_1+c \partial_x \et_2\|_{L^2}^2+\| \et_1+c\eta_2 \|_{L^2}^2
				+\| \partial_x^2 \et_2\|_{L^2}^2
	 			+\| \partial_x \et_2\|_{L^2}^2+\| \et_2\|_{L^2}^2
				\big)
\\
\lesssim &~{}
\delta  \big(\|\et_1+c\eta_2\|_{L^2}^2 + \| \et_2\|_{L^2}^2\big)
	\\&
	+\delta \big( \|w_1 \|_{L^2}^2 + \| \partial_x w_1 \|_{L^2}^2 +\| cw_1+w_2\|_{L^2}^2\big)
	+\delta B^{-4} \big(	\| \partial_x \et_1+c \partial_x \et_2\|_{L^2}^2
				+\| \partial_x^2 \et_2\|_{L^2}^2
	 			+\| \partial_x \et_2\|_{L^2}^2
				\big).
\end{aligned}
\]
Finally, using the above estimate in \eqref{Est_N} and having in mind \eqref{eq:A}, we conclude
\begin{equation*}
\begin{aligned}
	\dfrac{d}{dt} \N
\leq~{}&
	-\frac18 \int \left( 
				(\partial_x \et_1+c\partial_x \et_2)^2 
				+3(\partial_x^2 \et_2)^2
			 \right)
+C_3 \big( \| \et_2\|_{L^2}^2+\| \partial_x \et_2\|_{L^2}^2\big)
\\&
+	C_3 B^{-1} \| \et_1+c\eta_2\|_{L^2}^2
	+ C_3B^{-1}\big( 
					\|  w_1\|_{L^2}^2 
					+\|  \partial_x w_1\|_{L^2}^2
					+ \|cw_1+w_2 \|_{L^2}^2
		\big).
\end{aligned}
\end{equation*}
The proof is complete.
\end{proof}

Now we gather \eqref{eq:dtI_F}, \eqref{dt_J_F} and \eqref{Est_N_F} to conclude the proof of Theorem \ref{MT}.

\section{Proof of the Main Theorem}\label{Sec:7}

Finally, we prove Theorem \ref{MT}.

\subsection{Adding virial estimates}	Recall \eqref{eq:dtI_F}, \eqref{dt_J_F} and \eqref{Est_N_F}. Let us define
\[
\mathcal H:= 
\mathcal{J}+16 \delta^{3/40}C_2\mathcal{I} +40 C_1 C_2 \delta^{1/8} \mathcal{N}.
\]
We have
\begin{equation}\label{eq:dtH_inequality}
\begin{aligned}
\frac{d}{dt}\mathcal H(t) \le &~{} -C_2 \delta^{3/40} \int  \left( (w_2 + c   w_1)^2+  (\partial_x w_1)^2 +\big(1-c^2 \big) w_1^2 \right) \\
&~{} -4C_1 C_2 \delta^{1/8}  \int \left( 
(\et_1+c \et_2)^2+ (\partial_x \et_2)^2+\big(1-c^2\big)\et_2^2  + (\partial_x \et_1+c\partial_x \et_2)^2 
				+(\partial_x^2 \et_2)^2
			 \right).
\end{aligned}
\end{equation}
Indeed, gathering \eqref{dt_J_F}, \eqref{eq:dtI_F} and \eqref{Est_N_F}, we obtain
	\begin{equation*}
	\begin{aligned}
	\frac{d}{dt}\mathcal{H}
	\leq&~{}  
-\frac14  \int  \left( (\et_1+c \et_2)^2+3 (\partial_x \et_2)^2+\big(1-c^2\big)\et_2^2 -f'(\Qc) \et_2^2\right)
	 \\&
	 + (16 C_1 C_2 \delta^{1/20} +40C_1 C_2 C_3\delta^{1/8}) \big(\|\et_1+c\eta_2\|_{L^2}^2 + \| \et_2\|_{L^2}^2+\| \partial_x \et_2\|_{L^2}^2\big) 
\\&
	  -4\delta^{3/40}C_2  \int  \left( (w_2 + c   w_1)^2+ 3 (\partial_x w_1)^2 +\big(1-c^2 \big) w_1^2 \right)
	\\&
	 +16  C_1 C_2 \delta^{1/8} \big( \|w_1 \|_{L^2}^2 + \| \partial_x w_1 \|_{L^2}^2 +\| cw_1+w_2\|_{L^2}^2\big)
	\\&
	+ 40 C_1 C_2  C_3\delta^{1/10}\delta^{1/8}\big( 
					\|  w_1\|_{L^2}^2 
					+\|  \partial_x w_1\|_{L^2}^2
					+ \|cw_1+w_2 \|_{L^2}^2\big)
\\&
	-5C_1 C_2 \delta^{1/8}  \int \left( 
				(\partial_x \et_1+c\partial_x \et_2)^2 
				+3(\partial_x^2 \et_2)^2
			 \right)
	\\&
+C_2 \delta \big(		\| \partial_x \et_1+c \partial_x \et_2\|_{L^2}^2
				+\| \partial_x^2 \et_2\|_{L^2}^2
				\big)	
	+16  C_1 C_2 \delta^{11/40}  \big(		\| \partial_x \et_1+c \partial_x \et_2\|_{L^2}^2
				+\| \partial_x^2 \et_2\|_{L^2}^2
				\big).
	\end{aligned}
	\end{equation*}
Finally, arranging terms,
\begin{equation*}
\begin{aligned}	
\frac{d}{dt}\mathcal{H}
\leq&
-\frac14  \int  \left( (\et_1+c \et_2)^2+3 (\partial_x \et_2)^2+\big(1-c^2\big)\et_2^2 -f'(\Qc) \et_2^2\right)
	 \\&
	 +8 C_1 C_2 \delta^{1/20} ( 2+5 C_3\delta^{3/20} )\big(\|\et_1+c\eta_2\|_{L^2}^2 + \| \et_2\|_{L^2}^2++\| \partial_x \et_2\|_{L^2}^2\big) 
\\&
	  -8\delta^{3/40}C_2  \int  \left( (w_2 + c   w_1)^2+ 3 (\partial_x w_1)^2 +\big(1-c^2 \big) w_1^2 \right)
 \\&
+C_2\delta^{1/10} (1+16  C_1  \delta^{1/40}+ 40 C_1 C_3 \delta^{1/8} ) \big( \|w_1\|_{L^2}^2+\|\partial_x w_1\|_{L^2}^2+\|cw_1+w_2\|_{L^2}^2 \big)
\\&
	-5C_1 C_2 \delta^{1/8}  \int \left( 
				(\partial_x \et_1+c\partial_x \et_2)^2 
				+3(\partial_x^2 \et_2)^2
			 \right)
	\\&
+ C_2\delta^{11/40} ( 16  C_1+\delta^{29/40}  ) \big(		\| \partial_x \et_1+c \partial_x \et_2\|_{L^2}^2
				+\| \partial_x^2 \et_2\|_{L^2}^2
				\big).	
	\end{aligned}
	\end{equation*}
Using that $\delta^{1/40}\|w_1\|_{L^2}^2\leq (1-c^2)\|w_1\|_{L^2}^2$ along with \eqref{coercoer}, the previous estimates lead to \eqref{eq:dtH_inequality}.

\subsection{End of proof}
First of all, one has
\begin{equation*}
\int_{0}^\infty \left[\| w_2+c w_1 \|_{L^2}^2  +\|\partial_x w_1\|_{L^2}^2 
+(1-c^2)\|w_1\|_{L^2}^2  \right]dt\lesssim  \delta.
\end{equation*}
Indeed, integrating the estimate \eqref{eq:dtH_inequality} on $[0,t]$, we get
\[
\int_{0}^t \left[\| w_2+c w_1 \|_{L^2}^2  +\|\partial_x w_1\|_{L^2}^2 
+(1-c^2)\|w_1\|_{L^2}^2  \right]dt\lesssim 2\sup_{0\leq s\leq t} |\mathcal H(s)| \leq 2\sup_{s\geq 0} |\mathcal H(s)|.
\]
By Lemma \ref{lem:z_w}, estimates \eqref{eq:sech_v1_w1}-\eqref{eq:sech_v2_w2}, one obtains
\begin{equation}\label{eq:int_a1_a2_norm}
\int_{0}^\infty \left(
\int ((1-c^2)v_1^2+(\partial_x v_1)^2+(v_2+c v_1)^2 )\mu(y) \right) dt\leq \delta.
\end{equation}
Using the above estimate, we will conclude the proof of Theorem \ref{MT}. Let $\mu=\sech^2$, and
\begin{equation*}
\begin{gathered}
\ca{P}(t)=\int \mu(y) v_1^2+\int \mu(y) (\IOpg (v_2+cv_1))^2=:\ca{P}_1+\ca{P}_2.
\end{gathered}
\end{equation*}
For $\mathcal{K}_1$, using \eqref{eq:u_TD} and integrating by parts, we have
\[
\begin{aligned}
\frac{d \ca{P}_1}{dt}
=&~{} 2\int \mu(v_1 \dot{v_1}) 
\\
=&~{} 2\int \mu v_1 \big( \partial_y (v_2+c v_1) +(\rho'-c)\partial_y (\Qc+v_1)-c'~\Lambda (\Qc)\big)
\\
=&~{} 2\int \mu v_1 \partial_x (v_2+c v_1) 
+2(\rho'-c) \int \mu(y) v_1\partial_x (\Qc+v_1) -c'~2\int \mu v_1 \Lambda (\Qc)
\end{aligned}
\]
Then, applying H\"older inequality, Cauchy-Schwarz inequality, \eqref{eq:c'_new}  and $B^{1/4}\gamma>2$, we get
\[
\begin{aligned}
\left|\frac{d}{dt}\ca{P}_1(t)\right|
\lesssim & ~{}\int \mu(y)(v_1^2+(\partial_x v_1)^2+ (v_2+cv_1)^2) +|\rho'-c|^2+|c'|^2
\\
\lesssim & ~{}\int \mu(y)((1-c^2)v_1^2+(\partial_x v_1)^2+ (v_2+cv_1)^2) 
\end{aligned}
\]
For $\ca{P}_2$, passing to the variables $(z_1,z_2)$ (see \eqref{eq:z_cv}) 
\begin{equation*}
\begin{gathered}
\mathcal{P}_2=\int \mu(y)  z_2^2,
\end{gathered}
\end{equation*}
 and using \eqref{eq:z}, we get
\[
\begin{aligned}
\frac{d}{dt}{\ca{P}}_2
= &~{} 2\int \mu(y)  z_2 \dot{z_2}
=  2\int \mu(y)  z_2\big(  \partial_x (z_1+cz_2)+M_2+M_{2,c}\big)
\\
= &~{} 2\int \mu  z_2  \partial_x (z_1+cz_2)
+2\int \mu  z_2 M_2
+2\int \mu  z_2M_{2,c}
=: P_{21}+P_{22} +P_{23}.
\end{aligned}
\]
Integrating by parts in $P_{21}$, we have
\[
\begin{aligned}
P_{21}
=- 2\int (\mu'(y)  z_2+\mu(y)  \partial_y z_2  )(z_1+cz_2).
\end{aligned}
\]
Using H\"older's inequality, we obtain 
\[
\begin{aligned}
|P_{21}|
\lesssim&~{}\big( \|\mu^{1/2} z_2\|_{L^2}+\|\mu^{1/2} \partial_yz_2\|_{L^2}\big) \|\mu^{1/2}(z_1+cz_2)\|_{L^2}
\\
\lesssim&~{}\big( \|\mu^{1/2} z_2\|_{L^2}+\|\mu^{1/2} \partial_yz_2\|_{L^2}\big) \|\mu^{1/2}(z_1-cz_2+2cz_2)\|_{L^2}
.
\end{aligned}
\]
Recalling \eqref{def_z} and \eqref{eq:L_N}; using that $\left\| \sech\left( K y\right) (1-\gar\partial_x^2)^{-1} g\right\|_{L^2}\leq C \left\| (1-\gar\partial_x^2)^{-1}\left[ \sech\left( K y\right)  g\right]\right\|_{L^2}$, $K\leq A$,  we obtain
\[
\begin{aligned}
|P_{21}|
\lesssim&~{} \gar^{-1/2}\|\mu^{1/2} (cv_1+v_2)\|_{L^2} \big( \gar^{-1}\|\mu^{1/2} v_1\|_{L^2}+\|\mu^{1/2} (cv_1+v_2)\|_{L^2} \big)
\\
\lesssim&~{} \gar^{-3/2}\|\mu^{1/2} (cv_1+v_2)\|_{L^2} \big( \|\mu^{1/2} v_1\|_{L^2}+\|\mu^{1/2} (cv_1+v_2)\|_{L^2} \big)
.
\end{aligned}
\]
By Cauchy-Schwarz, we conclude
\[
\begin{aligned}
|P_{21}|
\lesssim_{\gar}&~{} \|\mu^{1/2} (cv_1+v_2)\|_{L^2}^2 + \|\mu^{1/2} v_1\|_{L^2}^2
.
\end{aligned}
\]

For $P_{22}$, we use Cauchy-Schwarz inequality jointly to  \eqref{eq:sech_Opg} and  \eqref{eq:sech_Opg_p}; then
\[
\begin{aligned}
|P_{22}| 
&
\lesssim~{}\| \mu^{1/2}  z_2\|_{L^2}\| \mu^{1/2}\Opg  N(v_1)\|_{L^2} 
\lesssim~{} \gar^{-1/2} \| \mu^{1/2} (cv_1+v_2)\|_{L^2} \| \mu^{1/2} v_1\|_{L^2} 
\\&
\lesssim_{\gar}~{} \| \mu^{1/2} (cv_1+v_2)\|_{L^2}^2 +\| \mu^{1/2} v_1\|_{L^2}^2
.
\end{aligned}
\]
The term $P_{23}$ is bounded as follows:
\[
\begin{aligned} 
|P_{23}|
=~{}& \left| 2\int \mu  z_2 M_{2,c}\right|
\lesssim~{}\|\mu^{1/2}z_2\|_{L^2}\|\mu^{1/2}(c' \Opg (\Qc+v_1) +(\rho'-c)  \partial_x z_2)\|_{L^2}
\\
\lesssim&~{}\|\mu^{1/2}z_2\|_{L^2} \big( |c'| \|\mu^{1/2} \Opg (\Qc+v_1)\|_{L^2}+ |\rho'-c| \|\mu^{1/2} \partial_x z_2\|_{L^2}\big)
\\
\lesssim&~{}\|\mu^{1/2}z_2\|_{L^2}^2+\|\mu^{1/2}\partial_x z_2\|_{L^2}^2+ |c'|^2+ \|\mu^{1/2} \Opg v_1\|_{L^2}^2
\\
\lesssim_{\gar}&~{} \|\mu^{1/2}(cv_1+v_2)\|_{L^2}^2+ \|\mu^{1/2} v_1\|_{L^2}^2+ |c'|^2.
\end{aligned}
\]
Then, we conclude
\[
\begin{aligned}
\left|\frac{d}{dt}{\mathcal{P}}_2(t)\right|
\lesssim_{\gar}& ~{}  \int  \mu(y)(v_1^2+(\partial_x v_1)^2+ (v_2+c v_1)^2).
\end{aligned}
\]
By \eqref{eq:int_a1_a2_norm}, there exists and increasing sequence $t_n\to \infty$ such that
\[
\lim_{n\to \infty} \left[\mathcal{P}_1(t_n)+\mathcal{P}_2(t_n)\right]=0.
\]
For $t\geq 0$, integrating on $[t,t_n]$, and passing to the limit as $n\to \infty$, we obtain
\[
\mathcal{P}(t)\lesssim \int_{t}^\infty \left[ \int \sech(y)(v_1^2+(\partial_x v_1)^2+ (v_2+cv_1)^2)\right]dt.
\]
By \eqref{eq:int_a1_a2_norm}, we deduce 
\begin{equation}\label{limP}
\lim_{t\to\infty} \mathcal{P}(t)=0.
\end{equation}
By the decomposition of solution, and the boundedness in $H^1$ of $u_1$, this implies $\lim_{t\to +\infty} \| \phi_1(t,\cdot + \rho(t))- Q_{c(t)}\|_{(L^2\cap L^\infty)(I)}=0$ for any bounded interval $I$.

\subsection{Convergence of the scaling parameter} Now we use \eqref{eq:c'_wz} again to get
\[ \label{eq:c'_wz}
\begin{aligned}
|c'|  \lesssim&~{}   \|w_1 \|_{L^2}^2 + \| \partial_x w_1 \|_{L^2}^2 +\| cw_1+w_2\|_{L^2}^2 \\
&~{} + \| \partial_x \et_1+c \partial_x \et_2\|_{L^2}^2+\| \et_1+c\eta_2 \|_{L^2}^2 +\| \partial_x^2 \et_2\|_{L^2}^2 +\| \partial_x \et_2\|_{L^2}^2+\| \et_2\|_{L^2}^2 =: \mathcal R.
\end{aligned}
\] 
From \eqref{eq:dtH_inequality} we get, for $0<t_1<t_2$,
\[
|c(t_1)-c(t_2)|  \leq \int_{t_1}^{t_2} \mathcal R \leq \int_{t_1}^{\infty}\mathcal R \leq  C_\delta <+\infty.
\]
Therefore $c(t)$ is a Cauchy sequence and converges to some $c_+$ satisfying (Theorem \ref{MT1}) $|c - c_+ | \lesssim C_0\delta.$ This ends the proof of \eqref{AS} and Theorem \ref{MT}.

\medskip

Finally, the proof of \eqref{AS2} follows directly from \eqref{limP}, interpolation and the boundedness of the $L^2$-norm of $\phi_2(t,\cdot + \rho(t))+c_+ Q_{c_+}$.

\appendix

\section{Proof of virial identities}\label{virial2}

In this section we prove the three virial identities used along this work.

\subsection{First virial identity} 
Proof of \eqref{eq:dtI}.   From \eqref{eq:I} and using \eqref{eq:eq_lin},
	\[ 
	\begin{aligned}
	\dfrac{d}{dt}\I
	=&\int \vA(\dot{v_1}v_2+v_1\dot{v_2}) \\
	=&\int \vA v_2  \partial_x (v_2+c v_1) 
	+\int \vA v_2 ( (\rho'-c)\partial_x (\Qc+v_1)-c'~\Lambda (\Qc))
	\\&
	+\int \vA  v_1 ( \partial_x \LL v_1 +c\partial_x (v_2+cv_1))
	+\int \vA  v_1(N+(\rho'-c)\partial_x(-c\Qc+v_2)+c'~\Lambda (c\Qc)).
	\end{aligned}
	\] 
Rearranging the terms, we obtain	
\begin{equation}\label{eq:I'_RHS}
\begin{aligned}
	\dfrac{d}{dt}\I
	=&\int \vA v_2  \partial_x v_2 
	+\int \vA  v_1  \partial_x (\LL v_1 +c^2 v_1)
	+c \int \vA [ v_2  \partial_x  v_1+   v_1 \partial_x v_2]
	+\int \vA v_2 ( (\rho'-c)\partial_x (\Qc)-c'~\Lambda (\Qc))
	\\&
	+\int \vA  v_1(N+(\rho'-c)\partial_x(-c\Qc)+c'~\Lambda (c\Qc)) 
	+(\rho'-c) \int \vA  (  v_2 \partial_x v_1+ v_1\partial_x v_2).
	\end{aligned}
	\end{equation}
For the third integral  and the last integral in the RHS of the above equation, we have $\int \vA [ v_2  \partial_x  v_1+   v_1 \partial_x v_2]= - \int \vA'  v_1 v_2.$ For the second term on \eqref{eq:I'_RHS}, applying the identity \cite{MaMu} $\Jap{\eta \partial_x\LL(f)}{f}= -\frac12  \int\eta' [ 3 (\partial_x f)^2 +V_0 f^2 ]+\frac12 \int \eta V_0'  f^2+\frac12 \int \eta''' f^2,$ we get
	\begin{equation*}
	\begin{aligned}
	&\int \vA v_1 \partial_x(\LL(u_1)+c^2 v_1)  \\
	&= \int \vA v_1 \partial_x L v_1=  -\frac12  \int\vA' [ 3 (\partial_x v_1)^2 +(1-f'(\Qc)) v_1^2 ] -\frac12 \int \vA (f'(\Qc))'  v_1^2+\frac12 \int \vA''' v_1^2 .
	\end{aligned}
	\end{equation*}
Collecting and arraying the terms, we get
\[ 
	\begin{aligned}
	\dfrac{d}{dt}\I
	=&-\frac12 \int \vA' \left[ (v_2 + c   v_1)^2+ 3 (\partial_x v_1)^2 +(1-c^2-f'(\Qc)) v_1^2 \right]
	 -(\rho'-c)  \int \vA'  v_1 v_2
	+\frac12 \int \vA''' v_1^2 
	\\&
	- \frac12 \int \vA (f'(\Qc))'  v_1^2
	+(\rho'-c) \int \vA \Qc' \left(  v_2 -c  v_1\right)
	- c' \int  \vA  \left( v_2  ~\Lambda (\Qc) +  v_1~\Lambda (c\Qc) \right)
	+\int \vA  v_1 N.
	\end{aligned}
\] 
This ends the proof.

\subsection{Second virial identity}
Proof of \eqref{eq:dJ_lema}.  We have from \eqref{eq:def_J} and \eqref{eq:z},
\[ 
\begin{aligned}
	\dfrac{d}{dt}\J
	=& \int   \psi_{A,B} (\dot{z_1} z_2+z_1 \dot{z_2}) \\
	=& \int   \psi_{A,B} (c(\partial_x z_1+c\partial_x z_2)+\LL \partial_x z_2 +M_1+M_{1,c}) z_2+ \int   \psi_{A,B} z_1 (\partial_x (z_1+cz_2)+M_2+M_{2,c}) \\
	=& \int   \psi_{A,B} (c(\partial_x z_1+c\partial_x z_2)+\LL \partial_x z_2 ) z_2
		+ \int   \psi_{A,B} z_1 \partial_x (z_1+cz_2) 
		\\& + \int   \psi_{A,B} (M_1+M_{1,c}) z_2 + \int   \psi_{A,B} z_1 (M_2+M_{2,c}) .
\end{aligned}
\] 
Rearranging the terms,
\begin{equation}\label{C3}
\begin{aligned}
	\dfrac{d}{dt}\J
	=& \int   \psi_{A,B} z_2 \LL \partial_x z_2 
		+ \int   \psi_{A,B} (z_1+c z_2) \partial_x (z_1+cz_2) 
	\\& + \int   \psi_{A,B} M_1 z_2 + \int   \psi_{A,B} M_{1,c} z_2 
	+ \int   \psi_{A,B}  M_2 z_1  	+ \int   \psi_{A,B}  M_{2,c} z_1.
\end{aligned}
\end{equation}
Appealing to the identity \cite{MaMu} $\Jap{\eta \LL(\partial_x f)}{f}= -\frac12  \int\eta' [ 3 (\partial_x f)^2 +V_0 f^2 ]-\frac12 \int \eta V_0'  f^2+\frac12 \int \eta''' f^2,$
one gets
	\[
	\begin{aligned}
	 \int   \psi_{A,B} z_2 \LL \partial_x z_2 =&  -\frac12  \int  \psi_{A,B}' [ 3 (\partial_x z_2)^2 +(1-c^2-f'(\Qc)) z_2^2 ]
		+\frac12 \int  \psi_{A,B} \partial_x(f'(\Qc))  z_2^2
		+\frac12 \int  \psi_{A,B}''' z_2^2. 
	\end{aligned}
	\]
Then, by the above  expression and integrating by parts we obtain
\[
\begin{aligned}
& \int   \psi_{A,B} z_2 \LL \partial_x z_2 
		+ \int   \psi_{A,B} (z_1+c z_2) \partial_x (z_1+cz_2) \\
&~{}= -\frac12  \int  \psi_{A,B}' [ (z_1+c z_2)^2+3 (\partial_x z_2)^2+(1-c^2 -f'(\Qc) )z_2^2]
		+\frac12 \int  \psi_{A,B} \partial_x(f'(\Qc))  z_2^2
		+\frac12 \int  \psi_{A,B}''' z_2^2	.
\end{aligned}
\]	
Gathering \eqref{C3} and the previous identity we get \eqref{eq:dJ_lema}.

\subsection{Third virial identity}
Proof of \eqref{eq:dt_N}. We have from \eqref{eq:virial_N} and \eqref{eq:z},
\[ 
\begin{aligned}
	\dfrac{d}{dt} \N
	=& \int   \psi_{A,B} (\partial_x \dot{z_1} \partial_x z_2+\partial_x z_1 \partial_x \dot{z_2}) \\
	=& \int   \psi_{A,B} \partial_x z_2\partial_x (c\partial_x (z_1+cz_2)+\LL \partial_x z_2 +M_1+M_{1,c})  
	+ \int   \psi_{A,B} \partial_x z_1 \partial_x( \partial_x (z_1+cz_2)+M_2+M_{2,c}) \\
	=& \int   \psi_{A,B} \partial_x z_2\partial_x (c\partial_x (z_1+cz_2)+\LL \partial_x z_2 )  
	+ \int   \psi_{A,B} \partial_x z_1 \partial_x( \partial_x (z_1+cz_2)) \\
	&+ \int   \psi_{A,B} \partial_x z_2\partial_x (M_1+M_{1,c})  
	+ \int   \psi_{A,B} \partial_x z_1 \partial_x( M_2+M_{2,c}) \\
	=& \int   \psi_{A,B} \partial_x z_2\partial_x \LL \partial_x z_2  
	+ \int   \psi_{A,B} (\partial_x z_1+c\partial_x z_2) \partial_x^2( z_1+cz_2) \\
	& +\int   \psi_{A,B} \partial_x z_2\partial_x (M_1+M_{1,c})  
	+ \int   \psi_{A,B} \partial_x z_1 \partial_x( M_2+M_{2,c}).
\end{aligned}
\] 
Integrating by parts and applying the identity \cite{MaMu} $\Jap{\eta \partial_x\LL(f)}{f}=  -\frac12  \int\eta' [ 3 (\partial_x f)^2 +V_0 f^2 ]+\frac12 \int \eta V_0'  f^2+\frac12 \int \eta''' f^2,$
\[ \label{eq:dtN_N}
\begin{aligned}
	\dfrac{d}{dt} \N
	=&-\frac12 \int   \psi_{A,B}' \bigg[ (\partial_x z_1+c\partial_x z_2)^2+3 (\partial_x^2 z_2)^2 \bigg]
	-\frac12 \int   \psi_{A,B}'  (1 -c^2-f'(\Qc))(\partial_x  z_2)^2
	\\&
	-\frac12 \int   \psi_{A,B} \partial_x(f'(\Qc)) (\partial_x z_2)^2
	+\frac12 \int   \psi_{A,B}''' (\partial_x  z_2)^2  
+\int   \psi_{A,B} \partial_x z_2\partial_x M_1
	 +\int   \psi_{A,B} \partial_x z_2\partial_x M_{1,c}
	\\&
	+ \int   \psi_{A,B} \partial_x z_1 \partial_x M_2
	+ \int   \psi_{A,B} \partial_x z_1 \partial_x M_{2,c}.
\end{aligned}
\] 
This proves \eqref{eq:dt_N}.

\subsection*{Data Availability} All the data obtained for this work is presented in the same manuscript.

\subsection*{Conflict of Interest} The authors declare no conflict of interest in the production and possible publication of this work.


\medskip

\begin{thebibliography}{99}
\bibitem{AS} J. C. Alexander, and R. Sachs, \emph{Linear instability of solitary waves of a Boussinesq-type equation: a computer assisted computation}. Nonlinear World 2 (1995), no. 4, 471--507.

\bibitem{BCS1} J. L. Bona, M. Chen, and J.-C. Saut, \emph{Boussinesq equations and other systems for small-amplitude long waves in nonlinear dispersive media. I: Derivation and linear theory}, J. Nonlinear. Sci. Vol. 12: pp. 283--318 (2002).

\bibitem{BCS2} J. L. Bona, M. Chen, and J.-C. Saut, \emph{Boussinesq equations and other systems for small-amplitude long waves in nonlinear dispersive media. II: The nonlinear theory}, Nonlinearity 17 (2004) 925--952.

\bibitem{Bona-Sachs} J. L. Bona, and R. L. Sachs, \emph{Global existence of smooth solutions and stability of solitary waves for a generalized Boussinesq equation.} Comm. Math. Phys. 118 (1988), no. 1, 15--29.

\bibitem{Bou1} J. Boussinesq, \emph{Th\'eorie des ondes et des remous qui se propagent le long d'un canal rectangulaire horizontal en communiquant au liquide contenu dans ce canal des vitesses sensiblement pareilles de la surface au fond}, J. Math. Pure Appl. (2) 17 (1872), 55--108.

\bibitem{Chang} S.-M. Chang, S. Gustafson, K. Nakanishi, and T.-P. Tsai, \emph{Spectra of Linearized Operators for NLS Solitary Waves}, SIAM J. Math. Anal. Vol. 39, 4 (2008). 

\bibitem{CM} R. C\^ote, and C. Mu\~noz, \emph{Multisolitons for nonlinear Klein-Gordon equations}, Forum of Mathematics, Sigma (2014), vol. 2, e15, 38 pp, \url{doi:10.1017/fms.2014.13.}

\bibitem{CMPS} R. C\^ote, C. Mu\~noz, D. Pilod, and G. Simpson, \emph{Asymptotic stability of high-dimensional Zakharov-Kuznetsov solitons}. Arch. Ration. Mech. Anal. 220 (2016), no. 2, 639--710.

\bibitem{FST} G.E. Fal'{}kovich, M.D. Spector, S.K. Turitsyn, \emph{Destruction of stationary solutions and collapse in the nonlinear string equation}, Physics Letters A,  99, Issues 6-7, (1983), 271--274.

\bibitem{Farah} L. G. Farah, \emph{Local solutions in Sobolev spaces with negative indices for the ''good'' Boussinesq equation}. Comm. Partial Differential Equations 34 (2009), no. 1-3, 52--73.

\bibitem{GSS_stability} M. Grillakis, J. Shatah and W. Strauss, \emph{Stability theory of solitary waves in the presence of symmetry I}, J. Funct. Anal., 74, 160-197 (1987).

\bibitem{pegoKP}  M. Haragus-Courcelle and R. L. Pego, \emph{Spatial wave dynamics of steady oblique wave interactions}. Physica D 145 (2000) 207--232.

\bibitem{Notes_linares} F. Linares, \emph{Notes on Boussinesq Equation}, available at \url{https://www.yumpu.com/s/a30kiVvsqSSk9YpU} 71pp. (2005).

\bibitem{KZ} V.K. Kalantarov, O.A. Ladyzhenskaya, \emph{The occurrence of collapse for quasilinear equations of parabolic and hyperbolic types}, J. Sov. Math.10, 53--70 (1978).

\bibitem{kishimoto} N. Kishimoto, \emph{Sharp local well-posedness for the "good" Boussinesq equation}, J. Differential Equations,  254 (2013), 2393--2433.

\bibitem{KM}  M. Kowalczyk, and Y. Martel, \emph{Kink dynamics under odd perturbations for (1+1)-scalar field models with one internal mode}. Math. Res. Lett. 31 (2024), no. 3, 795--832.

\bibitem{KMM2017} M. Kowalczyk, Y. Martel, and C. Mu\~noz, \emph{Kink dynamics in the $ \phi^4$ model: Asymptotic stability for odd perturbations in the energy space}, J. Amer. Math. Soc. 30 (2017), 769-798.


\bibitem{KMM} M. Kowalczyk, Y. Martel, and C. Mu\~noz, \emph{Soliton dynamics for the 1D NLKG equation with symmetry and in the absence of internal modes}. J. Eur. Math. Soc. (JEMS) 24 (2022), no. 6, 2133--2167.  

\bibitem{KMMV20} M. Kowalczyk, Y. Martel, C. Mu\~noz, and Hanne Van Den Bosch, \emph{A sufficient condition for asymptotic stability of kinks in general (1+1)-scalar field models}. Ann. PDE 7 (2021), no. 1, Paper No. 10, 98 pp.


\bibitem{Linares} F. Linares, \emph{Global existence of small solutions for a generalized Boussinesq equation}, J. Diff. Eqns. 106 (1993), 257--293.

\bibitem{Liu} Y. Liu, \emph{Instability and blow-up of solutions to a generalized Boussinesq equation}, SIAM J. Math. Anal., 26(6), 1995, 1527--1546.

\bibitem{Martel_linearKDV} Y. Martel, \emph{Linear problems related to asymptotic stability of solitons of the generalized KdV equations}, SIAM J. Math. Anal., 38(3), 759-781.

\bibitem{Martel-Merle1} Y. Martel, and F. Merle, \emph{A Liouville theorem for the critical generalized Korteweg-de Vries equation,} J. Math. Pures Appl. (9) 79 (2000), no. 4, 339--425.

\bibitem{Martel-Merle2} Y. Martel, and F. Merle, \emph{Asymptotic stability of solitons for subcritical generalized KdV equations}  Arch. Rat. Mech. Anal. 157 (2001), no. 3, 219--254.

\bibitem{MM_solitonsKdV} Y. Martel, and F. Merle, \emph{Asymptotic stability of solitons of the subcritical gKdV equations revisited}, Nonlinearity 18 (2005), no. 1, 55--80.

\bibitem{MK} H. P. McKean, \emph{Boussinesq's equation on the circle}, Comm. Pure Appl. Math., 34 (1981), pp. 599--691.

\bibitem{Mau} C. Maul\'en, \emph{Asymptotic stability manifolds for solitons in the generalized good Boussinesq equation}. J. Math. Pures Appl. (9) 177 (2023), 260--328. 

\bibitem{MaMu0} C. Maul\'en, and C. Mu\~noz, \emph{Decay in the one dimensional generalized Improved Boussinesq equation}, SN Partial Differential Equations and Applications volume 1, Article number: 1 (2020).

\bibitem{MaMu} C. Maul\'en, and C. Mu\~noz, \emph{Asymptotic stability of the fourth order $\phi^4$ kink for general perturbations in the energy space}, Ann. Inst. Henri Poincar\'e Analyse Nonlin\'eaire (2024) (\href{https://doi.org/10.4171/aihpc/112}{DOI 10.4171/AIHPC/112}).

\bibitem{MPP} C. Mu\~noz, F. Poblete and J.C. Pozo, \emph{Scattering in the energy space for Boussinesq equations}. Commun. Math. Phys. 361, 127--141 (2018).

\bibitem{MPPerratum} C. Mu\~noz, F. Poblete and J.C. Pozo, Erratum to \emph{Scattering in the energy space for Boussinesq equations}. Commun. Math. Phys. 361, 127--141 (2018) (2025).
\bibitem{Nishita}  T. Nishitani and M. Tajiri, \emph{On similarity solutions of the Boussinesq equation}, Phys. Lett. A89, 379-380 (1982).

\bibitem{SR} J. S. Russell, \emph{Report on Waves}, 14th meeting of the British Association for the Advancement of Science. Vol. 311. No. 390. (1844).

\bibitem{PW1} R. Pego, M.  Weinstein, \emph{Eigenvalues, and instabilities of solitary waves}. Philos. Trans. Roy. Soc. London Ser. A 340 (1992), no. 1656, 47--94.

\bibitem{PW2} R. Pego, M.  Weinstein, \emph{Convective Linear Stability of Solitary Waves for Boussinesq Equations}, Studies in Applied Mathematics, 99, pp. 311-375 (1997).

\bibitem{Sachs} R. L. Sachs, \emph{On the blow-up of certain solutions of the ''good'' Boussinesq equation}. Appl. Anal. 36 (1990), no. 3-4, 145--152.

\bibitem{Wei} M. I. Weinstein, \emph{Modulational stability of ground states of nonlinear Schr\"odinger equations}, SIAM J. Math. Anal. Vol. 16, No. 3, May 1985.

\bibitem{whitman} G. Whitham, \emph{Linear and nonlinear waves}, Pure and Applied Mathematics, John Wiley, 1974, 636pp. 


\end{thebibliography}
\end{document}